%% file: Splitting-Staircase.tex
\documentclass[11pt]{amsart}

\input{macros_kmx.tex}

\renewcommand{\gray}[1]{{\color{gray}{   }}}

\usepackage{verbatim}

\def\diag{\operatorname{diag}}
\def\supp{\operatorname{supp}}

\newcommand{\purple}[1]{{\color{purple}{PURPLE#1}}}

\begin{document}

\title[Breakdown of rigidity for Euclidean product structure]{Rigidity of Euclidean product structure: breakdown for low Sobolev exponents}
\author[B. Kleiner]{Bruce Kleiner}
\thanks{BK was supported by NSF grants DMS-1711556 and DMS-2005553, and a Simons Collaboration grant}
\author[S. M\"uller]{Stefan M\"uller}
\thanks{SM has been supported by the Deutsche Forschungsgemeinschaft (DFG, German Research Foundation) through
the Hausdorff Center for Mathematics (GZ EXC 59 and 2047/1, Projekt-ID 390685813) and the 
collaborative research centre  {\em The mathematics of emerging effects} (CRC 1060, Projekt-ID 211504053).}
\author[L. Sz{\'{e}}kelyhidi]{ L\'{a}szl\'{o} Sz{\'{e}}kelyhidi, Jr.}
\thanks{   LSz gratefully acknowledges the support of Grant Agreement No. 724298-DIFFINCL of the European Research Council and the support of the  Deutsche Forschungsgemeinschaft (DFG, German Research Foundation) through GZ SZ 325/2-1.}
\author[X. Xie]{Xiangdong Xie}
\thanks{XX has been supported by Simons Foundation grant \# 315130.}

\dedicatory{Dedicated to Professor Vladimir \v Sver\'ak on the occasion of his 65th birthday}
\maketitle

\setcounter{tocdepth}{2}
\tableofcontents


\section{Introduction}
The purpose of this paper is twofold. First we show that 
the results in the companion paper \cite{otherpaper} on product rigidity for maps $f: \Omega_1 \times \Omega_2 \subset \R^{n} \times 
\R^{n} \to \R^n \times \R^n$ in the Sobolev space $W^{1,p}$ are sharp
with respect to $p$. Specifically,  we show that for all $n \ge 2$
and all $p < 2$ there exist maps $f \in W^{1,p}$ such that the weak differential $\nabla f$ is invertible almost everywhere and preserves or reverses the product structure almost everywhere, but $f$ is not of the form $f(x,y) = (f_1(x), f_2(y))$ or $f(x,y) = (f_2(y), f_1(x))$, 
see Theorem \ref{t:main} below for the precise statement. 

Secondly,  we develop a general toolbox to study $W^{1,p}$ solutions
of differential inclusions $\nabla u \in K$ for unbounded sets $K$. 
A key notion is the concept that a subset $K$ of the space $\R^{d \times m}$  of $d \times m$ matrices can be reduced to another set
$K'$, see Definition \ref{d:reduced}. A closely related notion 
was introduced by M.~Sychev \cite{sychev01}    for bounded sets $K$. 
It turns out that the concept of reduction is both  simpler and more powerful for unbounded sets $K$, see the discussion after Definition \ref{d:reduced}.
As an illustration we show in the Appendix how result on
optimal $L^p$ regularity for elliptic systems with bounded measurable coefficients \cite{astala_faraco_szekelyhidi08}  as well as recent results on 
irregular solutions of the $p$-Laplace equation \cite{ColomboTione2022} can easily be obtained by this method, once one can perform at certain algebraic construction in matrix space leading to a staircase laminate (see Definition \ref{d:staircase}, 
Proposition \ref{p:staircaseconstruction} and Theorem \ref{t:SL-criterion}).

\subsection{Setting and main results}  \label{se:setting}
For sets $X_1, X_2 \subset \R^n$
we say that a mapping  $f:X_1\times X_2\ra \R^{2n}$ is split (or preserves product structure) if there exist functions $f_1 : X_1 \to  \R^n$ and $f_2 : X_2 \to \R^n$
  such that either $f(x,y) = (f_1(x),f_2(y))$ for all $(x,y) \in X_1\times X_2$
  or $ f(x,y)=(f_2(y), f_1(x)) $ for all   $(x,y) \in X_1\times X_2$. We are interested in the following question about mappings $f:\Om_1\times\Om_2\ra \R^{2n}$, where $\Om_1,\Om_2\subset \R^n$ are connected open subsets and $f$ is assumed to be in the Sobolev space $W^{1,p}_{\loc}$ for some $1\leq p<\infty$.
  
  \begin{question}  \label{qu:global_split}  
If the (approximate) differential $\nabla f(x)$ is  split and invertible for
almost every $x\in \Om$, is $f$ split?  
More generally, if the differential is ``approximately split'', must $f$ itself be ``approximately split''?
\end{question}
Our motivation for considering this question comes   from   geometric group theory, geometric mapping theory,
and the theory of nonlinear 
partial differential equations; see the introduction 
of \cite{otherpaper} and \cite{KMX1} for discussion of this context. 

We now fix connected open subsets $\Om_1,\Om_2\subset \R^n$, and let $\Om:=\Om_1\times\Om_2$. 
Note that Question~\ref{qu:global_split} is trivial for $C^1$ maps: if $f:\Om\ra \R^{2n}$ is $C^1$  and the differential $\nabla f(x)$ is 
bijective and split everywhere, then $f$ is clearly split,
since $\nabla f:\Om\ra \R^{2n\times 2n}$ is a continuous map taking values in the set of 
split and bijective linear maps, which consists of two components -- the block diagonal and the block anti-diagonal invertible matrices. 
On the other hand, if $f:\Om\ra \R^{2n}$ is Lipschitz then its differential is only measurable, so in principle oscillations between 
the two types of behavior might arise. 
In fact, for $n = 1$ it is easy to find Lipschitz maps such that $\nabla f$  is bijective and split a.e., but $f$ is not split, see \cite{otherpaper}.

In the companion paper \cite{otherpaper} we obtained
rigidity results for Sobolev maps with split or ``approximately split'' differentials. 
The first purpose of this paper 
is to show that the conditions on the $L^p$ integrability of the  weak derivative in these results are sharp. 
To do so,  we introduce a new strategy for constructing solutions to differential inclusions for low Sobolev exponents.

We first discuss  maps with split differentials.
In \cite{otherpaper} the following result was obtained.

\begin{theorem}[\cite{otherpaper}] 
\label{th:split}
 Suppose that $n \ge  2$,
 $f  \in  W^{1,2}(\Omega;\R^{2n})$  and that the weak differential $\nabla f(x)$  is split and bijective for a.e. $x \in \Omega$. Then $f$ is split.
 \end{theorem}
 
 Our first result  is that the integrability
 exponent $2$ is sharp. In fact we show the result can fail for Sobolev maps whose gradient in the Marcinkiewicz space weak-$L^2$,  even if we strengthen the condition that $\nabla f(x)$ be bijective to the 
 condition that $\det \nabla f(x) = 1$. 
 
 To state the result, we say that $\Omega \subset \R^{m}$ is a regular domain if $\Omega$ is 
 open, bounded, connected 
 and the boundary $\partial \Omega$ has  zero $m$-dimensional Lebesgue measure.
 We use the notation 
 \begin{equation} \label{e:split}
 L: = \{ X \in \R^{2n\times 2n}: \text{$X$ is split} \},
 \end{equation}
  \begin{equation}\label{e:det1}
  \Sigma:= \{ X \in \R^{2n\times 2n} :
  \det X = 1 \}
  \end{equation}

 \begin{theorem}\label{t:main}
There exists a constant $M\geq 1$ with the following property.
Let $\Omega\subset\R^{2n}$ be a regular domain, $A\in \R^{2n\times 2n}$, $\alpha\in [0,1)$ and $\delta >0$. Then there exists a continuous map $u:\overline{\Omega}\to\R^{2n}$ such that $u\in W^{1,1}(\Omega)\cap C^{\alpha}(\overline{\Omega})$ and the following properties hold:
\begin{itemize}
\item $u(x)=Ax$ on $\partial\Omega$, 
\item $\|u-Ax\|_{C^\alpha(\overline{\Omega})}<\delta$,
\item $\nabla u(x)\in L\cap\Sigma$ for almost every $x\in\Omega$,
\item for any $t>0$ 
$$
|\{x\in\Omega:\,|\nabla u(x)|>t\}|\leq M (1+|A|^2)t^{-2} |\Omega|.
$$
\end{itemize}
In particular $u\in W^{1,p}(\Omega)$ for any $p<2$ and, if $A\notin L$ then $u$ is not split. 
\end{theorem}
 
 To state the results for maps 
 with ``approximately split'' differentials
 we denote by $L_1$ the space of block-diagonal matrices
 and by $L_2$ the space of block anti-diagonal matrices, so that $L=L_1 \cup L_2$.  In \cite{otherpaper}
 we show the following rigidity result for sequences of maps
 which approximate the condition of having $\nabla f$ in $L \cap \Sigma$. We denote weak convergence by the half-arrow $\rightharpoonup$.
 
 \begin{theorem}    \label{th:approximate_solutions} 
Suppose that $n \ge  2$ and
\begin{eqnarray}  \label{eq:weak_convergence_fj}
f_j &\rightharpoonup &f    \quad \text{in $W^{1, 2n}(\Omega, \R^{2n})$},\\
 \label{eq:convergence_to_split}
\dist( \nabla f_j, L) &\to& 0 \quad \text{in $L^{1}(\Omega)$},
\end{eqnarray}
and
\begin{equation}  \label{eq:lower_bound_det}
\lim_{\delta \downarrow 0} \limsup_{j \to \infty}  | \{ x \in \Omega : \det \nabla f_j(x) < \delta \}| = 0.
\end{equation}
Then  $\nabla f \in L$ a.e. and hence  $f$ is  split. 
Moreover
\begin{equation}  \label{eq:approx_strong_to_Li}
\dist( \nabla f_j, L_i) \to 0 \quad \text{in $L^{q}(\Omega)$  \quad for $i=1$ or for  $i=2$.}
\end{equation}
and all $q < 2n$. 
\end{theorem}

Here we show that the integrability exponent $2n$ in
Theorem~\ref{th:approximate_solutions} is optimal.

\begin{theorem}  \label{th:approximate}   
There exists a constant $M\geq 1$ with the following property.
Let  $ A \in \R^{2n \times 2n} \setminus L$ with  $\rank A = 1$.
Let $\Omega\subset \R^n \times \R^n$ be a regular domain,
 and  $\alpha \in [0, 1)$. Then there exists a sequence of  maps
 $u^{(j)} : \overline \Omega \to \R^{2n}$ such that
$u^{(j)} \in  W^{1,1}(\Omega) \cap C^\alpha(\overline \Omega)$, $u^{(j)}=l_A$ on $\partial\Omega$, 
and
\begin{align}
\left|\{|\nabla u^{(j)}| > t \} \right| \le& \,  M (1 + |A|^{2n})  |\Omega| t^{-2n} \quad  \text{for $t > 0$,} \label{e:approximate_weak_L2n_intro} \\
\int_{ \{  \nabla u^{(j)} \notin L \cap \Sigma\}} (1 + |\nabla u^{(j)}|^s) \to& \,  0 \quad \text{for all $s \in [1, \infty)$,}   \label{e:approximate_bad_set_intro}\\
u^{(j)} \to& \,   l_A  \quad \text{in $C^\alpha(\overline \Omega)$,}\\
u^{(j)} \rightharpoonup& \,  l_A \quad \text{in $W^{1,p}(\Omega)$ for all $p \in [1, 2n)$}
\end{align}
and
\begin{equation}  \label{e:approximate_both_Li_intro}
\liminf_{j \to \infty} \|  \dist(\nabla u^{(j)},L_i) \|_{L^1(\Omega)} > 0  \quad \text{for $i=1, 2$.}
\end{equation}
\end{theorem}

In the statement above we write $l_A(x)=Ax$ for the linear map with gradient $A$ (see more on our notation below in Section \ref{ss:basic}). 
One may take, for example, $A = (e_1 + e_{n+1}) \otimes e_1$. We remark that the weak-type bound \eqref{e:approximate_weak_L2n_intro} and convergence in  $C^\alpha$ imply that the sequence $u^{(j)}$ is bounded in $W^{1,p}$ for any $p<2n$. Moreover, estimate \eqref{e:approximate_bad_set_intro} implies that $\dist(\nabla u^{(j)}, L \cap \Sigma) \to 0$ in $L^s(\Omega)$ for all $s \in [1, \infty)$
and $|\{ \det \nabla u^{(j)} \ne 1\}| \to 0$. 

\subsection{Strategy of the proof}  \label{se:strategy_proof}
The proofs of Theorem \ref{t:main} and Theorem \ref{th:approximate} follow a well-developed overall strategy for solving differential inclusions. The general setting is as follows. Let $K
 \subset \R^{d \times m}$ and 
 let $\Omega \subset \R^m$ be open. We want to 
 find a map $u: \Omega \to \R^d$ 
 in a suitable Sobolev space such that 
\begin{equation}\label{e:diffincl}
\nabla u(x)\in K\quad\textrm{for almost every }\,x\in\Omega.
\end{equation}
Moreover, as is typical in problems of this type, we seek solutions which in addition satisfy affine boundary conditions.

Differential inclusions of this type have a long history, both in the Lipschitz setting for compact $K$ \cite{gromov_pdr,muller_sverak96,dacorogna_marcellini99,muller99,sychev01,MullerSychev2001,Kirchheim:2002wc,kirchheim03,muller_sverak03,cellina2005,szekelyhidi2005} as well as the Sobolev setting for unbounded $K$ \cite{faraco03,Kirchheim:2002wc,conti_faraco_maggi05,conti_faraco_maggi_muller05,astala_faraco_szekelyhidi08,BorosSz2013,LiuMaly2016,oliva16,FaracoMoraCorall2018,FaracoLindberg2021,ColomboTione2022}. In most of these works the differential inclusion \eqref{e:diffincl} is solved using convex integration, an iterative construction of highly oscillatory approximate solutions which are locally almost one-dimensional. 

Our approach rests on the following notion of an 'approximate solution', which we believe is a very natural building block for
 the construction of solutions of the differential inclusion
  \eqref{e:diffincl}. Here and in the following we say that
  $\Omega \subset \R^m$ is a regular domain if it is open, bounded, connected and the boundary has vanishing $m$-dimensional measure. We say that a map  $u: \Omega \to \R^d$ from a regular domain $\Omega$ is piecewise affine if there exists  disjoint regular domains $\Omega_i \subset \Omega$ and a null set $\mathcal N$ such that $u$ is affine on $\Omega_i$ and $\Omega = \bigcup_i \Omega_i \bigcup \mathcal N$. Since each $\Omega_i$ has positive measure, the collection of sets $\Omega_i$ is at most countable. Finally, for $A \in \R^{d \times n}$ and $b \in \R^d$ we denote by $l_{A,b}$ the affine map given by 
  $l_{A,b}(x) = Ax +b $.
 
\begin{definition}\label{d:reduced}
For $K,K'\subset \R^{d\times m}$ and $1<p<\infty$ we say that 
\begin{center}
\emph{$K$ can be reduced to $K'$ in weak $L^p$}
\end{center}
provided there exists a constant $M=M(K,K',p)\geq 1$ with the following property: let $A\in K$, $b\in \R^d$, $\eps,\alpha\in(0,1)$ $s\in (1,\infty)$, and $\Omega\subset\R^m$ a regular domain. Then there exists a piecewise affine map $u \in W^{1,1}(\Omega) \cap C^\alpha(\overline \Omega)$ with $u = l_{A,b}$ on $\partial \Omega$ and such that, 
with $\Omega_{error} := \{  x \in \Omega: \nabla u(x) \notin K'\}$ we have
\begin{subequations}
\begin{align}
&\int_{\Omega_{error}}(1+|\nabla u|)^s\,dx  <\eps |\Omega|,     \label{eq:iteration_epserror}\\
&|\{x\in\Omega:\,|\nabla u(x)|>t\}|  \le M^p (1+|A|^p)|\Omega|t^{-p} \textrm{ for all }t>0.
\label{eq:iteration_weakLp} 
\end{align}
\end{subequations}

\gray{LS:\, commented out; here gray for comparison:

Moreover, for the case $p=\infty$ we say that 
\begin{center}
\emph{$K$ can be reduced to $K'$ in $L^\infty$}
\end{center}
provided there exists a constant $M=M(K,K')\geq 1$ with the above property, but \eqref{eq:iteration_epserror}-\eqref{eq:iteration_weakLp} replaced by
\begin{subequations}
\begin{align}
&\left|\Omega_{error}\right|<\eps |\Omega|,     \label{eq:iteration_epserrorinfty}\\
&\|\nabla u\|_\infty\le M\max\{1,|A|\}\,. \label{eq:iteration_Linfty} 
\end{align}
\end{subequations}
}
\end{definition}
 
\medskip

We remark that if $K$ can be reduced to $K'$ in weak $L^p$ for some $p<\infty$ then also $K$ can be reduced to $K'$ in weak $L^q$ for any $q<p$. To prove this one just needs to check that the weak $L^q$ bound follows from the weak $L^p$ bound (see Remark \ref{rmk:weakstrong}). 

There are three key properties which make our definition very useful.
\begin{enumerate}
    \item[(P1)] \emph{(Exact solutions)} If $\R^{d \times m}$ can be reduced to $K$ then for each $A \in \R^{d \times m}$ there exist $u \in W^{1,1} \cap C^\alpha$ such that $u = l_{A,b}$ on $\partial \Omega$, $\nabla u \in K$ a.e. and
    $\nabla u$ is in weak $L^p$ and hence in $L^q$ for all $q \in [1,p)$, see Theorem \ref{th:generalsolution}
    \item[(P2)] \emph{(Iteration property)} If $K$ be reduced to 
    $K'$ in weak $L^p$ and $K'$ and be reduced to $K''$ in weak $L^q$
    with $q \ne p$, then $K$ and be reduced to $K''$ in weak $L^{\min(p,q)}$, see Theorem \ref{th:iteration}
    \item[(P3)] \emph{(Sufficiency of staircase laminates)}  To show that the condition holds, it suffices to show the existence of certain probability measures -- called \emph{staircase laminates} -- with support in $K'$, barycenter in $K$ and corresponding moment bounds, see Theorem \ref{t:SL-criterion}. Roughly speaking, it  suffices to show that  there are sufficiently many rank-one connected matrices in $K'$ such that their convex combinations generate $K$. In fact, one does not even need rank-one connections in $K'$, but one can first use matrices outside $K'$ and then iteratively remove them.  
\end{enumerate}
In fact, the condition in Definition \ref{d:reduced} is not new. 
M.~Sychev \cite{sychev01} introduced a very 
similar condition for the case of compact sets (see also \cite{MullerSychev2001}). 
The main difference is that in these papers one 
requires in addition that $\nabla u \in \mathcal U$ for almost
every $x \in \Omega$, 
for some bounded open set $\mathcal U \subset \R^{d \times m}$. 
This is related to the fact that in \cite{sychev01} 
and  \cite{MullerSychev2001} one 
wants to find Lipschitz solutions and therefore one needs to ensure
that the gradients remain bounded also in $\Omega_{error}$.
In our case we are interested in solutions with unbounded gradients
and so conceptually we can take $\mathcal U = \R^{d \times m}$
which makes the condition $\nabla u \in \mathcal U$ vacuous. 
Thus, if we can show in the end that $K$ can be reduced to the whole space $\R^{d \times m}$ (in the sense of Definition \ref{d:reduced}) we do not need an extra condition.

\medskip

We now return to our concrete setting.
In view of properties (P1) and (P2) 
above the proof of Theorem \ref{t:main}
can be reduced to the following three statements
\begin{enumerate}
    \item[(Stage 1)] $\R^{2n \times 2n}$ can be reduced to  $\{ X \in \R^{2n \times 2n} : \rank X \le 1\}$ in weak $L^2$;
    \item[(Stage 2)]  $\{ X \in \R^{2n \times 2n} : \rank X \le 1\}$ can be reduced to $L$ in weak $L^q$ for all $q < \infty$
    \item[(Stage 3)] $L_i$ can be reduced to $L_i \cap \Sigma$ in weak $L^{2n}$, for $i=1,2$.
\end{enumerate}
All these results can be easily proved by first exploiting symmetry
to reduce the result to a statement about diagonal matrices and then finding a suitable staircase laminates in connection with property (P3). In fact, we will split Stage 1 further into $2n-1$ substages using property (P2) and show instead that 
$\Lambda^{(m)}$ can be reduced to $\Lambda^{(m-1)}$ in weak $L^m$ for any $m=2,\dots,2n$, 
where $\Lambda^{(m)}$ denotes the set of matrices with $\rank \le m$. 
We also show how Theorem \ref{th:approximate} follows
from Stages 2 and 3. 
Since Stage 1 is not needed here, we obtain
Sobolev integrability $p < 2n$ (or weak $L^{2n}$), while for Theorem \ref{t:main} can only get $ p < 2$ (or weak $L^2$).

\medskip

We believe that  Definition \ref{d:reduced}  in connection with the key properties (P1)--(P3) discussed above can be useful beyond the application  for the specific problem  in this paper. 
The main point is that it essentially reduces all the work to 
the (algebraic) task of finding
staircase laminates for the individual stages. All the 'analysis' 
has been put in black boxes corresponding to the three key properties. 

Let us mention a few examples.
\begin{enumerate}
\item[(a)] Reduction to singular matrices. Stage 1 in our construction,
the reduction to matrices of $\rank \le 1$,  in combination with 
Theorem \ref{th:generalsolution} shows that for every matrix $A$ there exists a map $u$ with affine boundary conditions $l_A$, $\rank \nabla u \le 1$ and $\nabla u$ in weak $L^2$. 
This result was already shown in 
\cite{FaracoMoraCorall2018,LiuMaly2016}. Our approach 
gives a simpler proof because we only need to show that  $\Lambda^{(m)}$ can be reduced to $\Lambda^{(m-1)}$ in weak $L^m$. This is easy  by exhibiting a simple staircase laminate on diagonal matrices and exploiting rotational symmetry.
In the original proof in \cite{FaracoMoraCorall2018} the laminates for different $m$ are not considered separately, but combined in a complicated and careful bookkeeping strategy.
\item[(b)] Very weak solutions to isotropic second-order elliptic equations with measurable coefficients, following \cite{astala_faraco_szekelyhidi08}.
\item[(c)] Recent examples of very weak solutions of the $p$-Laplace equation \cite{ColomboTione2022}
\end{enumerate}
Although these examples indicate the wide applicability of our framework, 
as a word of warning we point out that our approach so far depends on the fact that one ultimately can reduce to the full space $\R^{d \times m}$. For instance, this restriction means that we are not able to directly apply the general results in this paper to the case of higher integrability of $W^{1,2}$ weak solutions of second-order elliptic equations in \cite{astala_faraco_szekelyhidi08}.

\medskip

The rest of the paper is organized as follows. In Section \ref{sec:Lipschitz} we collect certain standard notions and definitions used in the theory of differential inclusions for Lipschitz mappings. In Section \ref{sec:staircase} we introduce `staircase laminates', as required for property (P3) for solving differential inclusions in the Sobolev setting, prove certain key properties of such laminates required for controlling weak $L^p$ bounds, and give several examples. The heart of the analysis is contained in Section \ref{s:general}, where we work with Definition \ref{d:reduced} and prove properties (P1)--(P3). In Section \ref{s:proofs} we then apply the general framework developed in Section \ref{s:general} to prove our main theorems. Finally, in the Appendix we show how the same framework can be used to treat examples (b) and (c) above with very little effort.

\section{Lipschitz differential inclusions}
\label{sec:Lipschitz} 

In this section we summarize the main toolbox used in the theory of differential inclusions of the type \eqref{e:diffincl}. All of the material in this section is well known and contained in various references cited above, it is included here for the convenience of the reader. 

\subsection{Basic definitions and tools}\label{ss:basic}

\begin{itemize}
\item {\bf Regular domains.} We say that $\Omega\subset\R^m$ is a regular domain if $\Omega$ is open, bounded, connected and the boundary $\partial \Omega$ has  zero $m$-dimensional Lebesgue measure. Throughout the paper we will always work under the assumption that $\Omega$ is a regular domain.
\item {\bf Push-forward measure.} Given the scaling and translation symmetries of differential inclusions of the type \eqref{e:diffincl}, a natural object is the pushforward measure of the gradient: for Lipschitz maps $u: \Omega \to \R^d$ it is defined as the measure $\nu_u\in \mathcal{P}(\R^{d\times m})$, defined by duality as
\begin{equation}\label{e:pushforward}
    \int_{\R^{d\times m}}F\,d\nu_u=\frac{1}{|\Omega|}\int_{\Omega}F(\nabla u(x))\,dx\quad\textrm{ for all }F\in C_c(\R^{d\times m}).
\end{equation}
With this definition it is easy to see that \eqref{e:diffincl} is equivalent to 
\begin{equation}\label{e:pushforwardincl}
\supp \nu_u\subset K.
\end{equation}
\item {\bf Affine maps.} For any $A\in \R^{d\times m}$ and $b\in\R^d$ we denote by $l_{A,b}$ the affine map $l_{A,b}(x)=Ax+b$.
\item {\bf Piecewise affine maps.} We call a map $u\in W^{1,1}(\Omega)$ piecewise affine if there exists an at most countable decomposition $\Omega=\bigcup_{i}\Omega_i\cup \mathcal{N}$ into pairwise disjoint regular domains $\Omega_i\subset\Omega$ and a nullset $\mathcal{N}$ such that $u$ agrees with an affine map on each $\Omega_i$. That is, for any $i$ there exists $A_i\in\R^{d\times m}$ and $b_i\in\R^d$ such that $u=l_{A_i,b_i}$ on $\Omega_i$. We will also denote by $\mathring{\Omega}_u=\bigcup_{i}\Omega_i$ (or simply $\mathring{\Omega}$ if the corresponding map $u$ is clear from the context) the open subset of $\Omega$ where $u$ is locally affine.  
Note that  the regular domains $\Omega_i$ are exactly the connected components of $\mathring \Omega$  and in particular the collection
$\{ \Omega_i\}$ is uniquely determined by  $\mathring \Omega$.
\item {\bf Gluing argument.} Let $u\in W^{1,1}(\Omega)\cap C^{\alpha}(\overline{\Omega})$ for some $\alpha\in [0,1)$, let $\{\Omega_i\}_i$ be a family of pairwise disjoint open subsets of $\Omega$, and for each $i$ let $v_i\in W^{1,1}(\Omega_i)\cap C^\alpha(\overline{\Omega_i})$ such that $v_i=u$ on $\partial\Omega_i$. Define
$$
\tilde u(x)=\begin{cases}v_i(x)&x\in\Omega_i\textrm{ for some }i,\\ u(x)& x\notin\bigcup_i\Omega_i.\end{cases}
$$
Then\footnote{Here we use that for every open set $U \subset \R^m$ the space $X = \{ u \in C(\overline U) \cap W^{1,1}(U) : u=0  \text{ on $\partial U$}$\}
is a subset of  $W^{1,1}_0(U)$ (the closure of $C_c^\infty(U)$ in $W^{1,1}(U)$) and thus the extension $\bar u : \R^m \to \R$
defined by  $\bar u = u$ in $\Omega$ and $\bar u = 0$ in $\R^m\setminus U$ belongs to $W^{1,1}(\R^m)$ and satisfies $D\bar u = 0$ a.e. in $\R^m \setminus \Omega$. To show that $X \subset W^{1,1}_0(U)$ one can argue as follows. Let $T : \R \to \R$ 
be a $C^1$  function such that $T(t) = t$ if $|t| \ge 1$, $T(t) = 0$ on $(-\tfrac12, \tfrac12)$ and $|T'| \le 4$. Set $T_k(t) = k^{-1} T(kt)$. 
Then  $T_k \circ u$ has compact support in $U$ and hence belongs to $W^{1,1}_0(U)$. Moreover we have
$ \lim_{k \to \infty} \int_{|u| \le k^{-1}} |Du| \, dx = \int_{u=0} |Du| \, dx = 0$ (for the last identity see \cite{GilbargTrudinger}[Lemma 7.7]).
Thus the chain rule implies that $T_k \circ u \to u$ in $W^{1,1}(U)$ and it follows that $u \in W^{1,1}_0(U)$. }
 $\tilde u\in W^{1,1}(\Omega)\cap C^\alpha(\overline{\Omega})$ with 
$$
\|\tilde u-u\|_{C^{\alpha}(\overline{\Omega})}\leq 2\sup_i\|v_i-u\|_{C^{\alpha}(\overline{\Omega_i})}.
$$

\item {\bf Rescaling and covering argument.}  Assume $v\in W^{1,1}(\Omega_0)\cap C^{\alpha}(\overline{\Omega_0})$ for some $\alpha\in [0,1)$ with $v=l_{A,b}$ on $\partial\Omega_0$ for some affine map $l_{A,b}$ and regular domain $\Omega_0$. Given any $\Omega\subset\R^m$ regular domain we can cover $\Omega$ by a countable family of suitably rescaled copies $\Omega_i=r_i\Omega_0+x_i$ upto measure zero. More precisely, for any $\eps>0$ there exist $r_i\in (0,1)$ with $r_i
^{1-\alpha}\leq \eps$ and $x_i\in\R^m$, $i=1,2,\dots$ so that $\{\Omega_i\}_{i}$ is a pairwise disjoint family of open subsets of $\Omega$ with $|\Omega\setminus\bigcup_i\Omega_i|=0$. Set  $v_i(x)=r_iv(\tfrac{x-x_i}{r_i})+Ax_i+(1-r_i)b$ for $x\in \Omega_i$. 
It can readily be checked that $v_i=l_{A,b}$ on $\partial\Omega_i$ and 
$\|v_i-l_{A,b}\|_{C^\alpha(\overline{\Omega_i})}\leq r_i^{1-\alpha}\|v-l_{A,b}\|_{C^\alpha(\overline{\Omega})}$.

Then the gluing argument above applies with $l_{A,b}$ in $\Omega$, $v_i$ in $\Omega_i$, and we obtain a map $u\in W^{1,1}(\Omega)\cap C^{\alpha}(\overline{\Omega})$ with $\nu_u=\nu_v$ (c.f.~\eqref{e:pushforward}). Moreover 
$$
\|u-l_{A,b}\|_{C^{\alpha}(\overline\Omega)}\leq 2\eps \|v-l_{A,b}\|_{C^{\alpha}(\overline\Omega)}.
$$
\end{itemize}

The rescaling and covering argument implies that the $C^\alpha$ estimate in Theorem \ref{t:main} follows automatically, once we have a map
 $u\in W^{1,1}(\Omega) \cap C^\alpha(\overline \Omega)$ with affine boundary conditions and so that $\nu_u$ satisfies \eqref{e:pushforwardincl}. Thus, the main issue is to be able to construct $\nu_u$ in prescribed ways. The basic tool for this is provided by the concept of \emph{laminates}.

\subsection{Laminates}\label{ss:laminates}

Let $\mathcal{M}(\R^{d\times m})$ denote the space of (signed) Radon measures on $\R^{d\times m}$ and let $\mathcal{P}(\R^{d\times m})$ be the subset of probability measures. It is well known that $\mathcal{M}(\R^{d\times m})$ can be identified with the dual space of $C_c(\R^{d\times m})$ of continuous functions with compact support, equipped with its natural locally convex topology. 

The concept of laminates of finite order and laminates with bounded support was introduced by Pedregal \cite{pedregal93}. The definition of a laminate of finite order has its root in the condition ($H_n$),
 introduced by 
Dacorogna \cite{dacorogna85} in connection with a formula for the rank-one convex envelope of a function.
The importance of laminates of finite order stems from the  following fact  (see Lemma~\ref{l:basicconstruction} below):
if $\mu\in\mathcal{P}(\R^{d\times m})$ is a laminate of finite order with center of mass $A$ and if $\Omega \subset \R^m$ is a regular domain,
then there exist piecewise affine Lipschitz maps $u: \Omega \to \R^d$ for which $\nu_u$ (see \eqref{e:pushforward}) 
approximates $\mu$ and $u(x) = Ax$ on $\partial \Omega$. We recall the main concepts.

\begin{itemize}
\item {\bf Elementary splitting.} Given probability measures $\nu,\mu\in \mathcal{P}(\R^{d \times m})$ we say that $\mu$ is obtained from $\nu$ \emph{by elementary splitting} if $\nu$ has the form $\nu=\lambda\delta_{A}+(1-\lambda)\tilde\nu$ for some $\tilde\nu\in\mathcal{P}(\R^{d \times m})$, there exist matrices $B, B'\in \R^{d \times m}$ and $\lambda'\in(0,1)$ such that
$A = \lambda' B + (1-\lambda') B'$ and $\rank (B'-B) = 1$, and moreover
$$
\mu = \lambda (\lambda'\delta_B +  (1- \lambda') \delta_{B'})+(1-\lambda)\tilde\nu.
$$
\item {\bf Laminates of finite order.} The set  $\mathcal L(\R^{d \times m})$  of \emph{laminates of finite order} is defined as the 
smallest set which is invariant under elementary splitting and contains all Dirac masses.
\item {\bf Splitting sequence.} One sees easily that $\mathcal L(\R^{d \times m})$ is equivalently characterized by the condition that it contains all probability measures which can be obtained by starting from a Dirac mass and applying elementary splitting a finite number of times - we will refer to this sequence of operations as a \emph{splitting sequence} for the laminate in question. We remark that the splitting sequence is in general not uniquely determined by the laminate.
\item {\bf Barycenter.} Each $\nu\in \mathcal L(\R^{d \times m})$ is supported on a finite set of matrices, i.e.~is of the form $\nu = \sum_{i=1}^N  \lambda_i  \delta_{A_i}$. The center of mass, or \emph{barycenter}, of $\nu$ will be denoted by $\bar \nu:=\sum_{i=1}^N \lambda_i A_i$. It is easy to see that the center of mass is invariant under splitting. 
\item {\bf Convex combinations.} Let $N\geq 2$ and $\nu_i\in \mathcal L(\R^{d \times m})$ for $i=1,\dots,N$. Further, assume that $\mu=\sum_{i=1}^N\lambda_i\delta_{A_i}$ is a laminate of finite order, where $A_i=\bar{\nu_i}$. Then $\sum_{i=1}^N\lambda_i\nu_i$ is also a laminate of finite order. In particular, if $\nu_1,\nu_2\in \mathcal L(\R^{d \times m})$ and $\rank(\bar \nu_2 - \bar \nu_1) \le 1$, then $\lambda \nu_1 + (1- \lambda) \nu_2 \in \mathcal L(\R^{d \times m})$ for any $\lambda\in (0,1)$. 
\item {\bf Linear transformations.} Let $T:\R^{d\times m}\to\R^{d\times m}$ be a linear map preserving rank-one matrices, i.e. with the property that $\rank(A)=1$ if and only if $\rank(T(A))=1$. Then, if $\nu=\sum_{i=1}^N\lambda_i\delta_{A_i}$ is a laminate of finite order, so is $T_*\nu:=\sum_{i=1}^N\lambda_i\delta_{T(A_i)}$.

\end{itemize}

\subsection{Laminates and differential inclusions}\label{ss:di}

Now we come to the cornerstone of the theory. Having introduced a sufficiently rich set of probability measures 
$\mathcal L(\R^{d \times m})$ we now show that these measures can be well approximated by gradient distributions of Lipschitz maps with affine boundary conditions. There are various versions of the statements below in the literature, e.g. \cite{Cellina93,dacorogna_marcellini99, sychev01,Kirchheim:2002wc,DMP2008,Pompe2010}, but for our purposes the following variant of the basic building block (see \cite{muller_sverak03,kirchheim03}) will prove most useful.

\begin{lemma}\label{l:roof}
Let $A,A_1,A_2\in\R^{d\times m}$ be matrices such that 
\begin{equation*}
\rank(A_1-A_2)=1,\textrm{ and }A=\lambda_1 A_1+\lambda_2 A_2
\end{equation*} 
for some $\lambda_1,\lambda_2> 0$, $\lambda_1+\lambda_2=1$. For any $b\in\R^d$, any $\eps>0$ and any regular domain $\Omega\subset\R^m$ there exists a piecewise affine Lipschitz map $u:\Omega\to\R^d$ such that 
$u=l_{A,b}$ on $\partial\Omega$ and $\|\nabla u\|_{L^{\infty}(\Omega)}\leq \max(|A_1|,|A_2|)$ and for $i=1,2$
\begin{equation*}
(1-\eps)\lambda_i   
|\Omega| \leq \left|\left\{x\in\Omega:\,\nabla u(x)=A_i\right\}\right|\leq (1+\eps)\lambda_i|\Omega|.
\end{equation*}
\end{lemma}

\begin{proof} 
Since $\rank(A_1-A_2)=1$, there exist nonzero vectors $\xi\in\R^m$, $\eta\in\R^d$ such that $A_2-A_1=\eta\otimes\xi$. Note that we can write
$A_1=A-\lambda_2\eta\otimes\xi$ and $A_2=A+\lambda_1\eta\otimes\xi$.

Let $r\in (0,1)$ to be fixed later and let $\xi^{(1)},\dots,\xi^{(J)}\in \R^m$ be further nonzero vectors with $|\xi^{(j)}|<r$ so that $0\in \textrm{int conv}\{\xi,-\xi,\xi^{(1)},\dots,\xi^{(J)}\}$. Then the set 
$$
\Omega_0:=\left\{x\in\R^m:\,x\cdot \xi^{(j)}>-1\textrm{ for all }j=1\dots J\textrm{ and }|x\cdot\xi|<1\right\}
$$
is a convex open and bounded set containing $0$.

Finally, let $h:\R\to [0, \infty)$ 
 be a $1$-periodic Lipschitz function with $h(0)=0$ and $h'(t)\in \{-\lambda_2,\lambda_1\}$ for a.e.~$t\in\R$, a saw-tooth function with 
$$
|\{t\in (0,1):h'(t)=-\lambda_2\}|=\lambda_1\textrm{ and }|\{t\in (0,1):h'(t)=\lambda_1\}|=\lambda_2.
$$
Observe that $h(N x \cdot \xi) = 0$ whenever $x \cdot \xi \in \{-1, 1\}$.
Then, for any $N\in\N$ the function $f_N:\Omega_0\to\R$ defined by
$$
f_N(x)=\min\left\{\min_{1\leq j\leq J}(1+x\cdot\xi^{(j)}),\,\frac{1}{N}h(Nx\cdot\xi)\right\}
$$
is a piecewise affine Lipschitz function with $f_N=0$ on $\partial \Omega_0$, and 
\begin{equation}\label{e:construction1}
\nabla f_N(x)\in\left\{-\lambda_1\xi, \lambda_2\xi, \xi^{(1)},\dots,\xi^{(J)}\right\}\quad\textrm{ a.e. }x\in\Omega_0.
\end{equation}
Moreover, by choosing $N$ sufficiently large we can achieve that 
$$
\frac{1}{N}h(Nx\cdot\xi)<\min_{1\leq j\leq J}(1+x\cdot\xi^{(j)})\quad\textrm{ on }(1-r)\Omega_0,
$$
from which we deduce
\begin{equation*}
\begin{split}
\left|\left\{x\in \Omega_0:\,\nabla f_N(x)=-\lambda_2\xi \right\}\right|&\geq (1-r)^m|\Omega_0|\\
\left|\left\{x\in \Omega_0:\,\nabla f_N(x)=\lambda_1\xi\right\}\right|&\geq (1-r)^m|\Omega_0|.
\end{split}
\end{equation*}
Then, by choosing $r>0$ sufficiently small, the map $u:\Omega_0\to\R^d$ defined by 
$$
u(x)=b+Ax+\eta f_N(x)
$$
satisfies all the claimed properties in the lemma for the special domain $\Omega_0$. Here we use the fact that $\max\{|A_1|,|A_2|\}>|A|$.
For a general regular domain $\Omega$ we apply a rescaling and covering argument from Section \ref{ss:basic}.
\end{proof}

An obvious iteration of Lemma \ref{l:roof} along the splitting sequence of any laminate $\nu\in \L(\R^{d\times m})$ (c.f.~Section \ref{ss:laminates})
leads to the following lemma, which makes laminates so useful for inclusion problems of the type \eqref{e:diffincl}.

\begin{lemma}\label{l:basicconstruction}
Let $\nu\in\L(\R^{d\times m})$ be a laminate of finite order with center of mass $A$. Write $\nu=\sum_{j=1}^J\lambda_j\delta_{A_j}$ with $\lambda_j>0$ and $A_j \ne A_k$ for $j \ne k$.
For any $b\in\R^d$, any $\eps>0$ and any regular domain $\Omega\subset\R^m$ there exists a piecewise affine Lipschitz map $u:\Omega\to\R^d$ with $u=l_{A,b}$ on $\partial\Omega$, $\|\nabla u\|_{L^{\infty}(\Omega)}\leq \max_i|A_i|$ and such that 
\begin{equation}\label{e:basicconstruction}
(1-\eps)\lambda_j|\Omega|\leq \left|\{x\in\Omega:\,\nabla u(x)=A_j\}\right|\leq (1+\eps)\lambda_j|\Omega|
\end{equation}
for each $j=1,\dots,J$. 
\end{lemma}
We remark that, since $\sum_{j=1}^J\lambda_j=1$, estimate \eqref{e:basicconstruction} also implies
\begin{equation}\label{e:basicconstruction-2}
\left|\{x\in\Omega:\,\nabla u(x)\notin \supp\nu\}\right|\leq \eps|\Omega|.
\end{equation}

\section{Staircase laminates}
\label{sec:staircase}

Lemma \ref{l:basicconstruction},  together with the 
basic splitting mechanism for laminates described in 
Section \ref{ss:laminates},  gives a flexible 
method to construct Lipschitz solutions of differential inclusions of the type \eqref{e:diffincl} involving compact sets $K$.

For differential inclusions involving unbounded sets, 
as in Theorem \ref{t:main}, one needs to extend this method to unbounded laminates (more precisely, sequences of laminates with increasing support) as well as the corresponding construction from Lipschitz to Sobolev maps. This was first recognized by D.~Faraco in \cite{faraco03} in the context of isotropic elliptic equations in the plane. 

To set the stage, we briefly describe the problem studied in
\cite{faraco03} and subsequently in \cite{astala_faraco_szekelyhidi08} and how it leads to  a staircase laminate.
The question  is to find the  optimal higher integrability of weak solutions of isotropic elliptic equations of the form 
\begin{equation}\label{e:elliptic-eq}
\textrm{div}(a(x)\nabla u)=0\textrm{ in }\Omega\subset\R^2,
\end{equation}
with $a(x)\in \{\mathcal{K},\mathcal{K}^{-1}\}$ for a.e. $x$. The associated differential inclusion in this case is of the form
\begin{equation}\label{e:k-elliptic}
    \nabla u(x)\in E_{\mathcal K}\cup E_{1/\mathcal{K}}\textrm{ a.e. }x
\end{equation}
where 
\begin{equation*}
E_{\rho}=\left\{\begin{pmatrix} \lambda&0\\0&\rho\lambda\end{pmatrix}R:\,\lambda\geq 0,\,R\in SO(2)\right\}\,.
\end{equation*}
The constant $\mathcal{K} >1$ is related to the ellipticity constant of the associated PDE. 

For the convenience of the reader we briefly recall the argument
to pass from the PDE  \eqref{e:elliptic-eq} to the differential
inclusion \eqref{e:k-elliptic}.
Indeed, assuming that $\Omega\subset\R^2$ is a simply connected regular domain, we see that for any $\sigma\in L^1(\Omega)$ the condition $\textrm{div}\sigma=0$ is equivalent to the existence of $w\in W^{1,1}(\Omega)$ with $\sigma^\perp=(-\sigma_2,\sigma_1)=\nabla w$. Then, writing $u=(v,w)$ we immediately deduce that the differential inclusion \eqref{e:k-elliptic} is equivalent to the equation \eqref{e:elliptic-eq}.

The question is about the optimal higher integrability of solutions $u\in W^{1,2}_{loc}$ of \eqref{e:k-elliptic}. Since such solutions are automatically $\mathcal{K}$-quasiregular, it follows \cite{astala1994} that $\nabla u\in L^{p}_{loc}$ for any $p<p_{\mathcal{K}}:=\frac{2\mathcal{K}}{\mathcal{K}-1}$, with radial mappings showing optimality in the class of $\mathcal K$-quasiregular mappings. Whether this bound is  also optimal in the isotropic case was the remaining issue.  

In \cite{faraco03} Faraco constructed, for any $\mathcal{K}>1$ a sequence of laminates of finite order $\nu^N$, $N=1,2,\dots$, with the properties that
$\nu^N=\tilde\nu^N+\beta_N\delta_{A_N}$, where $\supp\tilde\nu^N\subset E_\mathcal{K}\cup E_{-\mathcal{K}}$, $A_N=\begin{pmatrix}N&0\\0&N+1\end{pmatrix}$, and $\nu^{N+1}$ is obtained inductively from $\nu^N$ by a sequence of two elementary splittings, staring with $\delta_{A_N}$ (these splittings form the steps of the ``staircase''). The key computation in \cite{faraco03}, see also Proposition 3.10 in \cite{astala_faraco_szekelyhidi08}, is then the estimate
$$
\frac{1}{C}N^{-p_\mathcal{K}}\leq \nu^N(\{A_N\})\leq CN^{-p_\mathcal{K}}
$$
for some constant $C$ which is independent of  $N$. This estimate then implies the weak $L^{p_k}$ bound
for the limiting measure $\nu^\infty$:
$$
\frac{1}{C}t^{-p_\mathcal{K}}\leq \nu^\infty(\{|X|>t\})\leq Ct^{-p_\mathcal{K}}
$$
for some (possibly larger) constant $C>1$ and all $t>1$.

\subsection{Staircase laminates and differential inclusions}\label{s:staircaselaminates}

Since the original application of Faraco, staircase laminates have been applied in several situations where one can expect endpoint weak $L^p$ bounds \cite{Kirchheim:2002wc,conti_faraco_maggi05,conti_faraco_maggi_muller05,astala_faraco_szekelyhidi08,BorosSz2013,FaracoMoraCorall2018,FaracoLindberg2021,ColomboTione2022}. Although staircase laminates are  a very versatile tool, we were unable to find a general treatment of staircase laminates analogous to the case of bounded laminates described in Section \ref{ss:laminates} and in particular a corresponding generalization of  Lemma \ref{l:basicconstruction}. 

In the following we offer such a general treatment, which will be able to not just provide the solution to our setting in Theorem \ref{t:main}, but also applies to a number of further examples that have appeared in the literature. 

\begin{proposition}[Staircase laminates]\label{p:staircaseconstruction} 
Let $K\subset \R^{d \times m}$ and $A\notin K$. Suppose that there exists a sequence of matrices $A_n\in \R^{d\times m}\setminus K$, $n=0,1,2,\dots$ with $A_0=A$, a sequence of probability measures $\mu_n\in \mathcal{P}(K)$ supported in $K$ as well as scalars $\gamma_n\in (0,1)$ such that 
\begin{enumerate}
    \item for any $n\in\N$ the probability measures 
    $$\omega_n = (1-\gamma_n) \mu_n + \gamma_n \delta_{A_n}$$ 
    are laminates of finite order with barycenter $\overline{\omega_n}=A_{n-1}$; 
    \item the sequence $|A_n|$ is monotone increasing with $\lim_{n \to \infty} |A_n| = \infty$;
    \item $\lim_{n \to \infty} \beta_n=0$, where $\beta_n:=\prod_{k=1}^n\gamma_k$, $\beta_0=1$. 
\end{enumerate}
Define the probability measures $\nu^N$, $N=1,2,\dots$ by iteratively replacing $\delta_{A_{n-1}}$ by $\omega_n$ for $1 \le n \le N$, i.e., by 
\begin{align*} & \nu^N =  \sum_{n=1}^N 
    \beta_{n-1} (1-\gamma_n) \mu_n  
    + \beta_N \delta_{A_N}.
\end{align*} 
Then $\nu^N$ is a laminate of finite order with $\supp\nu^N\subset K\cup\{A_N\}$ and barycenter $\overline{\nu^N}=A$. Moreover, for any Borel set $E\subset \R^{d\times m}$ the limit
\begin{align*}
    \nu^\infty(E)=\lim_{N\to\infty}\nu^N(E)
\end{align*}
exists and defines a probability measure $\nu^\infty$ with $\supp\nu^\infty\subset K$ and $\overline{\nu^\infty}=A$.
\end{proposition}

\begin{proof}
The subprobability measures
\begin{equation}\label{e:nuNtilde}
\tilde \nu^N :=  \sum_{n=1}^N 
\beta_{n-1}(1-\gamma_n) \mu_n = \nu^N - \beta_N \delta_{A_N}
\end{equation}
are  supported in $K$ and increasing in $N$. 
Thus for each Borel set $E \subset \R^{d \times m}$ the 
limit $\nu^\infty(E):= \lim_{N \to \infty} \nu^N(E)$ exists. 
Since $\lim_{N \to \infty} \beta_N = 0$,  we see that $\nu^\infty$ 
is a probability measure supported in $K$. In fact,
$\nu^\infty$ is a 
countable sum of Dirac masses.
\end{proof}

Staircase laminates in general involve a particular inductive construction which can be used to ``push mass to infinity'', as denoted by condition (2) above. The rate at which mass can be pushed to infinity (equivalently, the rate of convergence $\beta_n\to 0$) determines the exponent $p$ at which weak $L^p$ bounds will hold.

\begin{definition}[Staircase laminates]\label{d:staircase}
A probability measure $\nu^\infty$ defined by the procedure outlined in Proposition \ref{p:staircaseconstruction} is called a \emph{staircase laminate}.
\end{definition}

\subsection{Properties of staircase laminates}

In this section we collect some basic properties of staircase laminates. First of all, the very definition of staircase laminates implies that the invariance property of finite order laminates (see Section \ref{ss:laminates}) can be directly transferred.

\begin{lemma}\label{l:invariance}
Let $T:\R^{d\times m}\to\R^{d\times m}$ be a linear map preserving rank-one matrices, i.e. with the property that $\rank(A)=1$ if and only if $\rank(T(A))=1$. Then, if $\nu^{\infty}=\sum_{i=1}^\infty\lambda_i\delta_{A_i}$ is a staircase laminate with barycenter $A$, so is $T_*\nu^{\infty}:=\sum_{i=1}^\infty\lambda_i\delta_{T(A_i)}$ with barycenter $T(A)$.
\end{lemma}

\begin{proof}
The proof follows directly from the invariance property of finite order laminates, applied to the sequence $\nu^N$ in Proposition \ref{p:staircaseconstruction}.
\end{proof}

Next we address criteria leading to bounds of the type \eqref{e:SL-bound}. 

\begin{lemma}[Weak $L^p$ bounds for staircase laminates]\label{l:weakLpstaircase}
Suppose $\nu^\infty$ is a staircase laminate with barycenter $A$, defined by sequences $\{A_n,\mu_n,\gamma_n\}_{n\in \N}$ as in Proposition \ref{p:staircaseconstruction}, and suppose 
\begin{equation}\label{e:staircaseLp0}
  |A_n|\leq |A_{n+1}|\leq c|A_n|\quad\textrm{ for all $n$}
\end{equation}
for some $c>1$.
\begin{itemize}
    \item ({\it Upper bound}) Assume that for some $1\leq p<\infty$ there exists $c_0, M_0\geq 1$ such that, for all $n\in\N$:
\begin{subequations}
\begin{align}
   \supp\mu_n&\subset \{X\in \R^{d\times m}:\,|X|\leq c_0|A_n|\},\label{e:staircaseLpupper1}\\
   \beta_n|A_n|^p&\leq M_0.\label{e:staircaseLpupper2}
\end{align}
\end{subequations} 
Then
\begin{equation}\label{e:staircaseLpupper}
\nu^{\infty}(\{X:|X|>t\})\leq M_0c^pc_0^{p}t^{-p}\quad\textrm{ for all }t>0.
\end{equation}
\item ({\it Lower bound}) Assume that for some $1\leq p<\infty$ there exists $0<c_1,M_1$ such that, for all $n\in\N$:
\begin{subequations}
\begin{align}
   \mu_n(\{X\in \R^{d\times m}:\,|X|\geq c_1|A_n|\})&\geq c_1,\label{e:staircaseLplower1}\\
   \beta_n|A_n|^p&\geq M_1.\label{e:staircaseLplower2}
\end{align}
\end{subequations} 
Then 
\begin{equation}\label{e:staircaseLplower}
\nu^{\infty}(\{X:|X|>t\})\geq M_1c^{-p}c_1^{1+p}t^{-p}\quad\textrm{ for all }t>c_1|A|.
\end{equation}
\end{itemize}
\end{lemma}
\begin{proof}\hfill 

\noindent{\bf Upper bound. }For any $n=1,2,\dots$ let $t_n=c_0|A_n|$. Using \eqref{e:staircaseLpupper1} we observe that $\mu_k(\{X:\,|X|>t_n\}=0$ for all $k\leq n$, and hence, for any $N\geq n$
\begin{equation*}
    \nu^N(\{X:\,|X|>t_n\})\leq \sum_{k=n+1}^N\beta_{k-1}(1-\gamma_k)+\beta_N.
\end{equation*}
Noting that $\beta_{k+1}=\gamma_{k+1}\beta_k$, we see that the sum is telescoping and we deduce, using in addition \eqref{e:staircaseLpupper2},
\begin{equation*}
    \nu^N(\{X:\,|X|>t_n\})\leq \beta_{n}\leq M_0|A_n|^{-p}=M_0c_0^p t_n^{-p}.
\end{equation*}
Letting $N\to\infty$ we obtain \eqref{e:staircaseLpupper} for $t_n$, $n=0,1,\dots$. More generally, for any $t\geq t_0$ choose $n\in \N$ so that $t_n\leq t<t_{n+1}$. Using \eqref{e:staircaseLp0} we then estimate
\begin{align*}
    \nu(\{X:\,|X|>t\})\leq M_0c_0t_n^{-p}\leq M_0c^pc_0^{p}t^{-p}.
\end{align*}
Finally, if $t<t_0=c_0|A_0|$, then, using again \eqref{e:staircaseLpupper2},
$$
M_0c^pc_0^{p}t^{-p}>M_0c^p|A_0|^{-p}\geq c^p\geq 1
$$
so that \eqref{e:staircaseLpupper} is trivially satisfied.

\smallskip

\noindent{\bf Lower bound.} Arguing analogously to above, for any $n=1,2,\dots$ we define $t_n=c_1|A_n|$. Using \eqref{e:staircaseLplower1} we observe that 
for any $k\geq n$ we have
\begin{equation*}
    \mu_k(\{X:\,|X|>t_n\})\geq \mu_k(\{X:\,|X|>t_k\})\geq c_1.
\end{equation*}
Then, for any $N\geq n$ we have
\begin{equation*}
    \nu^N(\{X:\,|X|>t_n\})\geq c_1\sum_{k=n}^N\beta_{k-1}(1-\gamma_k)+\beta_N.
\end{equation*}
As before, the sum is telescoping and we deduce, using in addition \eqref{e:staircaseLplower2},
\begin{align*}
    \nu^N(\{X:\,|X|>t_n\})&\geq c_1\beta_{n-1}+\beta_N(1-c_1)\geq c_1\beta_{n-1}\\
    &\geq M_1c_1|A_{n-1}|^{-p}= M_1c_1^{1+p}t_{n-1}^{-p}\geq M_1c_1^{1+p}t_{n}^{-p}.
\end{align*}
Letting $N\to\infty$ we obtain \eqref{e:staircaseLplower} for $t=t_n$, $n=0,1,\dots$. More generally, for any $t>t_0=c_1|A|$ choose $n\in \N$ so that $t_n\leq t<t_{n+1}$. Then using \eqref{e:staircaseLp0} we estimate
\begin{align*}
    \nu(\{X:\,|X|>t\})\geq M_1c_1^{1+p}t_{n+1}^{-p}\geq M_1c^{-p}c_1^{1+p}t^{-p}
\end{align*}
as required.

\end{proof}


\subsection{Examples of staircase laminates}

In this section we give several examples of staircase laminates. The first two examples are related to Propositions \ref{p:stage1} and \ref{p:stage3} which correspond to Stage 1 and Stage 3 in the construction  of non-split maps which have  split and invertible differentials almost everywhere. Subsequent examples are provided to illustrate the effectiveness of our general approach to differential inclusions.

\subsubsection{Example 1}
We use the notation
\begin{eqnarray*}
\mathcal{D} &=&\{\textrm{ diagonal $d\times d$ matrices }\},\\\Sigma&=&\{\textrm{ $d\times d$ matrices with determinant $=1$}\}
\end{eqnarray*}

\begin{lemma}[Laminates supported in $\det=1$]\label{l:staircase1} 
Let $A$ be a diagonal $d\times d$ matrix with entries $A=\diag(a_1,\dots,a_d)$ satisfying $|a_i|>1$ for all $i=1,\dots,d$. There exists a laminate $\omega\in\L(\R^{d\times d})$ with barycenter $\bar{\omega}=A$ such that $\omega$ can be written as $\omega=(1-\gamma)\mu+\gamma\delta_{2A}$ for some probability measure $\mu$ with support 
\begin{equation}\label{e:staircase1s}
\supp\mu\subset \mathcal{D}\cap\Sigma\cap \{X:\,|X|\leq 2|A|\}
\end{equation}
and
\begin{equation}\label{e:staircase1w}
\gamma=\frac{\det A-1}{2^d\det A-1}.
\end{equation}
Moreover there exists $c=c(A,d)>0$ such that
\begin{equation}\label{e:staircase1b}
    \mu(\{X:\,|X|\geq \frac{1}{\sqrt{d}}|A|\})\geq c. 
\end{equation}
\end{lemma}

\begin{proof}
We construct $\omega$ by exhibiting its splitting sequence:
\begin{equation}\label{e:splitting}
\begin{split}
\delta_A&\mapsto \alpha_1\delta_{B_1}+\alpha_1'\delta_{C_1}\\
&\mapsto \alpha_1\delta_{B_1}+\alpha_2\delta_{B_2}+\alpha_2'\delta_{C_2}\\
&\vdots\\
&\mapsto \sum_{j=1}^d\alpha_j\delta_{B_j}+\alpha_d'\delta_{C_d},
\end{split}
\end{equation}
in such a way that $B_j\in  \mathcal{D}\cap\Sigma$, $\alpha_d'=\gamma$, $C_d=2A$, and furthermore 
$$
\rank(B_j-C_j)=1,\, \alpha_j,\alpha_j'>0,\,\sum_{i=1}^j\alpha_i+\alpha_j'=1
$$
for all $j=1,\dots,d$. 

To this end we start by setting $D=\det A$, $B_1=\diag(\frac{a_1}{D},a_2,\dots,a_d)$, $C_1=\diag(2a_1,a_2,\dots,a_d)$. Then $\det B_1=1$ and $\rank(B_1-C_1)=1$. Moreover, 
$a_1=\alpha_1\frac{a_1}{D}+(1-\alpha_1)(2a_1)$ with $\alpha_1=\frac{D}{2D-1}$. Since we assume that $|D|>1$, it can be checked directly that $1/3<\alpha_1<1$, and then we obtain 
$A=\alpha_1B_1+\alpha_1'C_1$ with $\alpha_1'=1-\alpha_1$.

The definition of $B_j,C_j,\alpha_j$ for $j\geq 2$ proceeds analogously. In general we set $B_j=\diag(b_{j,1},\dots,b_{j,d})$ and $C_j=\diag(c_{j,1},\dots,c_{j,d})$ with
\begin{equation}\label{e:def-bij}
b_{j,i}=\begin{cases}2a_i&i<j,\\ \frac{a_i}{2^{j-1}D}& i=j,\\ a_i&i>j,\end{cases}\quad 
c_{j,i}=\begin{cases}2a_i&i\leq j,\\  a_i&i>j.\end{cases}
\end{equation}
Then 
$$
C_{j-1}=\tilde\alpha_jB_j+(1-\tilde\alpha_j)C_j,
$$
where $\tilde\alpha_j=\frac{2^{j-1}D}{2^jD-1}$. 
We can again check directly that $\tilde\alpha_j\in (1/3,1)$. Setting inductively 
\begin{equation}\label{e:ex1-alphaj}
\alpha_j=\tilde\alpha_j\alpha_{j-1}'\textrm{ and } \alpha_j'=(1-\tilde\alpha_j)\alpha_{j-1}'
\end{equation}
we obtain a laminate of finite order as in \eqref{e:splitting}. Furthermore
$$
\alpha_d'=\prod_{j=1}^d(1-\tilde\alpha_j)=\prod_{j=1}^d\frac{2^{j-1}D-1}{2^jD-1}=\frac{D-1}{2^dD-1}=\gamma,
$$
as claimed in \eqref{e:staircase1w}. 

Concerning the lower bound \eqref{e:staircase1b}, we can compute, using \eqref{e:ex1-alphaj},
$$
\mu_n(\{B_j\})=\alpha_j=\frac{2^{j-1}D(D-1)}{(2^jD-1)(2^{j-1}D-1)}.
$$
Using that $|D|=|\det A|>1$ we can directly verify 
\begin{equation*}
    \alpha_j\geq \begin{cases}\frac{D-1}{2^dD}&\textrm{ if }D>1,\\
    \frac{1}{3\cdot 2^d}&\textrm{ if }D<-1.\end{cases}
\end{equation*}
Thus, in either case there exists $c=c(A,d)>0$ such that 
\begin{equation}\label{e:ex1-lowerbound1}
\mu(\{B_j\})\geq c\textrm{ for all }j.
\end{equation}
Let $j_*\in\{1,\dots,d\}$ such that $|a_{j_*}|=\max_i|a_i|$. If $j_*<d$ then, by construction, $|B_{d}|\geq 2|a_{j_*}|$, whereas if $j_*=d$, then $|B_{1}|\geq |a_{j_*}|$. In either case 
\begin{equation}\label{e:ex1-lowerbound2}
\max_j|B_{j}|\geq \frac{1}{\sqrt{d}}|A|.
\end{equation} 
From \eqref{e:ex1-lowerbound1}-\eqref{e:ex1-lowerbound2} follows
\begin{equation*}
    \mu(\{X:\,|X|\geq \frac{1}{\sqrt{d}}|A|\})\geq \min_{j}\alpha_j\geq c
\end{equation*}
as claimed.
\gray{LS:\, commented out; here gray for comparison:

To this end observe that any laminate of finite order $\nu$ satisfies the ``commutativity relation'' (c.f.\cite{tartar79,diperna1985}) 
\begin{equation}\label{e:commutativity}
	\int\det X\,d\nu(X)=\det\left(\int X\,d\nu(X)\right)=\det(\overline{\nu}).
\end{equation}
 Applying \eqref{e:commutativity} to \eqref{e:splitting} we obtain
$$
\sum_{j=1}^d\alpha_j\det B_j+\alpha_d'\det C_d=\det A,
$$
thus $1-\alpha_d'+\alpha_d'2^d\det A=\det A$, where we have used that $\det B_j = 1$, $C_d=2A$ and $\sum_{j=1}^d\alpha_j=1-\alpha_d'$. Rearranging, we deduce $\alpha_d'=\gamma$ as claimed. 
}
\end{proof}

\bigskip

\begin{example}\label{ex:staircase1}
Lemma \ref{l:staircase1} can be directly applied in Proposition \ref{p:staircaseconstruction} as follows. Let $A\in \R^{d\times d}$ be a diagonal matrix with entries $A=\diag(a_1,\dots,a_d)$ satisfying $|a_i|\geq 2$ for all $i$, and let $A_n=2^nA$ for $n=0,1,2,\dots$. Applying Lemma \ref{l:staircase1} to each $A_n$ leads to finite order laminates $\omega_n=(1-\gamma_n)\mu_n+\gamma_n\delta_{A_n}$ with barycenter $\overline{\omega_n}=A_{n-1}$ with
\begin{equation*}
    \gamma_n=\frac{\det A_{n-1}-1}{2^d\det A_{n-1}-1}=\frac{2^{(n-1)d}\det A-1}{2^{nd}\det A-1}.
\end{equation*}
Then 
\begin{equation*}
    \beta_n=\prod_{k=1}^n\gamma_k=\frac{\det A-1}{2^{nd}\det A-1}.
\end{equation*}
Using $|\det A|\geq 2$ we obtain
\begin{equation}\label{e:staircase1-betan}
    2^{-nd-1}\leq \beta_n\leq 2^{-nd+1}.
\end{equation}
In particular the conditions of Proposition \ref{p:staircaseconstruction} are satisfied and we obtain the sequence of laminates $\nu^N$ with $\supp\nu^N\subset (\mathcal{D}\cap\Sigma)\cup \{A_N\}$ and barycenter $A$. 

Furthermore, \eqref{e:staircase1s} implies \eqref{e:staircaseLpupper1}, \eqref{e:staircase1b} implies \eqref{e:staircaseLplower1} with $c_0=2$, whereas \eqref{e:staircase1-betan} implies \eqref{e:staircaseLpupper2} and \eqref{e:staircaseLplower2} with $c_1=1/2$. Thus, the conditions in Lemma \ref{l:weakLpstaircase} are satisfied and we deduce the weak $L^d$ bounds
\begin{equation}\label{e:staircase1-bound}
    2^{-2-d}|A|^dt^{-d}\leq \nu^{\infty}(\{X:|X|>t\})\leq 2^{1+d}|A|^dt^{-d}
\end{equation}
for all $t>|A|$.
\end{example}

\bigskip

\subsubsection{Example 2}
In the second example we define
$$
\Lambda^{(m)}=\{A\in \R^{d\times d}:\,\rank(A)\leq m\},
$$
for $m=1,\dots,d$, the set of matrices of rank at most $m$. Thus in particular $\Lambda^{(1)}$ is the rank-one cone of $d\times d$ matrices, $\Lambda^{(d)}=\R^{d\times d}$ and in general $\Lambda^{(m)}\subset\Lambda^{(m+1)}$. These sets arise in the construction of Sobolev homeomorphisms with low rank gradients \cite{LiuMaly2016,FaracoMoraCorall2018}.

\begin{lemma}[Laminates supported in $\Lambda^{(m-1)}$]\label{l:staircase2}
Let $A\in \Lambda^{(m)}\cap \mathcal D$ for some $m\geq 2$. There exists a laminate $\omega\in\L(\R^{d\times d})$ with barycenter $\bar{\omega}=A$ such that $\omega$ can be written as $\omega=(1-\gamma)\mu+\gamma\delta_{2A}$ for some probability measure $\omega$ with support 
\begin{equation}\label{e:staircase2s}
\supp\mu\subset \Lambda^{(m-1)}\cap \mathcal{D}\cap\{X:\,|X|\leq 2|A|\}
\end{equation}
and
\begin{equation}\label{e:staircase2w}
\gamma=2^{-m}.
\end{equation}
Moreover, 
\begin{equation}\label{e:staircase2b}
\mu(\{X:\,|X|\geq \frac{1}{\sqrt{m}}|A|\})\geq \frac{1}{2^m}.
\end{equation}
	
\end{lemma}

\begin{proof}
We proceed analogously to the proof of Lemma \ref{l:staircase1}, by defining a splitting sequence as in \eqref{e:splitting}. After permuting the entries of $A$ if necessary, we may assume without loss of generality that $A=\diag(a_1,\dots,a_m,0,\dots,0)$ for some $m\geq 2$. Then the splitting sequence can be written even more explicitly as
\begin{align*}
\delta_{A}&\mapsto \frac12\delta_{\diag(0,a_2,\dots,a_m,0,\dots)}+\frac12\delta_{\diag(2a_1,a_2,\dots,a_m,0,\dots)}\\
&\mapsto \frac12\delta_{\diag(0,a_2,\dots,a_m,0,\dots)}+\frac12\left(\frac12\delta_{\diag(2a_1,0,a_3,\dots,a_m,0,\dots)}  \right.  \\
&   \quad   \, \left. +\frac12\delta_{\diag(2a_1,2a_2,a_3,\dots,a_m,0,\dots)}\right)\\
&\vdots\\
&\mapsto \sum_{j=1}^m\frac{1}{2^j}\delta_{B_j}+\frac{1}{2^m}\delta_{2A},	
\end{align*}
where  $B_j=\diag(b_{j,1},\dots,b_{j,d})$ is defined by
$$
b_{j,i}=\begin{cases}2a_i&\textrm{ if }i<j,\\
0&\textrm{ if }i=j\textrm{ or }i>m,\\
a_i&\textrm{ if }j<i\leq m.\end{cases}	
$$ 
In particular we directly obtain the formula \eqref{e:staircase2w}, i.e.~that $\gamma=2^{-m}$. 

\gray{LS:\, commented out; here gray for comparison:
We note in passing that the argument used in the proof of Lemma \ref{l:staircase1} could also be used to determine $\gamma$, this time with the commutativity relation \eqref{e:commutativity} applied with the $m\times m$ subdeterminant  \\$X\mapsto \det\left((X_{ij})_{i=1\dots m,j=1\dots m}\right)$.
}

Concerning the lower bound \eqref{e:staircase2b}, we proceed analogously to Lemma \ref{l:staircase1}. First of all, note that 
\begin{equation}\label{e:ex2-lowerbound1}
\mu(\{B_j\})\geq \frac{1}{2^m}\textrm{ for all }j.
\end{equation}
Let $j_*\in\{1,\dots,d\}$ such that $|a_{j_*}|=\max_i|a_i|$. If $j_*<d$ then, by construction, $|B_{d}|\geq 2|a_{j_*}|$, whereas if $j_*=d$, then $|B_{1}|\geq |a_{j_*}|$. In either case 
\begin{equation}\label{e:ex2-lowerbound2}
\max_j|B_{j}|\geq \frac{1}{\sqrt{m}}|A|.
\end{equation} 
From \eqref{e:ex2-lowerbound1}-\eqref{e:ex2-lowerbound2} follows
\begin{equation*}
    \mu(\{X:\,|X|\geq \frac{1}{\sqrt{m}}|A|\})\geq \frac{1}{2^m}
\end{equation*}
as claimed.

\end{proof}

\begin{example}\label{ex:staircase2}
As in the case of Lemma \ref{l:staircase1}, Lemma \ref{l:staircase2} can also be used to obtain a staircase laminate via Proposition \ref{p:staircaseconstruction}. Let $A\in \Lambda^{(m)}\cap\mathcal{D}$ for some $m\geq 2$ and set $A_n=2^nA$ for $n=0,1,2,\dots$. Applying Lemma \ref{l:staircase2} to each $A_n$ leads to finite order laminates $\omega_n=(1-\gamma_n)\mu_n+\gamma_n\delta_{A_n}$ with barycenter $A_{n-1}$ and $\gamma_n=2^{-m}$. Then 
\begin{equation*}
\beta_n=\prod_{k=1}^n\gamma_k=2^{-nm},
\end{equation*}
so that the assumptions of Proposition \ref{p:staircaseconstruction} are satisfied. We obtain a sequence of laminates $\nu^N$ with $\supp\nu^N\subset (\mathcal D\cap \Lambda^{(m-1)})\cup\{A_N\}$ and barycenter $A$. Furthermore, we deduce from \eqref{e:staircase2s} and \eqref{e:staircase2b} that the conditions of Lemma \ref{l:weakLpstaircase} are satisfied with $p=m$, $c_0=2$ and $c_1=2^{-m}$. Consequently, the limiting staircase laminate $\nu^{\infty}$ satisfies the weak $L^m$ bound
\begin{equation}\label{e:staircase2-bound}
    2^{-m(2+p)}|A|^mt^{-m}\leq \nu^{\infty}(\{X:|X|>t\})\leq 2^{1+m}|A|^mt^{-m}
\end{equation}
for all $t>|A|$.
\end{example}

\bigskip

\subsubsection{Example 3}

Our third example is from \cite{astala_faraco_szekelyhidi08} and arises in the theory of very weak solutions to linear elliptic equations with measurable coefficients in the plane (c.f. \eqref{e:k-elliptic}). For any $\mathcal{\rho}\geq 1$ set
\begin{equation*}
E_{\rho}=\left\{\begin{pmatrix} \lambda&0\\0&\rho\lambda\end{pmatrix}R:\,\lambda\geq 0,\,R\in SO(2)\right\}\,.
\end{equation*}

\begin{lemma}\label{l:staircase3}
Let $\mathcal K>1$. For any $x\geq 1$ define $A(x)$ to be the $2\times 2$ diagonal matrix with entries $A(x)=\diag(-x,x)$. There exists a laminate $\omega\in\mathcal{L}(\R^{2\times 2})$ with barycenter $\bar{\omega}=A(x)$ such that $\omega$ can be written as $\omega=(1-\gamma)\mu+\gamma\delta_{A(x+1)}$ for some probability measure $\mu$ with
\begin{equation}\label{e:staircase3s}
    \supp\mu\subset \mathcal{D}\cap (E_{\mathcal K}\cup E_{1/\mathcal K})
\end{equation}
and
\begin{equation}\label{e:staircase3w}
    \gamma=\frac{x}{x+1}\left(1-\frac{1-\mathcal{K}^{-1}}{1+(1+\mathcal{K}^{-1})x}\right).
\end{equation}
Moreover, 
\begin{equation}\label{e:staircase3b}
    \supp\mu\subset \{X:\,\frac{1}{\sqrt{2}}|A(x)|\leq |X|\leq \sqrt{2}|A(x)|\}.
\end{equation}
\end{lemma}

\begin{proof}
This time the splitting sequence for $\omega$ is as follows:
\begin{align*}
    \delta_{A(x)}&\mapsto \alpha_1\delta_{B_1(-x)}+(1-\alpha_1)\delta_{C(x)}\\
    &\mapsto \alpha_1\delta_{B_1(-x)}+\alpha_2\delta_{B_2(x+1)}+\gamma\delta_{A(x+1)}\,,
\end{align*}
where
\begin{equation*}
    B_1(x)=\diag(x,\tfrac{1}{\mathcal{K}}x),\,B_2(x)=\diag(\tfrac{1}{\mathcal{K}}x,x),\,C(x)=\diag(-x,x+1)
\end{equation*}
and
\begin{align*}
    \alpha_1&=\frac{1}{1+x(1+\mathcal{K}^{-1})},\,\alpha_2=\frac{x}{x+1}\frac{1}{1+x(1+\mathcal{K}^{-1})},
\end{align*}
and $\gamma$ is given by \eqref{e:staircase3w}. Observe that $B_1(x)\in E_{\mathcal{K}^{-1}}$ and $B_2(x)\in E_{\mathcal{K}}$ for all $x\in \R$. 
The bound \eqref{e:staircase3b} follows from 
\begin{equation}\label{e:staircase3-est}
    |x|\leq |B_1(x)|=|B_2(x)|\leq |A(x)|\leq \sqrt{2}|x|.
\end{equation}
\end{proof}

\begin{example}\label{ex:staircase3}
We apply Proposition \ref{p:staircaseconstruction} to Lemma \ref{l:staircase3} to obtain a staircase laminate with properties as used in Section 3.2 of \cite{astala_faraco_szekelyhidi08}. Let $A_n=A(n)$ for $n=1,2,\dots$. We obtain, using Lemma \ref{l:staircase3}, finite order laminates $\omega_n=(1-\gamma_n)\mu_n+\gamma_n\delta_{A_{n+1}}$ with barycenter $A_n$ and $\gamma_n$ given by \eqref{e:staircase3w} with $x=n$. Then 
\begin{equation*}
    \beta_n=\frac{1}{n+1}\prod_{k=1}^n\left(1-\frac{\mathcal{K}-1}{\mathcal{K}+k(\mathcal{K}+1)}\right).
\end{equation*}
We estimate
\begin{equation*}
    \frac{1}{n+1}\prod_{k=1}^n\left(1-\frac{\mathcal{K}-1}{\mathcal{K}+1}\frac{1}{k}\right)\leq \beta_n\leq \frac{1}{n+1}\prod_{k=1}^n\left(1-\frac{\mathcal{K}-1}{\mathcal{K}+1}\frac{1}{k+1}\right).
\end{equation*}
By taking logarithms\footnote{Let $\alpha_n=\prod_{k=1}^n(1-\frac{c}{k})$ for some $0<c<1$. Then $\log\alpha_n=-\sum_{k=1}^nf(k)$, where $f(x)=\log(x)-\log(x-c)$. By direct computation we verify that $f:[1,\infty)\to[0,\infty)$ is monotone decreasing, hence $\int_1^nf(x)\,dx\leq \sum_{k=1}^nf(k)\leq \int_1^{n+1}f(x)\,dx$. On the other hand, by evaluating the integral we see that
\begin{equation*}
    \int_1^nf(x)\,dx=c\log n-(n-c)\log(\tfrac{n-c}{n}) +(1-c)\log(1-c)=c\log n+O(1)
\end{equation*}
as $n\to \infty$. The assertion \eqref{e:takinglogarithms} follows with $c=\frac{\mathcal{K}-1}{\mathcal{K}+1}$.
}, we deduce that there exists $C=C(\mathcal{K})>1$ such that 
\begin{equation}\label{e:takinglogarithms}
    C^{-1}n^{-\bar{q}}\leq \beta_n\leq Cn^{-\bar{q}}
\end{equation}
with $\bar{q}=\frac{2\mathcal{K}}{\mathcal{K}+1}$. Thus, the assumptions of Proposition \ref{p:staircaseconstruction} are satisfied and we obtain a sequence of laminates $\nu^N$ with $\supp\nu^N\subset (\mathcal{D}\cap (E_{\mathcal K}\cup E_{1/\mathcal K})\cup\{A_{N+1}\}$ and barycenter $A(1)=\diag(-1,1)$. From \eqref{e:staircase3-est} and \eqref{e:takinglogarithms} we deduce
\begin{equation*}
    C^{-1}\leq \beta_n|A_n|^{\bar{q}}\leq C
\end{equation*}
for some $C=C(\mathcal{K})>1$, and then Lemma \ref{l:weakLpstaircase} implies
\begin{equation}\label{e:staircase3-weak}
\tilde C^{-1}t^{-\frac{2\mathcal{K}}{\mathcal{K}+1}}\leq \nu^{\infty}(\{X:\,|X|>t\})\leq \tilde Ct^{-\frac{2\mathcal{K}}{\mathcal{K}+1}}\quad\textrm{ for all }t>1.
\end{equation}
for some $\tilde C>1$.
\end{example}

\subsubsection{Example 4}

Our fourth example is from \cite{ColomboTione2022} and arises in the theory of the $p$-harmonic operator. For any $p\in (1,\infty)$ let
\begin{equation}\label{e:CT22-Kp}
    K_p=\left\{\begin{pmatrix}\lambda&0\\0&\lambda^{p-1}\end{pmatrix}R:\,\lambda\geq 0,\,R\in SO(2)\right\}.
\end{equation}

\begin{lemma}\label{l:staircase4}
Let $1<p<2$ and $b>1$. For any $x\geq 1$ define $A(x)$ to be the $2\times 2$ diagonal matrix with entries $A(x)=\diag(bx,-x^{p-1})$. There exists a laminate $\omega\in\mathcal{L}(\R^{2\times 2})$ with barycenter $\bar{\omega}=A(x)$ such that $\omega$ can be written as $\omega=(1-\gamma)\mu+\gamma\delta_{A(x+1)}$ for some probability measure $\mu$ with
\begin{equation}\label{e:staircase4s}
    \supp\mu\subset \mathcal{D}\cap K_p
\end{equation}
and
\begin{equation}\label{e:staircase4w}
    \gamma=\left(1-\frac{b}{(b+1)(1+x)}\right)\left(1-\frac{(1+x^{-1})^{p-1}-1}{b^{p-1}+(1+x^{-1})^{p-1}}\right).
\end{equation}
Moreover, 
\begin{equation}\label{e:staircase4b}
    \supp\mu\subset \{X:\,\frac{1}{\sqrt{2}b}|A(x)|\leq |X|\leq \sqrt{2}b|A(x)|\}.
\end{equation}
\end{lemma}

\begin{proof}
This time the splitting sequence for $\omega$ is as follows:
\begin{align*}
    \delta_{A(x)}&\mapsto \alpha_1\delta_{B_1(bx)}+(1-\alpha_1)\delta_{C(x)}\\
    &\mapsto \alpha_1\delta_{B_1(bx)}+\alpha_2\delta_{B_2(x+1)}+\gamma\delta_{A(x+1)}\,,
\end{align*}
where
\begin{equation*}
    B_1(x)=\diag(x,x^{p-1}),\,B_2(x)=\diag(-x,-x^{p-1}),\,C(x)=\diag(bx,-(x+1)^{p-1})
\end{equation*}
and
\begin{align*}
    \alpha_1&=\frac{(x+1)^{p-1}-x^{p-1}}{(bx)^{p-1}+(x+1)^{p-1}},\,\alpha_2=\frac{b}{(b+1)(x+1)}\frac{(bx)^{p-1}-x^{p-1}}{(bx)^{p-1}+(x+1)^{p-1}},
\end{align*}
and $\gamma$ is given by \eqref{e:staircase4w}. Observe that $B_1(x), B_2(x)\in K_p$ for all $x\in \R$.

The bound \eqref{e:staircase4b} follows from 
\begin{equation}\label{e:staircase4-est}
    |x|\leq |B_1(x)|=|B_2(x)|\leq |A(x)|\leq \sqrt{2}b|x|.
\end{equation}

\end{proof}

\begin{example}\label{ex:staircase4}
We apply Proposition \ref{p:staircaseconstruction} to Lemma \ref{l:staircase4} to obtain a staircase laminate with properties as used in \cite{ColomboTione2022}. Let $A_n=A(n)$ for $n=1,2,\dots$. We obtain, using Lemma \ref{l:staircase4}, finite order laminates $\omega_n=(1-\gamma_n)\mu_n+\gamma_n\delta_{A_{n+1}}$ with barycenter $A_n$ and $\gamma_n$ given by \eqref{e:staircase4w} with $x=n$. 
The main observation is that, for a suitable $b >1$, we have
\begin{equation}  \label{eq:bound_wN}
    \beta_n := \prod_{j=1}^n \gamma_j \sim n^{- \bar q}
    \quad \text{with $\bar q \in (1,p)$.}
\end{equation} 
To see this, note that (arguing analogously to \eqref{e:takinglogarithms})
\begin{equation}\label{eq:define_barq} 
- \log \gamma_{n+1} = \bar q  n^{-1} + \mathcal O(n^{-2}) \quad 
\text{with $\bar q =  \frac{p-1}{b^{p-1} + 1} 
+ \frac{b}{b+1}$} 
\end{equation}
Thus 
$ \log \beta_n -  \bar q \log n$ is uniformly bounded from above and below and hence 
\begin{equation}  \label{eq:bound_wN2}
   \frac1C n^{-\bar q} \le  \beta_n \le C n^{-\bar q}
\end{equation}
for some constant $C$. Clearly $\bar q < p$. Moreover, with $a = b^{-1}$ we have
$$ \bar q - 1 =  \frac{p-1}{b^{p-1} + 1} 
- \frac{1}{b+1} = \frac{(p-1) a^{p-1}}{1+ a^{p-1}} - \frac{a}{1+a}.
$$
Since $p-1 \in (0,1)$ we have $a^{p-1} \gg a$ for $0 < a \ll 1$
and we conclude that $\bar q > 1$ for sufficiently 
small $a > 0$ or, equivalently, for sufficiently large $b >1$.

 Thus, the assumptions of Proposition \ref{p:staircaseconstruction} are satisfied and we obtain a sequence of laminates $\nu^N$ with $\supp\nu^N\subset (\mathcal{D}\cap K_p\cup\{A_{N+1}\}$ and barycenter $A(1)=\diag(b,-1)$. From \eqref{e:staircase4-est} and \eqref{eq:bound_wN2} we deduce
\begin{equation*}
    C^{-1}\leq \beta_n|A_n|^{\bar{q}}\leq C
\end{equation*}
for some $C=C(p,b)>1$, and then Lemma \ref{l:weakLpstaircase} implies
\begin{equation}\label{e:staircase4-weak}
\tilde C^{-1}t^{-\bar{q}}\leq \nu^{\infty}(\{X:\,|X|>t\})\leq \tilde Ct^{-\bar{q}}\quad\textrm{ for all }t>1.
\end{equation}
for some $\tilde C=\tilde C(p,b)>1$.
\end{example}

\gray{LS:\, commented out; here gray for comparison:

\begin{lemma}[Laminates for the $p$-Laplacian in the plane, $p>2$]\label{l:staircase4}
Let $b>1$ and $p>2$. For any $x\geq 1$ define $A(x)$ to be the $2\times 2$ diagonal matrix with entries $A(x)=\diag(x,-b^{p-1}x|x|^{p-2})$. There exists a laminate $\omega\in \mathcal{L}(\R^{2\times 2})$ with barycenter $\bar{\omega}=A(x)$ such that $\omega$ can be written as $\omega=(1-\gamma)\mu+\gamma\delta_{A(2x)}$ for some probability measure $\mu$ with 
\begin{equation}\label{e:staircase4s}
\supp\mu\subset\mathcal{D}\cap K_p
\end{equation}
and 
\begin{equation}\label{e:staircase4w}
    \gamma=\frac{1+2b}{2+2b}\frac{1+b^{p-1}}{1+(2b)^{p-1}}.
\end{equation}
Moreover, there exists $C=C(p,b)>1$ such that
\begin{equation}\label{e:staircase4b}
\supp\mu\subset\{X:\,C^{-1}|A(x)|\leq |X|\leq C|A(x)|\}
\end{equation}
and furthermore
\begin{equation}\label{e:staircase4c}
\supp\mu\subset\{X:\,C^{-1}|A(x)|\leq |\pi(X)|^{p-1}\leq C|A(x)|\}\,.
\end{equation}

\end{lemma}

\begin{proof}
We proceed again as in Lemma \ref{l:staircase1} by exhibiting the splitting sequence of $\omega$ as in \eqref{e:splitting}. To this end we define auxiliary diagonal matrices
\begin{equation*}
    B(x)=\diag(x,x|x|^{p-2}),\, C(x)=\diag(x,-(2b)^{p-1}x|x|^{p-2})
\end{equation*}
for any $x\in\R$. Observe that $K_p\cap\mathcal{D}=\{B(x):x\in\R\}$. Direct computation shows that 
\begin{align*}
    \delta_{A(x)}&\mapsto \alpha_1\delta_{B(x)}+(1-\alpha_1)\delta_{C(x)}\\
    &\mapsto \alpha_1\delta_{B(x)}+\alpha_2\delta_{B(-2bx)}+\gamma\delta_{A(2x)}
\end{align*}
is a splitting sequence, where
\begin{align*}
    \alpha_1&=\frac{(2b)^{p-1}-b^{p-1}}{1+(2b)^{p-1}},\quad \alpha_2=\frac{1}{2+2b}\frac{1+b^{p-1}}{1+(2b)^{p-1}}\\
    \gamma&=\frac{1+2b}{2+2b}\frac{1+b^{p-1}}{1+(2b)^{p-1}}\,,
\end{align*}
thus verifying \eqref{e:staircase4s}-\eqref{e:staircase4w}. 

Next, since $x,b\geq 1$ by assumption, we have
\begin{equation}\label{e:staircase4-est1}
|x|^{p-1}\leq |B(x)|\leq |A(x)|\leq |B(-2bx)|\leq 2(2b)^{p-1}|x|^{p-1}
\end{equation}
as well as
\begin{equation}\label{e:staircase4-est2}
|x|= |\pi(B(x))|=|\pi(A(x))|\leq |\pi(B(-2bx))|\leq 2b|x|.
\end{equation}
The assertions \eqref{e:staircase4b}-\eqref{e:staircase4c} follow immediately.

\end{proof}

\begin{example}\label{ex:staircase4}
We now apply Proposition \ref{p:staircaseconstruction} to Lemma \ref{l:staircase4} to obtain a staircase laminate with properties as used in \cite{ColomboTione2022}. Let $x>1$ and set $A_n=A(2^nx)$ for $n=0,1,2,\dots$. We obtain, using Lemma \ref{l:staircase4}, finite order laminates $\omega_n=(1-\gamma_n)\mu_n+\gamma_n\delta_{A_n}$, $n=1,2,\dots$,  with barycenter $A_{n-1}$ and $\gamma_n=\gamma$ given by \eqref{e:staircase4w}. Then $\beta_n=\prod_{k=1}^n\gamma_k=\gamma^n$, and since $\gamma<1$, we see that the assumptions of Proposition \ref{p:staircaseconstruction} are satisfied. We obtain a sequence of laminates $\nu^N$ with $\supp\nu^N\subset (\mathcal D\cap K_p)\cup\{A_N\}$ and barycenter $A(x)$.


In order to be able to apply Lemma \ref{l:weakLpstaircase}, we 
we claim (following Section 5 in \cite{ColomboTione2022}) that there exists $b>1$ sufficiently large, for which 
\begin{equation}\label{e:staircase4-CT}
    \gamma=2^{-\bar{q}}\textrm{ for some }\bar{q}\in (p-1,p).
\end{equation}
To this end set $f_p(a)=\frac{1+a/2}{1+a}\frac{1+a^{p-1}}{1+(a/2)^{p-1}}$. Comparing with \eqref{e:staircase4w} we observe that $\gamma=2^{-(p-1)}f_p(1/b)$. By direct calculation, and using $p>2$ we can verify that $f_p(0)=1$ and $f'_p(0)=-1/2$. Therefore for sufficiently large $b\gg 1$ we have $f_p(1/b)<1$, from which \eqref{e:staircase4-CT} follows. 
In turn, \eqref{e:staircase4-est1} and \eqref{e:staircase4-CT} imply the estimate
\begin{equation}
    C^{-1}|x|^{\bar{q}}\leq \beta_n|A_n|^{\frac{\bar{q}}{p-1}}\leq C|x|^{\bar{q}}
\end{equation}
for some $C=C(b,p)>1$. Combined with the bounds \eqref{e:staircase4b}, we apply Lemma \ref{l:weakLpstaircase} to deduce the weak $L^\frac{\bar{q}}{p-1}$ bounds
\begin{equation}\label{e:staircase4-weak1}
\tilde C^{-1}|A|^{\frac{\bar{q}}{p-1}}t^{-\frac{\bar{q}}{p-1}}\leq \nu^{\infty}(\{X:\,|X|>t\})\leq \tilde C|A|^{\frac{\bar{q}}{p-1}}t^{-\frac{\bar{q}}{p-1}}
\end{equation}
for some $\tilde C=\tilde C(b,p)>1$, for all $t>|A|$.
In addition, in an entirely analogous manner to Lemma \ref{l:weakLpstaircase} one can prove that the bounds \eqref{e:staircase4c} imply
\begin{equation}\label{e:staircase4-weak1}
\tilde C^{-1}|\pi(A)|^{\bar{q}}t^{-\bar{q}}\leq \nu^{\infty}(\{X:\,|\pi(X)|>t\})\leq \tilde C|\pi(A)|^{\bar{q}}t^{-\bar{q}}
\end{equation}
for all $t>|\pi(A)|$.

\end{example}
}

\bigskip

\section{Differential inclusions for Sobolev maps}\label{s:general}

In this section we return to general differential inclusions of the form
\begin{equation}\label{e:diffincl1}
    \nabla u(x)\in K\textrm{ for almost every }x\in\Omega,
\end{equation}
where $u:\Omega\to \R^d$, $\Omega\subset\R^m$ is a regular domain, and $K\subset \R^{d\times m}$ is a prescribed (typically unbounded) closed set. As mentioned in the introduction we complement \eqref{e:diffincl1} with affine boundary conditions. 

In this section we recall the definition of the property
'$K$ can be reduced to $K'$ in weak $L^p$' and  verify the three key features of this property  announced in 
Section \ref{eq:pmain_Calpha_uk}:
\begin{itemize}
\item existence of solutions if $\R^{d \times m}$ can be reduced to $K$
\item stability of the reduction property  under iteration
\item sufficiency of staircase laminates for the reduction property
\end{itemize}

\subsection{Exact solutions}
Recall from our discussion in the introduction, that our general strategy is to solve \eqref{e:diffincl1} by first obtaining the following approximation result: for any regular domain $\Omega\subset\R^m$, any $A\in \R^{d\times m}$, $b\in\R^d$ and any $s\in(1,\infty)$, $\eps,\alpha\in(0,1)$ there exists a piecewise affine map $u\in W^{1,1}(\Omega)\cap C^{\alpha}(\overline{\Omega})$ with $u=l_{A,b}$ on $\partial\Omega$ and    
\begin{align*}
\int_{\{x\in\Omega:\,\nabla u(x)\notin K\}}(1+|\nabla u|^s)\,dx&<\eps|\Omega|, 
\\
|\{x\in\Omega:\,|\nabla u(x)|>t\}|\leq M^p (1+|A|^p)|\Omega|&t^{-p}\,\textrm{ for all }t>0.  
\end{align*}
If $K$ has this property with some $M$ and $p>1$, we say that \emph{$\R^{d\times m}$ can be reduced to $K$ in weak $L^p$} (c.f.~Definition \ref{d:reduced}). Our first goal in this section is to show how this property leads to solvability of the differential inclusion \eqref{e:diffincl1}.

\begin{theorem}\label{th:generalsolution}
Let $K\subset \R^{d\times m}$ and $1<p<\infty$ such that $\R^{d\times m}$ can be reduced to $K$ in weak $L^p$. Then for any regular domain $\Omega\subset \R^m$, any $A\in \R^{d\times m}$, $b\in \R^d$ and any $\delta>0$, $\alpha\in (0,1)$ there exists a piecewise affine map $u\in W^{1,1}(\Omega)\cap C^{\alpha}(\overline{\Omega})$ with $u=l_{A,b}$ on $\partial\Omega$ such that
\begin{subequations}
\begin{equation}\label{e:solution-ae}
\nabla u(x)\in K\textrm{ a.e. }x\in \Omega,
\end{equation}
\begin{equation}
\|u-l_{A,b}\|_{C^\alpha(\overline{\Omega})}<\delta,    
\end{equation}
and for all $t>0$
\begin{equation}  \label{eq:weak_Lpbound_exact_thm}
\left|\{x\in \Omega:\,|\nabla u(x)|>t\}\right|\leq 2M^p(1+|A|^p)|\Omega|t^{-p}.    
\end{equation}
\end{subequations}
\end{theorem}

\begin{remark}  \label{re:generalsolution_nonhomogeneous_bound}  
A version of this result also holds if the condition
$$|\{x\in\Omega:\,|\nabla u(x)|>t\}|  
\le M^p (1+|A|^p)|\Omega|t^{-p} \textrm{ for all }t>0
$$
in the definition of 
'$\R^{d\times m}$ can be reduced $K$' 
(see Definition \ref{d:reduced})
is replaced by the weaker condition 
that there exists an $r \ge p$ such that
\begin{equation}  \label{eq:weak_Lp_growth_Ar} |\{x\in\Omega:\,|\nabla u(x)|>t\}|  
\le M^p (1+|A|^r)|\Omega|t^{-p} \textrm{ for all }t>0
\end{equation}
Then the conclusion \eqref{eq:weak_Lpbound_exact_thm}
has to be replaced by the weaker conclusion
\begin{equation}  \label{eq:weak_Lpbound_exact}
\left|\{x\in \Omega:\,|\nabla u(x)|>t\}\right|\leq 2M^p(1+|A|^r)|\Omega|t^{-p}.    
\end{equation}
\end{remark}

\begin{proof} We show the stronger result mentionded in
Remark~\ref{re:generalsolution_nonhomogeneous_bound}.
To do so, we  construct inductively a sequence of piecewise affine maps $u_k\in W^{1,1}\cap C^{\alpha}$, $k=1,2,\dots$ such that $u_k=l_{A,b}$ on $\partial\Omega$ and, with $\Omega_{error}^{(k)}:=\{x\in\Omega_{u_k}:\,\nabla u_k(x)\notin K\}$ we have
\begin{subequations}  
\begin{align}
\int_{\Omega_{error}^{(k)}}(1+|\nabla u_k|^r)\,dx
&<2^{-k}|\Omega|, 
\label{e:pmaink1}\\
\|u_k-l_{A,b}\|_{C^\alpha(\overline{\Omega})}
&<\delta(1-2^{-k}), \label{e:pmaink3}\\
|\{x\in\Omega:\,|\nabla u_k(x)|>t\}|&\leq M^p(1+|A|^r)|\Omega|t^{-p}\sum_{i=0}^{k-1}2^{-i}.\label{e:pmaink2}
\end{align}
\end{subequations}
The map $u_1$ is obtained directly by applying Definition \ref{d:reduced} with $s=r$ (and, for $r\geq p$ using \eqref{eq:weak_Lp_growth_Ar} in Remark \ref{re:generalsolution_nonhomogeneous_bound}). In addition, invoking the rescaling and covering argument from Section \ref{ss:basic} we can ensure $\|u_1-l_{A,b}\|_{C^{\alpha}(\overline{\Omega})}< \delta/2$.

For the inductive step we assume the existence of $u_k$. Since $u_k$ is piecewise affine, there exist pairwise disjoint open subsets $\Omega_{i}\subset\Omega_{error}^{(k)}$ such that $|\Omega_{error}^{(k)}\setminus \bigcup_{i=1}^\infty\Omega_{i}|=0$ and $u_{k}=l_{A_{i},b_{i}}$ in $\Omega_{i}$.

We then apply Definition \ref{d:reduced} and \eqref{eq:weak_Lp_growth_Ar} in each $\Omega_{i}$ (again with $s=r$) to obtain piecewise affine maps $v_{i}\in  W^{1,1}(\Omega_{i})\cap C^{\alpha}(\overline\Omega_{i})$ with $v_{i}=l_{A_{i},b_{i}}$ on $\partial\Omega_{i}$ 
such that, with 
$$
\tilde\Omega_i:=\{x\in\Omega_{i}:\,\nabla v_{i}(x)\notin K\}
$$ 
we have
\begin{subequations}  
\begin{align}
\int_{\tilde\Omega_i}(1+|\nabla v_{i}|^r)\,dx&<2^{-(k+1)}|\Omega_{i}|,\label{e:pmainik1}
\\
\|v_i-l_{A_i,b_i}\|_{C^{\alpha}(\overline{\Omega_i})}&<\delta 2^{-k-2},\label{e:pmainik3}\\
|\{x\in\Omega_{i}:\,|\nabla v_{i}(x)|>t\}|&\leq M^p(1+|A_{i}|^r)|\Omega_{i}|t^{-p}.\label{e:pmainik2}
\end{align}
\end{subequations}
Using the basic gluing and covering/rescaling arguments from Section \ref{ss:basic} we thus obtain the piecewise affine map $u_{k+1}\in W^{1,1}\cap C^{\alpha}$ with $u_{k+1}=l_{A,b}$ on $\partial\Omega$ and with the following properties:
\begin{itemize}
\item By construction 
$$
\nabla u_{k+1}=\begin{cases} \nabla v_i&\textrm{ in }\Omega_i\\
\nabla u_k&\textrm{a.e. outside }\Omega_{error}^{(k)}\end{cases}
$$
and in particular $|\Omega_{error}^{(k+1)}\setminus\bigcup_{i=1}^\infty \tilde\Omega_i|=0$. 
\item Consequently, using \eqref{e:pmainik1},
\begin{align*}
\int_{\Omega_{error}^{(k+1)}}(1+|\nabla u_{k+1}|^r)\,dx&=\sum_i\int_{\tilde\Omega_i}(1+|\nabla v_{i}|^r)\,dx\\
&<2^{-(k+1)}\sum_i|\Omega_i|\leq 2^{-(k+1)}|\Omega|
\end{align*}
and we obtain \eqref{e:pmaink1} for $k+1$.
\item Using \eqref{e:pmainik2}, for any $t\geq 1$

\begin{align*}
& \,  |\{x\in\Omega:|\nabla u_{k+1}(x)|>t\}|\\
\leq & \, |\{x\in\Omega:\,|\nabla u_{k}(x)|>t\}| +\sum_i|\{x\in\Omega_i:\,|\nabla v_{i}(x)|>t\}|\\
\leq & \,  M^p(1+|A|^r)|\Omega|t^{-p}\sum_{i=0}^{k-1}2^{-i}+M^pt^{-p}\sum_i(1+|A_{i}|^r)|\Omega_{i}|\\
= & \, M^p(1+|A|^r)|\Omega|t^{-p}\sum_{i=0}^{k-1}2^{-i}+M^pt^{-p}
\int_{\Omega_{error}^{(k)}}(1+|\nabla u_{k}|^r)\,dx\\
\leq & \,  M^p(1+|A|^r)|\Omega|t^{-p}\sum_{i=0}^{k}2^{-i}.
\end{align*}

%
%
%
%
This shows \eqref{e:pmaink2} for $k+1$ (note that the bound is trivial for $t < 1$ since $M \ge 1$). 
\item Moreover, we obtain
\begin{equation}  \label{eq:pmain_Calpha_uk}
\|u_{k+1}-u_k\|_{C^{\alpha}(\overline\Omega)}\leq 2\max_i\|v_i-l_{A_i,b_i}\|_{C^{\alpha}(\overline\Omega_i)}\leq \delta 2^{-k-1},
\end{equation}
from which   \eqref{e:pmaink3} follows for $k+1$.
\end{itemize}
This concludes the induction step. 

Having constructed the sequence $(u_k)_k$,  we now show that a limit $k\to \infty$ exists and satisfies the properties in the statement of the theorem. First note that $u_{k+1}=u_k$ outside $\Omega_{error}^{(k)}$ and  $\Omega_{error}^{(k+1)}\subset \Omega_{error}^{(k)}$ for any $k$. Furthermore, it follows from 
\eqref{e:pmaink1} that $|\Omega_{error}^{(k)}|\to 0$ as $k\to\infty$. Thus  for almost every $x\in\Omega$ the pointwise limit $u(x):=\lim_{k\to\infty}u_k(x)$ exists and $u$ is piecewise affine. 
From \eqref{eq:pmain_Calpha_uk} we further obtain that $u_k$ converges to $u$ in $C^{\alpha}(\overline\Omega)$ with  
$$
u=l_{A,b}\,\textrm{ on }\partial\Omega,\quad\textrm{ and }\|u-l_{A,b}\|_{C^\alpha(\overline{\Omega})}<\delta.
$$
From \eqref{e:pmaink2} we deduce the uniform weak $L^p$ bound
$$
|\{x\in\Omega:\,|\nabla u_{k}(x)|>t\}|
\leq 2M^p(1+|A|^r)|\Omega|t^{-p},  
$$
which in particular implies that the sequence $\{u_k\}$ is uniformly bounded in $W^{1,q}(\Omega)$ for any $q<p$. It follows that the limit satisfies $u\in W^{1,q}(\Omega)$ for any $q<p$, as well as the same weak $L^p$ bound. This proves the statement of the theorem.
\end{proof}

\subsection{Iteration property}

In certain cases it may be simpler to show the reduction property $\R^{d\times m}$ to $K$, required in Theorem \ref{th:generalsolution}, via several intermediate stages.
The key point is to tie together these different reduction stages whilst retaining control of the appropriate weak $L^p$ bound. This is the  subject of the following statement:
\begin{theorem}\label{th:iteration}
Let $K,K',K''\subset\R^{d\times m}$. Suppose that for some $1\leq p,q<\infty$ with $p\neq q$: $K$ can be reduced to $K'$ in weak $L^p$ and that $K'$ can be reduced to $K''$ in weak $L^q$. Then $K$ can be reduced to $K''$ in weak $L^r$ where $r=\min\{p,q\}$.
\end{theorem}

Before we give the proof of Theorem \ref{th:iteration}, we collect some useful classical estimates. 
\begin{remark}\label{rmk:weakstrong}\hfill
\begin{itemize}
    \item Let us define, for any $1\leq p<\infty$, $\langle A\rangle_p:=(1+|A|^p)^{1/p}$, and $\langle A\rangle_{\infty}=\max\{1,|A|\}$. By direct computation one can check that for any $A$ the function $p\mapsto \langle A\rangle_p$ is monotonic decreasing.
    \item Using Chebyshev's inequality, a strong $L^p$ estimate of the type 
    $$
    \int_{\Omega}|\nabla u|^p\,dx\leq M^p(1+|A|^p)|\Omega|
    $$
    implies the weak $L^p$ estimate \eqref{eq:iteration_weakLp}. 
    \item Conversely, the weak $L^p$ estimate \eqref{eq:iteration_weakLp} implies a strong $L^q$ estimate for any $q<p$: for any $a>0$
    \begin{align*}
        \int_{\Omega}|\nabla u|^q\,dx&=q\int_0^\infty t^{q-1}\left|\{|\nabla u|>t\}\right|\,dt\\
        &\leq q\int_0^at^{q-1}\,dt|\Omega|+q\int_a^\infty t^{q-p-1}\,dt\,(M\langle A\rangle_p)^p|\Omega|\\
        &=a^q|\Omega|\left(1+\frac{q}{q-p}\left(\frac{M\langle A\rangle_p}{a}\right)^p\right)
    \end{align*}
    Choosing $a=\left(\frac{q}{p-q}\right)^{1/p}M\langle A\rangle_p$ and using the monotonicity of $p\mapsto \langle A\rangle_p$, we deduce
    $$
    \int_{\Omega}|\nabla u|^q\,dx\leq 2\left(\frac{q}{p-q}\right)^{q/p}M^q(1+|A|^q)|\Omega|.
    $$
\end{itemize}
\end{remark}

\begin{proof}[Proof of Theorem~ \ref{th:iteration}]
Let $\Omega\subset\R^m$ be a regular domain and fix $\eps>0$, $1\leq s < \infty$ and $\alpha \in [0,1)$. 
Let $A \in K$ and $b \in \R^d$.

By the assumption that $K$ can be reduced to $K'$ in weak $L^p$ we know that there exists a map $u' \in C^\alpha(\overline \Omega) \cap W^{1,1}(\Omega)$
with $u' = l_{A,b}$ on $\partial \Omega$ and mutually disjoint open sets $\Omega_i'$ and a nullset $\mathcal N'$  with $\Omega = \mathcal N' \cup \bigcup_{i} \Omega_i'$ such that $u'$ is affine on $\Omega_i'$ (i.e.~$u'\vert_{\Omega_i'}=l_{A_i,b_i}$) and 
\begin{equation}  \label{eq:error1}
\int_{\Omega'_{error} }  (1 + |\nabla u'|)^s \, dx <  \eps/2 |\Omega|
\end{equation}
where 
\begin{equation}\label{eq:error1defset}
\Omega'_{error}  := \{ x \in \mathring{\Omega}_{u'}:  \nabla u' \notin K'  \}\cup \mathcal N'.
\end{equation}
Moreover, for any $t>1$
\begin{equation}\label{e:composition_est1}
\left|\{|\nabla u'|>t\}\right|\leq (M'\langle A\rangle_p)^p|\Omega|t^{-p}
\end{equation}
Let $G$ be the set of ``good'' indices $G = \{ i :  \text{$\nabla u' \in K'$ in   $\Omega_i'$} \}$.

In each open set $\Omega_i'$ with $i\in G$ apply the assumption that $K'$ can be reduced to $K''$ in $L^p$ with $A_i\in K'$ and $b_i$.
This yields piecewise affine maps $u''_i \in C^\alpha(\overline{\Omega_i'}) \cap W^{1,1}(\Omega_i')$ with $u''_i = u'$ on $\partial \Omega_i'$, and a closed nullset $\mathcal{N}''_i$ such that $u''_i$ is locally affine on $\Omega_i'\setminus\mathcal{N}''_i$. Moreover,
\begin{equation}   \label{eq:error2}
\int_{\Omega''_{error,i} }  (1 + |\nabla u''_i|)^s \, dx \le \eps/2 |\Omega_i'|
\end{equation}
where 
\begin{equation}
\Omega''_{error,i}  := \{ x \in \mathring\Omega_{u''_i} :  \nabla u''_i(x) \notin K''  \}\cup \mathcal{N}''_i,
\end{equation}
and
\begin{equation}\label{e:composition-est2}
\left|\{x\in \Omega_i':|\nabla u''_i(x)|>t\}\right|\leq (M''\langle A_i\rangle_q)^q|\Omega_i'|t^{-q}
\end{equation}
By the gluing argument the map $u''$ defined by 
\begin{equation}
u''(x) = \begin{cases} u''_i(x) & \text{if $x \in \Omega_i, i\in G$} \\
u'(x) & \text{else,}
\end{cases}
\end{equation}
satisfies $u'' \in C^\alpha(\overline \Omega) \cap W^{1,1}(\Omega)$ and $u'' = l_{A,b}$ on $\partial \Omega$. 
Moreover $u''$ is piecewise affine (locally affine on $\Omega\setminus \mathcal{N}'\cup\bigcup_i\mathcal{N}''_i$) and $\nabla u''(x) \in K''$ if $x \in \mathring{\Omega}_{u''} \setminus  \Omega''_{error}$ where
$$  \Omega''_{error} := \bigcup_{i \in G}  \Omega''_{error,i}   \, \cup \Omega'_{error}.$$
Summing  \eqref{eq:error2} over $i \in G$ and using   \eqref{eq:error1} as well as  the fact that   $u''$ and $u'$ agree on
 $\Omega'_{error}$ we get
\begin{equation}
\int_{ \Omega''_{error}} (1 + |\nabla u''|)^s \, dx < \eps |\Omega|.
\end{equation}
It only remains to show that   \eqref{eq:iteration_weakLp} holds with exponent $r=\min\{p,q\}$. More precisely, we claim the estimate
\begin{equation}\label{e:finalestr}
    \bigl|\{|\nabla u|>t\}\bigr|\leq 4(1+\tfrac{q}{|q-p|})(M'M'')^r(1+|A|^r)|\Omega|t^{-r}.
\end{equation}
In order to show this we treat the cases $p<q$ and $p>q$ separately.

\smallskip

\noindent{\bf The case $p<q$}\hfill

We calculate, for any $t>0$:
\begin{align*}
    \bigl|\{x\in \Omega&:\,|\nabla u(x)|>t\}\bigr|=\sum_{i} \left|\{x\in \Omega_i':\,|\nabla u(x)|>t\}\right|\\
    &=\sum_{i: |A_i|>t/M''} \left|\{x\in \Omega_i':\,|\nabla u(x)|>t\}\right|+\sum_{i: |A_i|\leq t/M''} \left|\{x\in \Omega_i':\,|\nabla u(x)|>t\}\right|\\
    &\leq \sum_{i: |A_i|>t/M''} |\Omega'_i|+\sum_{i: |A_i|\leq t/M''} \left|\{x\in \Omega_i':\,|\nabla u(x)|>t\}\right|\\
    &= \bigl|\{x\in \Omega:\,|\nabla u'(x)|>t\}\bigr|+\sum_{i\in G: |A_i|\leq t/M''} \left|\{x\in \Omega_i':\,|\nabla u''_i(x)|>t\}\right|.
\end{align*}
In the last sum we could restrict to $i\in G$, because $M''\geq 1$ and if $i\notin G$ then $\nabla u(x)=\nabla u_i'(x)=A_i$ for all $x\in\Omega_i'$. 
The first term can directly estimated using \eqref{e:composition_est1}. On the second term we use \eqref{e:composition-est2} to obtain
\begin{align*}
   \sum_{i\in G: |A_i|\leq t/M''}& \bigl|\{x\in \Omega_i':\,|\nabla u''_i(x)|>t\}\bigr|\leq (M'')^qt^{-q}\sum_{i\in G: |A_i|\leq t/M''}(1+|A_i|^q)|\Omega_i'|\\
   &\leq (M'')^qt^{-q}\int_{\{|\nabla u'|\leq t/M''\}}(1+|\nabla u'|^q)\,dx\\
   &\leq (M'')^q|\Omega|t^{-q}+q(M'')^qt^{-q}\int_0^{t/M''}s^{q-1}|\{|\nabla u_i'|>s\}|\,ds\\
   &\overset{\textrm{\eqref{e:composition_est1}}}{\leq}(M'')^q|\Omega|t^{-q}+q(M'')^q(M')^p\langle A\rangle_p^p|\Omega|t^{-q}\int_0^{t/M''}s^{q-p-1}\,ds\\
   &=(M'')^q|\Omega|t^{-q}+\frac{q}{q-p}(M'')^p(M')^p\langle A\rangle_p^p|\Omega|t^{-p}\,.
\end{align*}
Putting everything together we deduce
\begin{equation*}
    \bigl|\{x\in \Omega:\,|\nabla u(x)|>t\}\bigr|\leq (1+\tfrac{q}{q-p})(M'M'')^p(1+|A|^p)|\Omega|t^{-p}+(M'')^qt^{-q}|\Omega|.
\end{equation*}
If $t\geq M''$, then $(M'')^qt^{-q}\leq (M'')^pt^{-p}$ (since $p<q$), hence in this case we can estimate
\begin{equation}\label{e:finalestp<q}
    \bigl|\{x\in \Omega:\,|\nabla u(x)|>t\}\bigr|\leq (2+\tfrac{q}{q-p})(M'M'')^p(1+|A|^p)|\Omega|t^{-p}.
\end{equation}
On the other hand, if $t< M''$, then the right hand side of \eqref{e:finalestp<q} is bounded below by $|\Omega|$, which is the trivial upper bound for $|\{|\nabla u(x)|>t\}|$. Therefore in this case \eqref{e:finalestp<q} is also valid.

\smallskip

\noindent{\bf The case $p>q$}\hfill

This time we calculate
\begin{align*}
    \bigl|\{x\in \Omega&:\,|\nabla u(x)|>t\}\bigr|=\sum_{i} \left|\{x\in \Omega_i':\,|\nabla u(x)|>t\}\right|\\
    &=\sum_{i\in G} \left|\{x\in \Omega_i':\,|\nabla u''_i(x)|>t\}\right|+\sum_{i\notin G} \left|\{x\in \Omega_i':\,|\nabla u'(x)|>t\}\right|\\
    &\overset{\textrm{\eqref{eq:error1defset}}}{=}\sum_{i\in G} \left|\{x\in \Omega_i':\,|\nabla u''(x)|>t\}\right|+\bigl|\{x\in \Omega_{error}':\,|\nabla u'(x)|>t\}\bigr|\\
    &\overset{\textrm{\eqref{e:composition-est2}}}{\leq}(M'')^qt^{-q}\sum_{i\in G}(1+|A_i|^q)|\Omega_i'|+t^{-q}\int_{\Omega_{error}'}|\nabla u'|^q\,dx\\
    &\leq (M'')^qt^{-q}|\Omega|+(M'')^qt^{-q}\int_{\Omega}|\nabla u'|^q\,dx\\
    &\leq (M'')^qt^{-q}|\Omega|+2\left(\tfrac{q}{p-q}\right)^{q/p}(M')^q(M'')^q(1+|A|^q)|\Omega|t^{-q},
\end{align*}
where in the last inequality we have used Remark \ref{rmk:weakstrong}. Since $q<p$ we conclude 
\begin{equation}\label{e:finalestp>q}
    \bigl|\{x\in \Omega:\,|\nabla u(x)|>t\}\bigr|\leq 2(2+\tfrac{q}{p-q})(M'M'')^q(1+|A|^q)|\Omega|t^{-q}.
\end{equation}

\end{proof}

\begin{remark}
The estimates in Theorem \ref{th:iteration} have nothing to do with the gradient structure, they hold for unbounded probability measures as follows: Let $\nu'=\sum_{i=1}^\infty\lambda_i\delta_{A_i}$ be a probability measure with 
\begin{equation}\label{e:nu'weakLp}
    \nu'(\{|X|>t\})\leq M'(1+|A|^p)t^{-p}\quad\textrm{ for all }t>0
\end{equation}
and for every $i$ $\nu''_i$ be a probability measure with 
\begin{equation}\label{e:nu''weakLq}
    \nu''_i(\{|X|>t\})\leq M''(1+|A_i|^q)t^{-q}\quad\textrm{ for all }t>0.
\end{equation}
Assume that $p\neq q<\infty$. Then the measure \begin{equation}\label{e:diamondnu}
\nu=\sum_{i=1}^\infty\lambda_i\nu''_i
\end{equation}
is a probability measure with 
\begin{equation}\label{e:nuweakLr}
     \nu(\{|X|>t\})\leq C_{p,q}M'M''(1+|A|^r)t^{-r}\quad\textrm{ for all }t>0,
\end{equation}
with $r=\min\{p,q\}$ and $C_{p,q}=4(1+\tfrac{q}{|p-q|})$. The proof is entirely analogous to the case of gradients presented above.
\end{remark}

\begin{remark}
If $p=q < \infty$ then it may happen that  $|\nu'|_p < \infty$, $ |\nu''|_p < \infty$, but  $ |\nu|_p = \infty$, where $\nu$ is defined in \eqref{e:diamondnu}.
Consider, for example, $p \in (1, \infty)$,  $d=m=1$, $A_i = 2^i$, and $\lambda_i = c_p  2^{-ip}$ with  $c_p = (1 - 2^{-p})$.
Then $\nu' = \sum_{i=0}^\infty  \lambda_i \delta_{A_i}$ is a probability measure with barycentre
$A = \frac{c_p}{c_{p-1}}$. Set $\nu''_i = \sum_{k=0}^\infty  \lambda_k \delta_{\frac1A 2^{i+k}}$. 
Then $\nu''_i$ is a probability measure with barycentre $A_i = 2^i$. Moreover
$$ \nu = \sum_{i,k= 0}^\infty    \lambda_i \lambda_k \delta_{\frac1A 2^{i+k}} = c_p^2 \sum_{l=0}^\infty  (l+1) 2^{-lp} \delta_{\frac1A 2^l}.
$$
Thus  $|\nu'|_p < \infty$, $ |\nu''|_p < \infty$, but  $ |\nu|_p = \infty$.
\end{remark}

\subsection{Staircase laminate criterion}

We saw in the previous subsections that the condition \emph{$K$ can be reduced to $K'$ in weak $L^p$} (c.f. Definition \ref{d:reduced}) is key to being able to solve the differential inclusion \eqref{e:diffincl1}. The following is a useful criterion for verifying this property, based on the notion of staircase laminates introduced in Section \ref{s:staircaselaminates}.

\begin{theorem}[Staircase laminate criterion]\label{t:SL-criterion}
Let $K,K'\subset\R^{d\times m}$ and $1<p<\infty$, and assume that there exists a constant $M\geq 1$ with the following property: for any $A\in K$ there exists a staircase laminate $\nu^{\infty}_A$ supported on $K'$, with barycenter $A$ and satisfying the bound
\begin{equation}\label{e:SL-bound}
\nu^{\infty}_A(\{X:\,|X|>t\})\leq M^p(1+|A|^p)t^{-p}\quad \textrm{ for all }t>1.
\end{equation}
Then $K$ can be reduced to $K'$ in weak $L^p$. 
\end{theorem}

The proof of Theorem \ref{t:SL-criterion} actually follows from the more general statement of Proposition~\ref{pr:laminates_to_reduction} below, which in turn can be seen as an analogue of Lemma \ref{l:basicconstruction} for unbounded staircase laminates, where the $L^\infty$ bound on the gradient is replaced by a weak $L^p$ bound. 

\begin{proposition}\label{pr:laminates_to_reduction}
Suppose $\nu^\infty$ is a staircase laminate supported on $K$, with barycenter $A$ and satisfying the bound \eqref{e:staircaseLpupper} for some $p>1$. Then, for each $b \in \R^d$,  $\eps \in (0,1)$, 
$\alpha \in (0,1)$, $s \in (1, \infty)$,  and 
each regular domain $\Omega \subset \R^m$,  
there exists a 
piecewise affine map $u \in W^{1,1}(\Omega)
\cap C^\alpha(\overline \Omega)$ with $u = l_{A,b}$
on  $\partial \Omega$ 
and the following properties: 
with 
$\Omega_{error} 
:= \{ x \in \Omega : \nabla u(x) \notin K\}$
we have
\begin{subequations}
\begin{equation}
\int_{\Omega_{error}}(1+|\nabla u|)^s\,dx 
<\eps |\Omega|,     \label{eq:iteration_epserror_staircase}
\end{equation}
and, for each Borel set 
$E \subset \R^{d \times m}$,
\begin{equation}
 (1- \eps)  \nu^\infty(E) \le \frac{|\{ x \in \Omega  : 
\nabla u(x) \in E \}|}{|\Omega|} \le (1+ \eps)  \nu^\infty(E).
\label{eq:bound_by_staircase_laminate2}
\end{equation}
\end{subequations}
\end{proposition}

\medskip

We show first how Theorem \ref{t:SL-criterion} follows easily from Proposition \ref{pr:laminates_to_reduction}.

\begin{proof}[Proof of Theorem \ref{t:SL-criterion}]
The proof follows immediately from Proposition \ref{pr:laminates_to_reduction} below, by taking $E = \{ X \in \R^{d \times m} : |X| > t\}$ and $\eps=1/2$ in \eqref{eq:bound_by_staircase_laminate2} and using \eqref{e:SL-bound}.
\end{proof}

\begin{proof}[Proof of Proposition \ref{pr:laminates_to_reduction}] First note that the assumption $A_N\notin K$ implies $\nu^N(\{A_N\})=\beta_N$. 

Let $s\in (1,\infty)$, $\eta > 0$ and set
$$
c_N=\prod_{j=1}^N(1+2^{-j}\eta).
$$
We construct a sequence of piecewise affine Lipschitz maps $u_N:\Omega\to\R^{d}$ with $u_N=l_{A,b}$ on $\partial\Omega$ such that the following holds: recalling that $\Omega=\mathring{\Omega}\cup\mathcal{N}$ is the decomposition defined in Section \ref{ss:basic} corresponding to the piecewise affine map $u_N$, set
\begin{align*}
\Omega_{error}^{(N)}&:=\left\{x\in \mathring\Omega
:\,\nabla u_N(x)\notin\supp\nu^{(N)}_A\right\}\cup \mathcal{N},\\
\Omega_{inductive}^{(N)}&:=\left\{x\in\mathring\Omega:\,\nabla u_N(x)=A_N\right\}.
\end{align*}
We then have,
for every Borel set $E \subset \R^{d \times m}$,
\begin{subequations}
\begin{align}
&\int_{\Omega^{(N)}_{error}}  1+  |\nabla u_N|^s\,dx 
\le \eta|\Omega|(1-2^{-N})  
\label{e:staircase1-N1_general}\\
& u_N = u_k   \quad \text{in $\Omega\setminus  \Omega_{inductive}^{(k)}$,   for $1 \le k \le N-1$,}  \label{e:staircase1-N3_general}  \\
  c_N^{-1}\nu^N(E)\leq &\frac{|\{ x \in \Omega: \nabla u_N(x) \in E\}|}{|\Omega|}\leq c_N\nu^N(E)\label{e:staircase1-NgeneralE}
\end{align}
\end{subequations}

The existence of $u_1$  satisfying the estimates above follows immediately from applying Lemma \ref{l:basicconstruction} to $\nu^1$,
 and in particular from the estimates \eqref{e:basicconstruction}-\eqref{e:basicconstruction-2} with a suitable choice of $\eps > 0$.

To obtain $u_{N+1}$ from $u_N$ we note that, by construction, there exists a decomposition of $\Omega_{inductive}^{(N)}$ into a disjoint union of (at most) countably many  regular domains: $\Omega_{inductive}^{(N)}=\bigcup_i\Omega_i$, such that $u_N$ is affine on $\Omega_i$ with $\nabla u_N=A_N$. Then, we obtain $u_{N+1}$ by applying Lemma \ref{l:basicconstruction} to $\omega_{N+1}$ 
in each $\Omega_i$, and gluing the 
resulting mapping $v_i$ to $u_N$ as explained 
in Section \ref{ss:basic}, i.e.
\begin{equation}\label{e:staircase1-sets0_general}
u_{N+1}=\begin{cases}u_N&\textrm{ outside }\Omega_{inductive}^{(N)}, \\ v_i&\textrm{ in }\Omega_i.\end{cases}
\end{equation}

In particular, we can achieve that 
\begin{equation} \label{eq:stage3_Calpha_general}
\| u_{N+1} - u_N\|_{C^\alpha} \le 2^{-N} \eta.
\end{equation}
Furthermore, for each Borel set $E \subset \R^{d \times n}$,
\begin{align*}
    &|\{x\in \Omega:\,\nabla u_{N+1}(x)\in E\}|=\\
    \overset{\textrm{\eqref{e:staircase1-sets0_general}}}{=}&|\{x\in \Omega\setminus\Omega_{inductive}^{(N)}:\,\nabla u_{N}(x)\in E\}|+|\{x\in \Omega_{inductive}^{(N)}:\,\nabla u_{N+1}(x)\in E\}|\\
    =&|\{x\in \Omega:\,\nabla u_{N}(x)\in E\setminus\{A_N\}\}|+\sum_i|\{x\in \Omega_i:\,\nabla v_i(x)\in E\}|
\end{align*}
Now, by the inductive assumption \eqref{e:staircase1-NgeneralE}, 
$$
c_N^{-1}\nu^N(E\setminus\{A_N\})\leq \frac{|\{x\in \Omega:\,\nabla u_{N}(x)\in E\setminus\{A_N\}\}|}{|\Omega|}\leq c_N\nu^N(E\setminus\{A_N\}).
$$
Moreover, by applying Lemma \ref{l:basicconstruction} with an appropriate choice of $\eps>0$ in each $\Omega_i$ we may ensure that
$$
(1+2^{-(N+1)}\eta)^{-1}\omega_{N+1}(E)\leq \frac{|\{x\in \Omega_i:\,\nabla v_i(x)\in E\}|}{|\Omega|}\leq 
(1+2^{-(N+1)}\eta)\omega_{N+1}(E).
$$
Recall from \eqref{e:nuNtilde} that $\nu^N=\tilde\nu^N+\beta_N\delta_{A_N}$ and $\nu^{N+1}=\tilde\nu^N+\beta_N\omega_N$. Since $A_N\notin K$ and $\supp\tilde\nu^N\subset K$, we obtain $\nu^N(E\setminus\{A_N\})=\tilde\nu^N(E)$. Thus, we deduce
\begin{align*}
    |\{x\in \Omega:\,&\nabla u_{N+1}(x)\in E\}|\leq c_N\tilde \nu^N(E)+(1+2^{-(N+1)}\eta)\sum_i|\Omega_i|\omega_{N+1}(E)\\
    &= c_N|\Omega|\tilde\nu^{N}(E)+(1+2^{-(N+1)}\eta)|\Omega_{inductive}^{(N)}|\omega_{N+1}(E)\\
    &\overset{\textrm{(*)}}{\leq} c_N|\Omega|\tilde\nu^{N}(E)+(1+2^{-(N+1)}\eta)c_N|\Omega|\gamma_N\omega_{N+1}(E)\\
    &\leq c_{N+1}|\Omega|\tilde\nu^{N+1}(E)\,.
\end{align*}
In inequality (*) we used again the inductive assumption \eqref{e:staircase1-NgeneralE} with $E=\{A_N\}$. The lower bound follows entirely analogously, thus verifying \eqref{e:staircase1-NgeneralE} for $N+1$.

 Regarding \eqref{e:staircase1-N3_general}, we have $u_{N+1} = u_N$ on $\Omega \setminus \Omega^{(N)}_{inductive}$. Since 
 $  \Omega \setminus \Omega^{(k)}_{inductive} \subset   \Omega \setminus \Omega^{(N)}_{inductive}$  for $k < N$,  assertion
  \eqref{e:staircase1-N3_general} follows for $N+1$ from the induction assumption.
  
It follows from \eqref{e:staircase1-sets0_general} that 

\begin{align}
 &\Omega_{error}^{(N)}\subset \Omega_{error}^{(N+1)}\subset \Omega_{error}^{(N)}\cup \Omega^{(N)}_{inductive}.\label{e:staircase1-sets2_general}
\end{align}

Thus to verify \eqref{e:staircase1-N1_general} we estimate, 
using \eqref{e:staircase1-sets2_general},
\begin{align*}
& \int_{\Omega^{(N+1)}_{error}}1+|\nabla u_{N+1}|^s  
\,dx\\
\leq & \int_{\Omega^{(N)}_{error}}1+|\nabla u_N|^s   \,dx+\int_{\Omega^{(N)}_{inductive}\cap \Omega^{(N+1)}_{error}}1+|\nabla u_{N+1}|^s  \,dx\\
\leq  &\eta (1-2^{-N})|\Omega| 
+\sum_{i}\int_{\Omega_i\cap \Omega^{(N+1)}_{error} }1+|\nabla v_i|^s \,dx\\
\overset{(*)}{\leq} & \eta (1-2^{-N})|\Omega| +\sum_{i}\eta 2^{-(N+1)}|\Omega_i|\\
\leq & \eta (1-2^{-(N+1)})|\Omega|,
\end{align*}
where inequality (*) is
a consequence of \eqref{e:basicconstruction-2}
in the application of Lemma \ref{l:basicconstruction},
with a suitable choice of $\eps >0$.
This concludes the proof of properties 
\eqref{e:staircase1-N1_general}--
\eqref{e:staircase1-NgeneralE}.

\smallskip

We now study the limit $N \to \infty$. 
It follows from  \eqref{eq:stage3_Calpha_general} that there exists a $u \in C^\alpha(\overline \Omega)$ such that $u_N \to u$ uniformly.
We next study the distribution function of $\nabla u_N$.

It follows from \eqref{e:staircase1-NgeneralE} and the choice of $c_N$ that, for any Borel set $E\subset\R^{d\times m}$ we have
\begin{equation}\label{e:staircase1-N_generalE2}
    e^{-\eta}\nu^N(E)|\Omega|\leq |\{x\in \Omega:\,\nabla u_N(x)\in E\}|\leq e^\eta\nu^N(E)|\Omega|.
\end{equation}
With $E=\{A_N\}$ we obtain
\begin{equation*}
|\Omega_{inductive}^{(N)}|\leq e^\eta\beta_N|\Omega|\to 0.
\end{equation*}
In addition, from \eqref{e:staircase1-N3_general} it follows that $u = u_N$ almost everywhere on $\Omega \setminus \Omega_{inductive}^{(N)}$. Consequently $\nabla u_N$ converges
to $\nabla u$ in measure. Next, we set $E = \{ X \in \R^{m \times d} : |X| > t\}$ in \eqref{e:staircase1-N_generalE2} and use \eqref{e:staircaseLpupper} to conclude
\begin{equation*}
    |\{x\in \Omega:\,|\nabla u(x)|>t\}|\leq C|A|^pt^{-p}
\end{equation*}
for some $C>1$ and all $t>0$.
Together with convergence of $\nabla u_N$ in measure
we deduce that  $\nabla u_N \to \nabla u$
 in $L^q$ for any $q<p$,  in particular $u\in W^{1,1}(\Omega)$.
Moreover,  
\begin{equation*}
\nabla u(x)\in K \quad \text{if $x \notin \Omega_{error} = \bigcup_{N=1}^\infty \Omega_{error}^{(N)}$.} 
\end{equation*}
 Fatou's lemma and \eqref{e:staircase1-N1_general} imply
\begin{align*}
      \int_{\Omega_{error}} (1 + |\nabla u|)^s \, dx 
  \le  \, 2^{s-1}  \int_{\Omega_{error}} (1 + |\nabla u|^s) \, dx 
  \le 2^{s-1} \eta |\Omega|
  \end{align*}
 Thus, given $\eps > 0$,  we can choose $\eta >0$ 
 sufficiently small so that the 
 condition \eqref{eq:iteration_epserror_staircase} is satisfied, whereas \eqref{e:staircase1-NgeneralE} follows from the passage $N\to\infty$ in \eqref{e:staircase1-N_generalE2} provided that $e^\eta \le 1+\eps$ (in which case also $e^{-\eta} \ge 1-\eps$). 
This concludes the proof of 
Proposition \ref{pr:laminates_to_reduction}.
\end{proof}

\section{Proof of the main results}  \label{s:proofs}

In this section we give the proofs of Theorem \ref{t:main} and Theorem \ref{th:approximate}. These are about the existence of exact and approximate solutions to the differential inclusion
\begin{equation}\label{e:LSigma}
\nabla u(x)\in L\cap\Sigma\textrm{ a.e. }x\in\Omega,
\end{equation}
with appropriate weak $L^2$ bounds and affine boundary conditions. Recall that $L$ is the set of split $2n\times 2n$ matrices and $\Sigma$ is the set of matrices of determinant $=1$, see \eqref{e:split}-\eqref{e:det1}.

\subsection{Breaking the construction into stages and the proof of the main theorems}\hfill

Our strategy to solve the problem \eqref{e:LSigma} consists of three stages, each corresponding to a simpler inclusion problem:

\begin{enumerate}
\item We construct for any $l_{A,b}$ piecewise affine approximate solutions to the problem 
\begin{subequations}
\begin{align*}
\rank (\nabla u(x))&=1\quad\textrm{ for a.e. }x\in\Omega,\\
u(x)&=l_{A,b}(x)\quad\textrm{ for }x\in\partial\Omega.
\end{align*}
\end{subequations}
\item Having reduced to affine pieces with rank one, we construct piecewise affine approximate solutions to 
\begin{subequations}
\begin{align*}
\nabla u(x)&\in L\quad\textrm{ for a.e. }x\in\Omega,\\
u(x)&=l_{A,b}(x)\quad\textrm{ for }x\in\partial\Omega
\end{align*}
\end{subequations}
for any $A$ with $\rank( A)=1$.
\item Finally, we pass from general affine pieces in $L$ to $L\cap\Sigma$ by constructing piecewise affine approximate solutions to 
\begin{subequations}
\begin{align*}
\nabla u(x)&\in L_i\cap \Sigma\quad\textrm{ for a.e. }x\in\Omega,\\
u(x)&=l_{A,b}(x)\quad\textrm{ for }x\in\partial\Omega
\end{align*}
\end{subequations}
for any $A\in L_i$, for both $i=1,2$ separately.
\end{enumerate}

The idea behind this overall strategy is that even though there are no rank-one connections between $L_1\cap\Sigma$ and $L_2\cap\Sigma$, rank-one connections between $L_1$ and $L_2$ do exist for singular matrices non-vanishing matrices. For example we have  $e_1\otimes e_1\in L_1$,  $e_{n+1}\otimes e_1\in L_2$, and $\rank(e_1\otimes e_1-e_{n+1}\otimes e_1)=1$.

The precise statements corresponding to the three stages above are as follows:
\begin{proposition}[Stage 1]\label{p:stage1}
$\R^{2n\times 2n}$ can be reduced to $\{X:\,\rank X\leq 1\}$ in weak $L^2$.
\end{proposition}

Proposition \ref{p:stage1}, together with Theorem \ref{th:generalsolution}, implies the existence of continuous Sobolev mappings with arbitrary affine boundary values, with weak $L^2$ gradients and such that $\rank\nabla u\leq 1$ a.e. This latter statement is already known, see \cite{FaracoMoraCorall2018,LiuMaly2016}.

\begin{proposition}[Stage 2]\label{p:stage2} The set
$\{X:\,\rank X\leq 1\}$ can be reduced to $L$ 
 in weak $L^p$ for any $p<\infty$.
\end{proposition}

\begin{proposition}[Stage 3]\label{p:stage3}
The set $L_1$ can be reduced to $L_1\cap \Sigma$ in weak $L^{2n}$, and $L_2$ can be reduced to $L_2\cap \Sigma$ in weak $L^{2n}$.
\end{proposition}

\begin{proof}[Proof of Theorem \ref{t:main}]
According to Theorem \ref{th:generalsolution}, Theorem \ref{t:main} follows from the statement that $\R^{d\times m}$ can be reduced to $L\cap \Sigma$ in weak $L^2$. In turn, this latter statement is a direct consequence of Propositions \ref{p:stage1}-\ref{p:stage3} together with Theorem \ref{th:iteration}.
\end{proof}

\medskip

\begin{proof}[Proof of Theorem~\ref{th:approximate}] Suppose $\rank A = 1$ and $A \notin L$.
 Let $s_k > 2n$ with $s_k < s_{k+1}$ and $s_k \to \infty$ as $s \to \infty$. 
By combining Stages 2 (Proposition \ref{p:stage2} and 3 (Proposition \ref{p:stage3} via Theorem \ref{th:iteration}, we obtain the statement: \emph{The set $\{X:\,\rank X\leq 1\}$ can be reduced to $L\cap\Sigma$ in weak $L^{2n}$.} Therefore there exist  maps $u^{(k)} : \Omega \to \R^{2n}$ such that 
 $u^{(k)} \in W^{1,1}(\Omega)\cap C^{\alpha}(\overline{\Omega})$ with $u^{(k)}=l_{A}$ on $\partial\Omega$ and
we have
\begin{subequations}
\begin{align}
\int_{ \{\nabla u^{(k)} \notin L \cap \Sigma \}}(1+|\nabla u^{(k)}|^{s_k})\,dx&<2^{-k}|\Omega|,\label{e:stage23-1}\\
|\{x\in\Omega:\,|\nabla u^{(k)}(x)|>t\}|&\leq M^{2n} (1+|A|^{2n})|\Omega|t^{-{2n}}.\label{e:stage23-2}
\end{align}
\end{subequations}
By the rescaling and covering argument we may assume that, in addition,  
$\| u^{(k)} - l_{A}\|_{C^\alpha} < 2^{-k}$. Hence $u^{(k)} \to l_A$ in $C^\alpha(\overline \Omega)$. 
Let $1 \le p < 2n$. Then it follows from \eqref{e:stage23-2} that $\sup_k \|  \nabla u^{(k)}\|_{L^p} < \infty$. 
Thus  $  u^{(k)} \rightharpoonup l_A$ in $W^{1,p}(\Omega)$.  
Since $s_k \to \infty$, it follows from \eqref{e:stage23-1} that  $\int_{ \{\nabla u^{(k)} \notin L \cap \Sigma \}}(1+|\nabla u^{(k)}|^{s})\,dx$
converges to zero, for all $ s\in [1, \infty)$.

Finally, if $ \liminf_{k \to \infty} \dist(\nabla u^{(k)}, L_1) = 0$ in $L^1(\Omega)$,  then arguing along a subsequence for which the
limit inferior is achieved we see that  $A \in L_1$, since $L_1$ is a linear subspace and hence
weakly closed. This contradicts the hypothesis $A \notin L$.  Similarly we see that $\liminf_{k \to \infty}\\
 \dist(\nabla u^{(k)}, L_2) > 0$.
 \end{proof}

\bigskip

\subsection{Stage 1: Proof of Proposition \ref{p:stage1}}\label{ss:stage1}

Recall the sets $\Lambda^{(m)}$ introduced in Section \ref{ss:laminates}, for the case $d=2n$:
$$
\Lambda^{(m)}=\{X\in \R^{2n\times 2n}:\,\rank X\leq m\}.
$$
We start with the following consequence of Theorem \ref{t:SL-criterion} and the construction of staircase laminates in Example \ref{ex:staircase2}:
\begin{corollary}\label{c:staircase2}
   For any $2\leq m\leq 2n$ the set $\Lambda^{(m)}$ can be reduced to $\Lambda^{(m-1)}$ in weak $L^m$.
\end{corollary}

\begin{proof}
Let $A\in \Lambda^{(m)}$. The singular value decomposition of $A$ has the form $RDQ$ with $R,Q\in O(2n)$ orthogonal and $D\in \mathcal{D}\cap\Lambda^{(m)}$ diagonal. Then Example \ref{ex:staircase2} provides a staircase laminate $\nu^{\infty}_{D}$ with barycenter $D$ and weak $L^m$ bound
\begin{equation*}
\nu^{\infty}(\{X:\,|X|>t\})\leq 2^{1+m}|D|^mt^{-m}\textrm{ for all }t>0.
\end{equation*}
Then we apply the invariance principle Lemma \ref{l:invariance} with $T(F)=RFQ$ for $F\in \R^{2n\times 2n}$. Observe that $T(\Lambda^{(m-1)})=\Lambda^{(m-1)}$, so that $\nu^{\infty}_A:=T_*\nu^{\infty}$ is a staircase laminate supported in $\Lambda^{(m-1)}$ with barycenter $RDQ=A$. Moreover, $|A|\leq |R||D||Q|\leq d|D|$ and hence there exists $C=C(m,n)>1$ such that
\begin{equation*}
\nu^{\infty}_A(\{X:\,|X|>t\})\leq C(m,n)|A|^mt^{-m}\textrm{ for all }t>0.
\end{equation*}
The claim follows from Theorem \ref{t:SL-criterion}.
\end{proof}

\begin{proof}[Proof of Proposition \ref{p:stage1}]

Applying Corollary \ref{c:staircase2} and Theorem \ref{th:iteration} inductively, starting with $m=2n$ and then $2n\mapsto 2n-1\mapsto \dots \mapsto 1$, we see that
$\Lambda^{(2n)}=\R^{2n\times 2n}$ can be reduced to $\Lambda^{(1)}$ in weak $L^2$. 
This concludes the proof of Proposition \ref{p:stage1}.

\end{proof}

\subsection{Stage 2: Proof of Proposition \ref{p:stage2}}

The proof of Proposition \ref{p:stage2} immediately follows from Lemma~\ref{l:basicconstruction} together with the following 

\begin{lemma}
Let $A\in \R^{2n\times 2n}$ with $\rank A\leq 1$. Then there exists a discrete laminate $\nu$ with barycenter $\bar{\nu}=A$ and support
$$
\supp\nu\subset L\cap\left\{|X|\leq 2|A|\right\}.
$$ 	
\end{lemma}

\begin{proof}
Let $A=a\otimes b$ for $a,b\in\R^{2n}$ and write 
$$
a=\begin{pmatrix}a_1\\a_2\end{pmatrix}\,,\quad b=\begin{pmatrix}b_1\\b_2\end{pmatrix},\quad A=\begin{pmatrix}a_1\\a_2\end{pmatrix}\otimes \begin{pmatrix}b_1\\b_2\end{pmatrix}.
$$
with $a_1,a_2,b_1,b_2\in\R^n$. 
Define the matrices
\begin{align*}
A_1&=\begin{pmatrix}2a_1\\0\end{pmatrix}\otimes \begin{pmatrix}2b_1\\0\end{pmatrix},\quad 
A_2=\begin{pmatrix}2a_1\\0\end{pmatrix}\otimes \begin{pmatrix}0\\2b_2\end{pmatrix},
\\
A_3&=\begin{pmatrix}0\\2a_2\end{pmatrix}\otimes \begin{pmatrix}2b_1\\0 \end{pmatrix},
\quad
A_4=\begin{pmatrix}0\\2a_2\end{pmatrix}\otimes \begin{pmatrix}0\\2b_2\end{pmatrix},\\
B_1&=\begin{pmatrix}2a_1\\0\end{pmatrix}\otimes \begin{pmatrix}b_1\\b_2\end{pmatrix},
\quad
B_2=\begin{pmatrix}0\\2a_2\end{pmatrix}\otimes \begin{pmatrix}b_1\\b_2\end{pmatrix}.
\end{align*}
Observe that $\textrm{rank}(B_1-B_2)\leq 1$ and $\frac{1}{2}B_1+\frac{1}{2}B_2=A$. Moreover,
$\textrm{rank}(A_1-A_2)\leq 1$, $\textrm{rank}(A_3-A_4)\leq 1$, and $\frac{1}{2}A_1+\frac{1}{2}A_2=B_1$, $\frac{1}{2}A_3+\frac{1}{2}A_4=B_2$. Therefore
$$
\nu=\frac{1}{4}\sum_{i=1}^4\delta_{A_i}
$$
is a laminate with the desired properties.  	
\end{proof}

\subsection{Stage 3: Proof of Proposition \ref{p:stage3}}

Let $\mathcal{D}$ denote the $2n\times 2n$ diagonal matrices, 
\begin{align*}
\mathcal{D}_{\geq 2}&=\{X=\diag(x_1,\dots,x_{2n})\textrm{ with }|x_i|\geq 2\textrm{ for all }i\},\\
\Sigma&=\{X\in \R^{2n\times 2n}:\,\det X=1\}.
\end{align*}
Theorem \ref{t:SL-criterion} and Example \ref{ex:staircase1} leads to the following statement.
\begin{corollary}\label{c:staircase1}
   The set $\mathcal{D}_{\geq 2}$ can be reduced to $\mathcal{D}\cap \Sigma$ in weak $L^{2n}$.
\end{corollary}

\begin{proof}
The statement follows from the existence of the staircase laminates constructed in Example \ref{ex:staircase1} with $d=2n$, together with Theorem \ref{t:SL-criterion}.
\end{proof}

Next, we have the following elementary construction.
\begin{lemma}\label{l:stage3l}
For any $A\in \mathcal{D}$ there exists a finite order laminate $\nu\in\mathcal{L}(\R^{2n\times 2n})$ with barycenter $\bar{\nu}=A$ and support
$$
\supp\nu\subset \mathcal{D}_{\geq 2}\cap\left\{|X|\leq C(1+|A|)\right\},
$$ 	
where $C=C(n)\geq 1$.
\end{lemma}

\begin{proof}
Let us write $A=\diag(a_1,\dots,a_{2n})$ and, without loss of generality, assume that $|a_1|\leq 2$. Then $a_1\in \{-2,2\}^{co}$, more precisely we can write $a_1=\frac{2+a_1}{4}\cdot 2+\frac{2-a_1}{4}\cdot (-2)$ as a convex combination. Correspondingly, since $\rank(A_1-A_2)=1$, the (possibly degenerate) splitting
$$
\delta_A\mapsto \frac{2+a_1}{4}\delta_{A_1}+\frac{2-a_1}{4}\delta_{A_2},
$$
where $A_1=\diag(2,a_2,\dots,a_{2n})$, $A_2=\diag(-2,a_2,\dots,a_{2n})$, defines a simple laminate with barycenter $A$. By repeating the same splitting procedure for the entries $a_2,\dots,a_{2n}$, if necessary, we deduce that there exists a laminate of finite order $\nu\in \mathcal{L}(\R^{2n\times 2n})$ with barycenter $\bar{\nu}=A$ and 
$$
\nu=\sum_{i=1}^{2n}\lambda_i\delta_{A_i}\quad\textrm{ with }A_i\in \mathcal{D}_{\geq 2}\textrm{ for all }i.
$$
Note also that there exists a constant $C=C(n)$ such that, for all $A \in \mathcal D$,
\begin{equation} 
|A_i|  \le C (1 + |A|)   \label{e:split_cube2}
\end{equation}
\end{proof}

\begin{proof}[Proof of Proposition \ref{p:stage3}]
First of all we observe that Lemma \ref{l:basicconstruction}, applied to the laminate in Lemma \ref{l:stage3l}, implies: the set $\mathcal{D}$ can be reduced to $\mathcal{D}_{\geq 2}$ in weak $L^p$ for any $p<\infty$. By choosing $p>2n$ and combining with Corollary \ref{c:staircase1} via Theorem \ref{th:iteration} leads to the statement: 
\begin{center}
    (*) the set $\mathcal{D}$ can be reduced to $\mathcal{D}\cap\Sigma$ in weak $L^{2n}$. 
\end{center}
Next, observe that left and right multiplication by $R,Q\in SO(2n)\cap L_1$ leaves the set $L_1\cap\Sigma$ invariant. Therefore, by the invariance property of laminates (Section \ref{ss:laminates} and Lemma \ref{l:invariance}) and the fact that (*) was the consequence of the existence of two laminates (the staircase laminate in Example \ref{ex:staircase1} and the finite order laminate in Lemma \ref{l:stage3l}), Proposition \ref{p:stage3} follows if we can show that for any $A\in L_1$ there exist $R,Q\in SO(2n)\cap L_1$ and $D\in L_1\cap\mathcal{D}$ such that $A=RDQ^T$. 

To this end write 
$$
A=\begin{pmatrix}A_1&0\\0&A_2\end{pmatrix}
$$
 in $n\times n$ block matrix form. 
The singular value decomposition of $A_i$ has the form $R_iD_iQ_i^T$ with $R_i,Q_i\in O(n)$ orthogonal and $D_i\in \mathcal{D}$ diagonal $n\times n$ matrices. Then we obtain also $A=\tilde R\tilde D\tilde Q^T$, where 
$$
\tilde R=\begin{pmatrix}R_1&0\\0&R_2\end{pmatrix}, \tilde Q=\begin{pmatrix}Q_1&0\\0&Q_2\end{pmatrix}, \tilde D=\begin{pmatrix}D_1&0\\0&D_2\end{pmatrix}.
$$
If $\det R_1\det R_2=\det Q_1\det Q_2=1$, then we simply set $R=\tilde R$, $Q=\tilde Q$ and $D=\tilde D$. If, instead, for instance $\det R=-1$, then set $R=\tilde RJ$ and $D=J\tilde D$, where $J=\diag(-1,1,\dots,1)$. We can deal with the case $\det Q=-1$ similarly.

\end{proof}

\appendix

\section{Further examples based on staircase laminates}

We start with a simple but useful generalization of Proposition \ref{pr:laminates_to_reduction}. 

\begin{remark}\label{d:extendedstaircase}
 The statement of Proposition \ref{pr:laminates_to_reduction} (and consequently that of Theorem \ref{t:SL-criterion}) continues to hold if the requirement that ``$\nu_A^\infty$ is a staircase laminate supported in $K$'' is replaced by requiring that it is a probability measure of the form
   \begin{equation}\label{e:extendedstaircase}
   \nu^\infty=\sum_{j=1}^{J'}\lambda_j\delta_{B_j}+\sum_{j= J'+1}^{J}\lambda_j\tilde\nu^{\infty}_j,
   \end{equation}
   with
   \begin{itemize}
       \item[(a)] $\sum_{j=1}^{J}\lambda_j\delta_{B_j}$ is a laminate of finite order with barycenter $A$;
       \item[(b)] For every $j=1,\dots,J'$ we have $B_j\in K$;
       \item[(c)] For every $j=J'+1,\dots,J$ the probability measure $\tilde\nu^\infty_j$ is a staircase laminate supported in $K$ in the sense of Definition \ref{d:staircase} with barycenter $B_j$.
   \end{itemize}
 The proof of this claim follows immediately from applying Lemma \ref{l:basicconstruction} to the laminate of finite order in (a), followed by $J-J'$ applications of Proposition \ref{pr:laminates_to_reduction}, applied to each $\nu^\infty_j$ in (c).
\end{remark}

\subsection{Very weak solutions to linear elliptic equations with measurable coefficients}

The main result in \cite{astala_faraco_szekelyhidi08} rests on the following result, which is a variant of Theorem 3.18 in \cite{astala_faraco_szekelyhidi08}. 
Recall that
$$
E_\rho = \left\{  
\begin{pmatrix} 
\lambda & 0 \\ 0  & \rho \lambda \end{pmatrix} R :
\lambda \ge 0, \, R \in SO(2)
\right\}.
$$

\begin{theorem}\label{th:AFS}
For any $\mathcal{K}>1$ there exists $M>1$ with the following property. For any $A\in \R^{2\times 2}$, $b\in\R^2$, any $\alpha,\delta\in [0,1)$ and any $\Omega\subset\R^2$ regular domain there exists a piecewise affine mapping $u\in W^{1,1}(\Omega)\cap C^{\alpha}(\overline{\Omega})$ with 
\begin{itemize}
    \item $u=l_{A,b}$ on $\partial\Omega$,
    \item $\|u-l_{A,b}\|_{C^{\alpha}(\overline{\Omega})}<\delta$,
    \item $\nabla u(x)\in E_{\mathcal{K}}\cup E_{1/\mathcal{K}}$ for almost every $x\in \Omega$,
    \item for any $t>   1 + |A|$
    \begin{equation*}
        M^{-1}(1+|A|^{q_{\mathcal{K}}})t^{-q_{\mathcal{K}}}\leq \frac{|\{x\in\Omega:\,|\nabla u(x)|>t\}|}{|\Omega|}\leq M(1+|A|^{q_{\mathcal{K}}})t^{-q_{\mathcal{K}}},
    \end{equation*}
\end{itemize}
where $q_{\mathcal{K}}=\frac{2\mathcal{K}}{\mathcal{K}+1}$.
In particular $u\in W^{1,p}(\Omega)$ for any $p<q_{\mathcal{K}}$ but $\int_{\Omega}|\nabla u|^{q_{\mathcal{K}}}\,dx=\infty$.
\end{theorem}

To prove Theorem \ref{th:AFS} we need the following extension of the construction in Example \ref{ex:staircase3}.
\begin{lemma}\label{l:staircase3-extended}
For any $\mathcal{K}>1$ there exists a constant $M=M(\mathcal{K})>1$ with the following property. For any $A\in \R^{2\times 2}$ there exists a probability measure $\nu^\infty_A$ of the form \eqref{e:extendedstaircase} with barycenter $A$ which is supported on  $E_{\mathcal{K}}\cup E_{1/\mathcal{K}}$ and satisfies the bound
\begin{equation}\label{e:staircase3-extended}
M^{-1}(1+|A|^{q_{\mathcal{K}}})t^{-q_{\mathcal{K}}}\leq \nu^{\infty}_A(\{X:|X|>t\})\leq M(1+|A|^{q_{\mathcal{K}}})t^{-q_{\mathcal{K}}}
\end{equation}
for all $t>1+|A|$.
\end{lemma}

\begin{proof}
We proceed by different levels of generality of the matrix $A$. Observe that Example \ref{ex:staircase3} precisely treats the case $A=\diag(-1,1)$, yielding a staircase laminate $\nu^\infty$.

\smallskip

\noindent{\bf Step 1. }If $A=\diag(-x,x)$ for some $x\geq 2$, we can use the invariance property in Lemma \ref{l:invariance} with $T(A)=xA$. Indeed, it is clear that both the set of rank-one matrices and $E_{\mathcal{K}}\cup E_{1/\mathcal{K}}$ are invariant under $T$. Then $T_*\nu^\infty$, where $\nu^\infty$ is the staircase laminate from Example \ref{ex:staircase3}, has barycenter $T(\diag(-1,1))=\diag(-x,x)$, and we directly obtain \eqref{e:staircase3-extended}.

\smallskip

\noindent{\bf Step 2. }If $A=\diag(x,y)$ with $\max\{|x|,|y|\}\geq 2$, assume without loss of generality $y\geq 2$, $y\geq |x|$
and $y \ne -x$. Consider the laminate
\begin{equation}\label{e:afs-1}
    \tilde\nu=\alpha\delta_{\diag(-y,y)}+(1-\alpha)\delta_{\diag(\mathcal{K}y,y)},
\end{equation}
with $\alpha=\frac{\mathcal{K}-\frac{x}{y}}{\mathcal{K}+1}$. Since $y\geq |x|$ and $y \ne -x$, it follows that  $1 > \alpha\geq \frac{\mathcal{K}-1}{\mathcal{K}+1}$. Then, if $\nu^{\infty}_y$ is the staircase laminate from Step 1 with barycenter $\diag(-y,y)$, the probability measure
\begin{equation*}
    \tilde\nu^{\infty}=\alpha\nu^{\infty}_y+(1-\alpha)\delta_{\diag(\mathcal{K}y,y)}
\end{equation*}
satisfies the conditions of  Remark \ref{d:extendedstaircase} with $J'=1$, $J=2$, and the $\mathcal{K}$-dependent lower bound on $\alpha$ implies that \eqref{e:staircase3-extended} continues to hold with a possibly larger constant $M$ (but only depending on $\mathcal{K}>1$).

\smallskip

\noindent{\bf Step 3. }If $A=\diag(x,y)$ with $\max\{|x|,|y|\}<2$, then we consider the splitting sequence
\begin{align*}
    \delta_{\diag(x,y)}&\mapsto \alpha_1\delta_{\diag(-2,y)}+(1-\alpha_1)\delta_{\diag(2,y)}\\
    &\mapsto\alpha_1(\alpha_2\delta_{\diag(-2,-2)}+(1-\alpha_2)\delta_{\diag(-2,2)})+\\
    &\quad+(1-\alpha_1)(\alpha_2\delta_{\diag(2,-2)}+(1-\alpha_2)\delta_{\diag(2,2)})
\end{align*}
where $\alpha_1=\frac{2-x}{4}$, $\alpha_2=\frac{2-y}{4}$.
The two terms $\delta_{\diag(-2,-2)}$ and $\delta_{\diag(2,2)}$ can now be split further as in \eqref{e:afs-1}. Overall we then obtain a laminate of finite order 
\begin{equation*}
   \mu=\lambda_1\delta_{\diag(2,-2)}+\lambda_2\delta_{\diag(-2,2)}+\lambda_3\delta_{\diag(-2\mathcal{K},-2)}+\lambda_4\delta_{\diag(2\mathcal{K},2)},
\end{equation*}
with $\lambda_1+\lambda_2\geq \frac{\mathcal{K}-1}{\mathcal{K}+1}$.
 supported on the the matrices $\{(\pm 2,\pm 2)\}$ with 
\begin{equation*}
    \mu(\{\diag(2,-2),\diag(-2,2)\})\geq \frac{\mathcal{K}-1}{\mathcal{K}+1}.
\end{equation*}
Then, with $\nu^\infty_{\pm 2}$ being the staircase laminates from Step 1, the measure 
\begin{equation*}
   \tilde\nu^\infty=\lambda_1\nu^\infty_{-2}+\lambda_2\nu^\infty_2+\lambda_3\delta_{\diag(-2\mathcal{K},-2)}+\lambda_4\delta_{\diag(2\mathcal{K},2)}
\end{equation*}
is of the form \eqref{e:extendedstaircase} with $J'=2$, $J=4$, and the estimate \eqref{e:staircase3-extended} again holds with a $\mathcal{K}$-dependent constant $M$.

\smallskip

\noindent{\bf Step 4. } 
Next, we may use the invariance property (Lemma \ref{l:invariance}) with $T(A)=AR$ for $R\in SO(2)$ together with the invariance of the set $E_{\mathcal{K}}\cup E_{1/\mathcal{K}}$ under $T$ to show that the statement of the Lemma holds for matrices of the form $A=DR$, where $D$ is diagonal and $R\in SO(2)$. In particular this is the case for any conformal or anti-conformal matrix (i.e. matrices of the form $\begin{pmatrix}x&y\\-y&x\end{pmatrix}$ or $\begin{pmatrix}x&y\\y&-x\end{pmatrix}$). 



\noindent{\bf Step 5}

More generally, any $A\in\R^{2\times 2}$ can be decomposed\footnote{For the well-known connection relating this decomposition to equations of the form \eqref{e:elliptic-eq} and quasiconformal mappings of the plane we refer to \cite{astala_faraco_szekelyhidi08}} as $A=A_++A_-$, with $A_+$ conformal and $A_-$ anti-conformal. Assuming that $A_+,A_-\neq 0$, we can write $A=\lambda B+(1-\lambda)C$, with 
\begin{equation*}
B=\frac{|A_+|+|A_-|}{|A_+|}A_+,\quad C=\frac{|A_+|+|A_-|}{|A_-|}A_-,\quad \lambda=\frac{|A_+|}{|A_+|+|A_-|}.
\end{equation*}
Since $\det(B-C)=0$, the measure $\lambda\delta_B+(1-\lambda)\delta_C$ is a laminate with barycenter $A$. We can then further split this measure using Step 4 and thus obtain a measure of the form \ref{e:extendedstaircase}.
\end{proof}

\gray{LS:\, commented out; here gray for comparison:

We start again with a simple lemma on the existence of certain finite order laminates. 

\begin{lemma}
Let $A=\diag(x,y)$ be a $2\times 2$ diagonal matrix. There exists a finite order laminate $\nu\in \mathcal{L}(\R^{2\times 2})$ supported on diagonal matrices with barycenter $\bar{\nu}=A$ and such that 
\begin{equation*}
    \supp\nu\subset E_{\mathcal{K}}\cup E_{1/\mathcal{K}}\cup \{\diag(-t,t):\,t\geq 2\}
\end{equation*}
Moreover
\begin{equation*}
    \nu(\{\diag(-t,t):\,t\geq 2\})\geq \frac{\mathcal{K}-1}{\mathcal{K}+1}.
\end{equation*}
\end{lemma}

\begin{proof}
First assume that $\max\{|x|,|y|\}\geq 2$. Without loss of generality we may then further assume $y\geq |x|$ and $y\geq 2$. Then we may simply set
\begin{equation}\label{e:afs-1_old}
    \nu=\alpha\delta_{\diag(-y,y)}+(1-\alpha)\delta_{\diag(\mathcal{K}y,y)},
\end{equation}
with $\alpha=\frac{\mathcal{K}-\frac{x}{y}}{\mathcal{K}+1}$. Since $y\geq |x|$, it follows that $\alpha\geq \frac{\mathcal{K}-1}{\mathcal{K}+1}$ as required.

Now consider the case $\max\{|x|,|y|\}<2$. In this case we consider the splitting sequence
\begin{align*}
    &\delta_{\diag(x,y)}\mapsto \alpha_1\delta_{\diag(-2,y)}+(1-\alpha_1)\delta_{\diag(2,y)}\mapsto\\
    &\alpha_1(\alpha_2\delta_{\diag(-2,-2)}+(1-\alpha_2)\delta_{\diag(-2,2)})+(1-\alpha_1)(\alpha_2\delta_{\diag(2,-2)}+(1-\alpha_2)\delta_{\diag(2,2)})
\end{align*}
where $\alpha_1=\frac{2-x}{4}$, $\alpha_2=\frac{2-y}{4}$. The two terms $\delta_{\diag(-2,-2)}$ and $\delta_{\diag(2,2)}$ can now be split further as in \eqref{e:afs-1}. Overall we then obtain a laminate as required, with 
\begin{equation*}
    \nu(\{\diag(2,-2),\diag(-2,2)\})\geq \frac{\mathcal{K}-1}{\mathcal{K}+1}.
\end{equation*}
\end{proof}

Next, we may apply Example \ref{ex:staircase3} together with the scaling transformation $A\mapsto tA$ for any $t\neq 0$ and the invariance property Lemma \ref{l:invariance} to obtain, for any diagonal $2\times 2$ matrix $A$ a staircase laminate $\nu^{\infty}_A$ supported on $E_{\mathcal{K}}\cup E_{1/\mathcal{K}}$ with
\begin{equation*}
    \tilde C^{-1}|A|^{q_{\mathcal{K}}}t^{-q_{\mathcal{K}}}\leq \nu^{\infty}_{A}(\{X:\,|X|>t\})\leq \tilde C|A|^{q_{\mathcal{K}}}t^{-q_{\mathcal{K}}}
\end{equation*}
for all $t>|A|$.
Combined with the invariance of $E_{\mathcal{K}}\cup E_{1/\mathcal{K}}$ under the transformation $A\mapsto AR$ for any $R\in O(n)$ leads to the same, for any $A\in \R^{2\times 2}$.

}

It follows from Theorem \ref{t:SL-criterion} that $\R^{2\times 2}$ can be reduced to $E_{\mathcal{K}}\cup E_{1/\mathcal{K}}$ in weak $L^{q_{\mathcal{K}}}$. In turn, Theorem \ref{th:generalsolution} then almost implies the statement of Theorem \ref{th:AFS}, except the lower bound. However, the proof of Theorem \ref{th:generalsolution} can easily be modified to yield the following statement:

\begin{theorem}\label{th:generalsolution1}
Let $K\subset\R^{d\times m}$,  $1<p\leq r<\infty$  and 
$0 \le \rho \le r$ be such that, for some $M>1$ the following holds. For any $A\in\R^{d\times m}$ there exists a probability measure $\nu_A^\infty$ of the form \ref{e:extendedstaircase} with barycenter $A$ which is supported on $K$ and satisfies the bound
\begin{equation*}
    M^{-p}  (1+|A|^\rho) t^{-p}\leq \nu_A^\infty(\{X:|X|>t\})\leq M^p(1+|A|^r)t^{-p}
\end{equation*}
for all $t >  1 + |A|$. 
Then for any regular domain $\Omega\subset\R^m$, any $A\in\R^{d\times m}$, $b\in \R^d$ and any $\delta>0$, $\alpha\in[0,1)$ there exists a piecewise affine map $u\in W^{1,1}(\Omega)\cap C^{\alpha}(\overline{\Omega})$ with $u=l_{A,b}$ on $\partial\Omega$ such that 
\begin{align}     
\nabla u\in K\textrm{ a.e. in }&\Omega,   \label{eq:weakell1}\\
\|u-l_{A,b}\|_{C^{\alpha}(\overline{\Omega})}&<\delta
 \label{eq:weakell2}
\end{align}
and for all $t>|A|$
\begin{equation}   \label{eq:weakell3}
        \frac{1}{2}M^{-p}(1+  |A|^\rho)t^{-p}\leq \frac{|\{x\in\Omega:\,|\nabla u(x)|>t\}|}{|\Omega|}\leq 2M^p(1+|A|^r)t^{-p}.
    \end{equation}
\end{theorem}

\begin{proof}
The proof is precisely the proof of Theorem \ref{th:generalsolution} and Remark \ref{re:generalsolution_nonhomogeneous_bound} with the following additional observation. As a result of Proposition \ref{pr:laminates_to_reduction}, in the proof of Theorem \ref{th:generalsolution} the first approximation $u_1$ satisfies in addition the lower bound 
\begin{equation*}
        \frac{|\{x\in\Omega:\,|\nabla u_1(x)|>t\}|}{|\Omega|}\geq \frac{2}{3}M^p(1+|A|^\rho)t^{-p}
    \end{equation*}
for any $t> 1+ |A|$. Furthermore, using estimate \eqref{eq:iteration_epserror_staircase} in Proposition \ref{pr:laminates_to_reduction} with a sufficiently small $\varepsilon>0$, we can ensure
\begin{equation*}
 \frac{|\{x\in\Omega_{error}^{(1)}:\,|\nabla u_1(x)|>t\}|}{|\Omega|}\leq \frac{1}{6}M^{-p} t^{-p},
\end{equation*} 
therefore we obtain
\begin{equation*}
        \frac{|\{x\in\Omega\setminus\Omega_{error}^{(1)}:\,|\nabla u_1(x)|>t\}|}{|\Omega|}\geq \frac{1}{2}M^p(1+|A|^\rho)t^{-p}.
    \end{equation*}
Now, since for any $k\geq 2$ the subsequent approximations $u_k$ satisfy $\nabla u_k=\nabla u_1$ outside $\Omega_{error}^{(1)}$, it follows that the limit $u=\lim_{k\to\infty}u_k$ satisfies
\begin{equation*}
        \frac{|\{x\in\Omega:\,|\nabla u(x)|>t\}|}{|\Omega|}\geq \frac{1}{2}M^p(1+|A|^\rho)t^{-p}.
\end{equation*}
This proves the additional lower bound, as stated in Theorem \ref{th:generalsolution1}.
\end{proof}

By standard arguments one also has the following extension of Theorem \ref{th:generalsolution1}:
\begin{corollary}\label{cor:everywhere}
   Under the conditions of Theorem \ref{th:generalsolution1}
   there exists $u \in W^{1,1}(\Omega) \cap C^{\alpha}(\overline \Omega)$ such that \eqref{eq:weakell1} and \eqref{eq:weakell2} hold and, in addition, 
   \begin{equation}\label{e:integrabilityeverywhere}
       \int_{B}|\nabla u|^p\,dx=\infty
   \end{equation}
   for all balls $B\subset \Omega$.
\end{corollary}
We include the proof for the convenience of the reader (such use of the Baire category theorem has appeared e.g.~in \cite{kirchheim03,astala_faraco_szekelyhidi08}).
\begin{proof}
Let 
$$
X=\left\{u\in C^{\alpha}(\overline{\Omega}):\,u=l_{A,b}\textrm{ on }\partial\Omega,\,\|u-l_{A,b}\|_{C^{\alpha}(\overline{\Omega})}\leq \delta\right\}.
$$
Equipped with the $C^0$ topology,  $X$ is a complete metric space. Further, 
for any ball $B\subset\Omega$ and any $R>R_B:=|A|^p|B|$ let
$$
X_{B,R}=\left\{u\in X :   u|_B \in  W^{1,p}(B), \, \,\int_{B}|\nabla u|^p\,dx\leq R\right\}.
$$
By   weak lower-semicontinuity of the $L^p$ norm it follows that $X_{B,R}$ is a closed subset of $X$. 

Moreover, using Theorem \ref{th:generalsolution1} one can easily show that $X_{B,R}$ has empty interior. Indeed, for any $u\in X_{B,R}$ let $v=l_{A,b}+\lambda (u-l_{A,b})$ for some $\lambda\in (0,1)$. Since $R>R_B$, by the triangle inequality we obtain $\|\nabla v\|_{L^p(B)}\leq (1-\lambda)\|A\|_{L^p(B)}+\lambda\|\nabla u\|_{L^p(B)}\leq (1-\lambda)R_B+\lambda R<R$ and $\|v-l_{A,b}\|_{C^{\alpha}(\overline{\Omega})}<\delta$. After choosing $1-\lambda$ sufficiently small and then approximating $v$ uniformly by a piecewise affine mapping, for any $\varepsilon$ we can obtain a piecewise affine $u_1\in X_{B,R}$ with $\|u-u_1\|_{C^0(\overline{\Omega})}<\varepsilon$. In particular there exists a nonempty open subset $\tilde\Omega\subset B$ where $u_1$ is affine. Then, we can apply Theorem \ref{th:generalsolution1} to $u_1|_{\tilde\Omega}$ and obtain $u_2\in X$ with $\|u_2-u_1\|_{C^0(\overline{\Omega})}<\varepsilon$ and $\int_{\tilde\Omega}|\nabla u_2|^p\,dx=\infty$. This shows that $X_{B,R}$ has empty interior.

We now apply the Baire category theorem to conclude that $Y:=\bigcup_{B\subset\Omega}\bigcup_{R>R_B}X_{B,R}$ 
is meager\footnote{It suffices to take the union over all balls $B\subset\Omega$ with rational center and rational radius, and the union over all rational $R>R_B$}, i.e. $X\setminus Y$ is dense, in particular nonempty. On the other hand $X\setminus Y$ consists of all $u\in X$ which satisfy the conclusions of Theorem \ref{th:generalsolution1} (except possibly not piecewise affine) and in addition also \ref{e:integrabilityeverywhere} for all $B\subset\Omega$.
    \end{proof}

\subsection{Weak solutions of the $p$-Laplace equation}
In a recent paper Colombo and 
Tione \cite{ColomboTione2022} 
solved a longstanding 
question  by Iwaniec and Sbordone  
about uniqueness and higher regularity of 
low regularity solutions of the $p$-Laplace equation.

\begin{theorem}[\cite{ColomboTione2022}]  \label{th:plaplace_very_weak} Let $B \subset \R^2$  an open disc. 
Let $p \in (1, \infty) \setminus \{2\}$. Then 
there exists 
$q \in( \max(1, p-1), p)$
and a continuous  solution  $v \in W^{1,q}(B) \cap C(\overline B)$  of the $p$-Laplace equation
\begin{equation}  \label{eq:plaplace}
    \mathop{\mathrm{div}} |\nabla v|^{p-2} \nabla v = 0  \quad \text{in $B$}
\end{equation}
with affine boundary conditions such that
\begin{equation}  \label{eq:nowhere_Lp}
    \int_{B'} |\nabla v|^p \, dx = \infty \quad \text{for each disc $B' \subset B$}
\end{equation}
\end{theorem}
In the above,  \eqref{eq:plaplace} is understood in the sense of distributions.

Our aim here is to show that Theorem \ref{th:plaplace_very_weak}
 easily follows from our general results about the passage from staircase laminates to approximate and exact solutions,
 once one has the key insight in \cite{ColomboTione2022}, namely the construction of a staircase laminate with the right integrability properties. 
Actually  in \cite{ColomboTione2022} a slightly 
sharper version of the statement is shown:  
one can achieve in addition that $\frac34 < \partial_2 v < \frac54$. Our approach does not give this extra condition, but leads to a shorter proof of the theorem as stated.
 
\medskip

First of all, arguing analogously to the case of elliptic equations with measurable coefficients (c.f.~\ref{e:elliptic-eq} and \eqref{e:k-elliptic} in the introduction of Section \ref{sec:staircase}), we see that 
\eqref{eq:plaplace} is equivalent to
the first order differential inclusion
\begin{equation}  \label{eq:inclusion_plaplace}
    \nabla u \in K_p \quad \text{almost everywhere}
\end{equation}
where 
\begin{equation}  \label{eq:K_plaplace}
    K_p = \left\{ \begin{pmatrix} \lambda & 0 \\0  &\lambda^{p-1}  \end{pmatrix} R : \lambda \ge 0, R \in SO(2) \right\},
\end{equation}
We seek a solution of $u = \binom{v}{w}$ of 
 \eqref{eq:K_plaplace} such that $v \in W^{1,q}(B)$ 
 for some $q \in (\max(1,p-1),p)$ and 
 $\int_{B'} |\nabla v|^p \, dx
 =\infty$ for every open disc $B' \subset B$.
 
 In the following we will focus on the case
\begin{equation} \label{eq:plaplace_p_less_2}
    p \in (1,2).
\end{equation}

By the following duality argument this is no loss of generality.
Let $p' = \frac{p}{p-1}$ denote the dual exponent of $p$.
Then $p'-1 = \frac{1}{p-1}$.  Setting $\mu = \lambda^{p-1}$ one
easily sees that  $u = \binom{v}{w}$ satisfies $\nabla u \in K_p$ a.e.\ if and only if
$u' = \binom{w}{v}$ satisfies $\nabla u' \in K_{p'}$ a.e.
Moreover $|\nabla w|^{p'-1} = |\nabla v|$. Thus $\nabla v$ 
in $L^q$  with $q \in (\max(1,p-1),p)$ 
if and only if $\nabla w \in L^s$ 
with $s \in (\max(1,p'-1),p')$ 
and $|\nabla w|^{p'} = |\nabla v|^p$. 

For $p \in (1,2)$ solutions $u = \binom{v}{w}$ of 
\eqref{eq:inclusion_plaplace} satisfy 
$|\nabla w| = |\nabla v|^{p-1}$ and hence,
by Young's inequality,
\begin{equation} \label{eq:nablav_equiv_nablau}
    |\nabla v|^2 \le |\nabla u|^2 \le 1 + 2  |\nabla v|^2
\end{equation} 
Thus in the following we can focus on the integrability
properties of $\nabla u$. The key result is the following
\begin{theorem}
For any $1<p<2$ there exists $\bar{q}_p\in (1,p)$ and $M>1$ with the following property. For any $A\in \R^{2\times 2}$ and $\alpha,\delta\in (0,1)$ and any  regular domain $\Omega\subset\R^2$ there exists a piecewise affine mapping $u\in W^{1,1}(\Omega)\cap C^{\alpha}(\overline{\Omega})$ with 
\begin{itemize}
    \item $u(x)=Ax$ on $\partial\Omega$,
    \item $\|u-Ax\|_{C^{\alpha}(\overline{\Omega})}<\delta$,
    \item $\nabla u(x)\in K_p$ for almost every $x\in \Omega$,
    \item for any $t>  1 + |A|$  
    
    \begin{equation*}
        M^{-1}(1+|A|^{\frac{\bar{q}_p}{p-1}})t^{-q_{p}}\leq \frac{|\{x\in\Omega:\,|\nabla u(x)|>t\}|}{|\Omega|}\leq M(1+|A|^{\frac{\bar{q}_p}{p-1}})t^{-q_{p}}.
    \end{equation*}
\end{itemize}
In particular $u\in W^{1,q}(\Omega)$ for any $q<\bar{q}_p$ but $\int_{\Omega}|\nabla u|^{\bar{q}_p}\,dx=\infty$.
\end{theorem}

 The proof follows from Theorem \ref{th:generalsolution1}, 
 provided we can show the existence of certain laminates with the right integrability properties. This is based on Example \ref{ex:staircase4} and Lemma \ref{l:staircase4-extended} below.  
 First of all, recall from from \eqref{eq:define_barq} in  Example \ref{ex:staircase4} the function 
 $$
 \bar{q}(p,b)=\frac{p-1}{b^{p-1}+1}+\frac{b}{b+1}.
 $$
 For definiteness, for any $p\in(1,2)$ set $\bar{q}_p=\max_{b>1}\bar{q}(p,b)$ and  denote by $\bar{b} = \bar b(p) >1$ a  value of $b$ for which the maximum is achieved (note that $\bar{q}(p,1)=p/2<1$ and $\bar{q}(p,b)\to 1$ as $b\to\infty$, so that, by the argument in Example \ref{ex:staircase4}, $\bar{q}_p\in (1,p)$ and $\bar{b}\gg 1$ exists and is finite, for any $p\in (1,2)$.

\begin{lemma}\label{l:staircase4-extended}
For any $1<p<2$ there exists a constant $M=M(p)>1$ with the following property. For any $A\in \R^{2\times 2}$ there exists a probability measure $\nu^\infty_A$ of the form \eqref{e:extendedstaircase} with barycenter $A$ which is supported on  $K_p$ and satisfies the bound
\begin{equation}\label{e:staircase4-extended}
M^{-1}(1+|A|^{(p-1)\bar{q}_p)}) \,  t^{-q_{p}}\leq \nu^{\infty}_A(\{X:|X|>t\})\leq M(1+|A|^{\frac{\bar{q}_p}{p-1}}) \, t^{-q_{p}}
\end{equation}
for all $t>1+|A|$.
\end{lemma}

\begin{proof}

As in Lemma \ref{l:staircase3-extended} we proceed by different levels of generality of the matrix $A$. We start by noting that Example \ref{ex:staircase4} with $b=\bar{b}(p)$ treats the case $A=\diag(\bar{b},-1)$, yielding a staircase laminate $\nu^\infty_1$. 

\smallskip

\noindent{\bf Step 1. }If $A = \diag(-\bar b, 1)$ we use the invariance property (Lemma \ref{l:invariance}) with $T(X) = - X$. 
The linear map $T$ clearly preserves rank-one lines and also the set $K_p$. If $\nu^\infty_1$ is the staircase laminate from Example \ref{ex:staircase4}, then $\nu_{-1}^\infty:=T_*\nu^\infty_1$ is a staircase laminate supported in $K_p$ with barycenter $T(\diag(b,-1))=A$. Moreover $$T_*\nu^\infty_1(\{X:|X|>t\})=\nu^\infty_1(\{X:|TX|>t\}).$$
Since $|T(X)| = |X|$ we get
\begin{equation}  \label{eq:nu_infty_minus}
 \nu_{-1}^\infty(\{X:|X|>t\})=\nu^\infty_1(\{X:|X|>t\}).   
\end{equation}
Thus the estimate \eqref{e:staircase4-extended} follows from  the estimate \eqref{e:staircase4-weak} for $\nu_1^\infty$ in Example \ref{ex:staircase4}.

\smallskip

\noindent{\bf Step 2. }If $A = \diag(x,y)$ with $\max(|x|, |y|) \le \frac12$ we consider the splitting sequence
\begin{align*}
    \delta_{\diag(x,y)}&\mapsto \alpha_1\delta_{\diag(x,-1)}+(1-\alpha_1)\delta_{\diag(x,1)}\\
    &\mapsto\alpha_1(\alpha_2\delta_{\diag(-1,-1)}+(1-\alpha_2)\delta_{\diag(\bar b,-1)})\\
   &\quad+(1-\alpha_1)(\alpha_3\delta_{\diag(-\bar{b},1)}+
  (1-\alpha_3)\delta_{\diag(1,1)}),
\end{align*}
where $\alpha_1=\frac{1-y}{2}$, $\alpha_2=\frac{\bar{b}-x}{\bar b +1}$, $\alpha_3=\frac{1-x}{\bar b +1}$. 
The two terms $\delta_{\diag(\bar{b},1)}$ and $\delta_{\diag(-\bar{b},1)}$ can now be split further using
Example \ref{ex:staircase4} and Step 1. We finally  obtain the probability measure 
\begin{equation*}
   \tilde\nu^\infty=\lambda_1\nu^\infty_{1}+\lambda_2\nu^\infty_{-1}+\lambda_3\delta_{\diag(-1,-1)}+\lambda_4\delta_{\diag(1,1)},
\end{equation*}
with $\lambda_1 = \alpha_1 (1-\alpha_2)$ and $\lambda_2 = (1-\alpha_1) \alpha_3$. This measure
 is of the form \eqref{e:extendedstaircase} with $J'=2$, $J=4$. 
 Since $\min(\alpha_1, 1-\alpha_1) \ge \frac14$ and  $(1-\alpha_2) =\alpha_3\geq \frac{1}{2(\bar{b}+1)}$, the estimate \eqref{e:staircase4-extended} again holds with a $\bar{b}$-dependent constant $M$.
 
 \smallskip

\noindent{\bf Step 3. }If $A=\diag(x,y)$ with $\max(|x|, |y|) > \frac12$ we set 

$$ \lambda = \max( 2|x|, (2|y|)^{\frac{1}{p-1}}),
\quad (\bar x, \bar y) = \diag(\lambda^{-1} x, \lambda^{-(p-1)} y).$$
Then $\lambda > 1$ and $\max(|\bar x|, |\bar y|)= \frac12$. 
Thus by Step 2 there exists a measure  $\nu_{\bar x, \bar y}^\infty$ of the form \eqref{e:extendedstaircase} with barycentre $\diag(\bar x, \bar y)$ which satisfies

\begin{equation}  \label{eq:barxy_extended4}
 M^{-1} t^{-\bar q_{p}}\leq \nu^{\infty}_{\bar x, \bar y}(\{X:|X|>t\})\leq M t^{-\bar q_{p}}  
\end{equation}
for all $t \ge 1$ (the upper bound trivially holds also for $t < 1$).
We now  use the invariance property (Lemma \ref{l:invariance}) with $T(X) = \diag(\lambda, \lambda^{p-1}) X$. 
The linear map $T$ clearly preserves rank-one lines and also the set $K_p$. By  the invariance property, the  pushforward measure $\nu_A^\infty := T_*\nu^\infty_{\bar x,\bar y}$ is of the form \eqref{e:extendedstaircase} and it has barycentre 
$\diag(\lambda, \lambda^{p-1}) \diag(\bar x, \bar y)  = A$.
Moreover 
$$T_*\nu^\infty_{\bar x, \bar y}(\{X:|X|>t\})=\nu^\infty_{\bar x, \bar y}(\{X:|TX|>t\}).$$
Since $\lambda > 1$ we have
$$ \lambda^{p-1} |X| \le |TX| \le \lambda |X|.$$
Hence  \eqref{eq:barxy_extended4} implies that, for all $t \ge \lambda^{p-1}$, 
$$ M^{-1} \lambda^{(p-1) \bar q_p} t^{-q_p} \le  \nu^\infty_A( X: |X| > t) \le  M \lambda^{\bar q_p} t^{-q_p} $$
Since $\frac12 \lambda^{p-1} \le |A| \le \lambda$
we get, for all $t \ge 2 |A|$, 
$$ M^{-1} |A|^{(p-1) \bar q_p} t^{- \bar q_p} 
\le  \nu^\infty_A( X: |X| > t) \le 
M |A|^{\frac{\bar q_p}{p-1}} 
t^{-\bar q_p}.
$$
Thus we get the desired estimate
\eqref{e:staircase4-extended}.

\noindent{\bf Step 4. }For a general $A\in\R^{2\times 2}$ we use once more the invariance property (Lemma \ref{l:invariance}) with $T(A)=AR$ for $R\in SO(2)$, together with the invariance of the set $K_p$ under $T$, and argue as in Steps 4 and 5 of the proof of Lemma \ref{l:staircase3-extended}. This concludes the proof of Lemma \ref{l:staircase4-extended}.
\end{proof}

\gray{
\bigskip
for any $|\lambda|\geq 1$, we use the invariance property (Lemma \ref{l:invariance}) with $$
T(X) = \begin{pmatrix}\lambda
&0\\0&|\lambda|^{p-2} \lambda \end{pmatrix}.
$$
 The linear map $T$ clearly preserves rank-one lines and also the set $K_p$. If $\nu^\infty_1$ is the staircase laminate from Example \ref{ex:staircase4}, then $\nu_{\lambda}^\infty:=T_*\nu^\infty_1$ is a staircase laminate supported in $K_p$ with barycenter $T(\diag(b,-1))=A$. Moreover, $T_*\nu^\infty_1(\{X:|X|>t\})=\nu^\infty_1(\{X:|TX|>t\})$. Now observe that, since $1<p<2$ and $|\lambda| \ge 1$,  we have, for any $X\in K_p\cap\mathcal{D}$,  
$$
\frac{1}{\sqrt{2}}|\lambda| |X|\leq |T(X)|\leq |\lambda| |X|
$$
and therefore, for any $t\geq |\lambda|^{p'-1}$
\begin{align*}
\nu^\infty_1(\{X:|X|>\sqrt{2}\lambda^{-(p'-1)}t\})&\leq T_*\nu^\infty_1(\{X:|X|>t\})\\
&\leq \nu^\infty_1(\{X:|X|>\lambda^{-(p'-1)}t\}).
\end{align*}
From this and the estimate $|A|\sim \lambda^{p'-1}$ we deduce, using \eqref{e:staircase4-weak}, the estimate
\begin{equation*}
   C^{-1}|A|^{\frac{\bar{q}_p}{p-1}}t^{-\bar{q}_p} \leq \nu_\lambda^\infty(\{X:|X|>t\})\leq C|A|^{\frac{\bar{q}_p}{p-1}}t^{-\bar{q}_p}
\end{equation*}
for any $t\geq |A|$.

\smallskip

\noindent{\bf Step 2. }If $A=\diag(x,y)$, let
and $\bar{y}=\bar{b}^{p-1}\bar{x}^{p-1}$ and consider the splitting sequence 
\begin{align*}
    \delta_{\diag(x,y)}&\mapsto \alpha_1\delta_{\diag(x,\bar{y})}+(1-\alpha_1)\delta_{\diag(x,-\bar{y})}\\
    &\mapsto\alpha_1(\alpha_2\delta_{\diag(-\bar{x},\bar{y})}+(1-\alpha_2)\delta_{\diag(\bar{b}\bar{x},\bar{y})})+\\
    &\quad+(1-\alpha_1)(\alpha_3\delta_{\diag(\bar{x},-\bar{y})}+(1-\alpha_3)\delta_{\diag(-\bar{b}\bar{x},-\bar{y})}),
\end{align*}
where $\alpha_1=\frac{y+\bar{y}}{2\bar{y}}$, $\alpha_2=\frac{\bar{b}-\tfrac{x}{\bar{x}}}{\bar{b}+1}$, $\alpha_3=\frac{\bar{b}+\tfrac{x}{\bar{x}}}{\bar{b}+1}$. 
The two terms $\delta_{\diag(-\bar{x},\bar{y})}$ and $\delta_{\diag(\bar{x},-\bar{y})}$ can now be split further using Step 1. Overall we then obtain the probability measure 
\begin{equation*}
   \tilde\nu^\infty=\lambda_1\nu^\infty_{-\bar{x}}+\lambda_2\nu^\infty_{\bar{x}}+\lambda_3\delta_{\diag(-\bar{b}\bar{x},-(\bar{b}\bar{x})^{p-1})}+\lambda_4\delta_{\diag(\bar{b}\bar{x},(\bar{b}\bar{x})^{p-1})},
\end{equation*}
which is of the form \eqref{e:extendedstaircase} with $J'=2$, $J=4$. Since $\alpha_2,\alpha_3\geq \frac{\bar{b}-1}{\bar{b}+1}$, the estimate \eqref{e:staircase4-extended} again holds with a $\bar{b}$-dependent constant $M$.

End new version Steps 1 and 2

\noindent{\bf Step 1. }If $A=\diag(\bar{b}\lambda,-|\lambda|^{p-2}\lambda)$ for any $|\lambda|\geq 1$, we use the invariance property (Lemma \ref{l:invariance}) with $$
T(X)=\begin{pmatrix}\lambda|\lambda|^{p'-2}&0\\0&\lambda\end{pmatrix}X,
$$
where $p'=\frac{p-1}{p}>2$ is the H\"older dual of $p$. The linear map $T$ clearly preserves rank-one lines and also the set $K_p$. If $\nu^\infty_1$ is the staircase laminate from Example \ref{ex:staircase4}, then $\nu_{\lambda}^\infty:=T_*\nu^\infty_1$ is a staircase laminate supported in $K_p$ with barycenter $T(\diag(b,-1))=A$. Moreover, $T_*\nu^\infty_1(\{X:|X|>t\})=\nu^\infty_1(\{X:|TX|>t\})$. Now observe that, since $1<p<2$, for any $X\in K_p\cap\mathcal{D}$ we have 
$$
\frac{1}{\sqrt{2}}|\lambda|^{p'-1}|X|\leq |T(X)|\leq |\lambda|^{p'-1}|X|
$$
and therefore, for any $t\geq |\lambda|^{p'-1}$
\begin{align*}
\nu^\infty_1(\{X:|X|>\sqrt{2}\lambda^{-(p'-1)}t\})&\leq T_*\nu^\infty_1(\{X:|X|>t\})\\
&\leq \nu^\infty_1(\{X:|X|>\lambda^{-(p'-1)}t\}).
\end{align*}
From this and the estimate $|A|\sim \lambda^{p'-1}$ we deduce, using \eqref{e:staircase4-weak}, the estimate
\begin{equation*}
   C^{-1}|A|^{\frac{\bar{q}_p}{p-1}}t^{-\bar{q}_p} \leq \nu_\lambda^\infty(\{X:|X|>t\})\leq C|A|^{\frac{\bar{q}_p}{p-1}}t^{-\bar{q}_p}
\end{equation*}
for any $t\geq |A|$.

\smallskip

\noindent{\bf Step 2. }If $A=\diag(x,y)$, let $\bar{x}=\max\{1,|x|\}$ and $\bar{y}=\bar{b}^{p-1}\bar{x}^{p-1}$ and consider the splitting sequence 
\begin{align*}
    \delta_{\diag(x,y)}&\mapsto \alpha_1\delta_{\diag(x,\bar{y})}+(1-\alpha_1)\delta_{\diag(x,-\bar{y})}\\
    &\mapsto\alpha_1(\alpha_2\delta_{\diag(-\bar{x},\bar{y})}+(1-\alpha_2)\delta_{\diag(\bar{b}\bar{x},\bar{y})})+\\
    &\quad+(1-\alpha_1)(\alpha_3\delta_{\diag(\bar{x},-\bar{y})}+(1-\alpha_3)\delta_{\diag(-\bar{b}\bar{x},-\bar{y})}),
\end{align*}
where $\alpha_1=\frac{y+\bar{y}}{2\bar{y}}$, $\alpha_2=\frac{\bar{b}-\tfrac{x}{\bar{x}}}{\bar{b}+1}$, $\alpha_3=\frac{\bar{b}+\tfrac{x}{\bar{x}}}{\bar{b}+1}$. 
The two terms $\delta_{\diag(-\bar{x},\bar{y})}$ and $\delta_{\diag(\bar{x},-\bar{y})}$ can now be split further using Step 1. Overall we then obtain the probability measure 
\begin{equation*}
   \tilde\nu^\infty=\lambda_1\nu^\infty_{-\bar{x}}+\lambda_2\nu^\infty_{\bar{x}}+\lambda_3\delta_{\diag(-\bar{b}\bar{x},-(\bar{b}\bar{x})^{p-1})}+\lambda_4\delta_{\diag(\bar{b}\bar{x},(\bar{b}\bar{x})^{p-1})},
\end{equation*}
which is of the form \eqref{e:extendedstaircase} with $J'=2$, $J=4$. Since $\alpha_2,\alpha_3\geq \frac{\bar{b}-1}{\bar{b}+1}$, the estimate \eqref{e:staircase4-extended} again holds with a $\bar{b}$-dependent constant $M$.
}
\smallskip


\gray{LS:\, commented out; here gray for comparison:

\medskip

\medskip

\paragraph{Step 2: Construction of a staircase laminate  with good integrability properties}
Recall that we assume $p \in (1,2)$ and that  $\mathcal D$ denotes the set of diagonal matrices.
The key insight in \cite{ColomboTione2022} is the  construction of a staircase 
laminate supported 
in $K_p \cap \mathcal D$
with good integrability properties and  a non-trivial barycenter $A$.
Since $\Id, -\Id \in SO(2)$ 
we have
\begin{equation} K_p \cap \mathcal D = 
\left\{ \begin{pmatrix} x & 0\\ 0 & |x|^{p-2} x \end{pmatrix}
: x \in \R \right\} 
\end{equation}
In the following we identify the diagonal matrix with diagonal entries
$X_{11} =x$ and $X_{22} = y$ with the point $(x,y)$ in $\R^2$. Then rank-one lines in $\mathcal D$ correspond to horizontal or vertical lines in $\R^2$. 

Let $b > 1$.  We set $ A = A_0 = (1,-1)$ and  $A_n =(x_n, -y_n)$ with $$x_n = b n, \quad  y_n = n^{p-1}, \quad z_n = n
\quad \text{for $n\ge 1$.}$$
We consider the splittings
\begin{equation}
    \delta_{(1,-1)} \mapsto  \frac{b-1}{b+1}\delta_{(-1,-1)} + \frac2{b+1}\delta_{(x_1, -y_1)}
\end{equation}
and, for $n \ge 1$,
\begin{equation}\label{e:splitting_plaplace2}
\begin{split}
\delta_{ (x_n, -y_n) }  
&\mapsto \alpha_{n+1} \delta_{(x_n, x_n^{p-1})} +\alpha'_{n+1}\delta_{(x_n ,-y_{n+1})}\\
&\mapsto 
\alpha_{n+1}\delta_{(x_n, x_n^{p-1})}
+ \beta_{n+1} \delta_{(-z_{n+1}, -y_{n+1})}
+\gamma_{n+1}\delta_{ (x_{n+1} ,-y_{n+1})}
\end{split}
\end{equation}
with 
$$ \alpha'_{n+1} = \frac{x_n^{p-1} + y_n}{x_n^{p-1}+y_{n+1}}=
\frac{(b^{p-1}+1) n^p}{b^{p-1} n^p + (n+1)^{p-1}    }
= \frac{(b^{p-1}+1)}{b^{p-1}  + (1 + n^{-1})^{p-1}}.
$$
$$
\gamma_{n+1} =  \alpha_1' \frac{z_{n+1}+x_n}{z_{n+1}+x_{n+1}} =
\frac{(b^{p-1}+1)}{b^{p-1}  + (1 + n^{-1})^{p-1}}  \frac{b+1+n^{-1}}{b + 1 + (b+1) n^{-1}}
$$
and $\alpha_{n+1} = 1 - \alpha'_{n+1}$, 
$\beta_{n+1} = 1-\alpha_{n+1}-\gamma_{n+1}$.

The main observation is that, for a suitable $b >1$, we have
\begin{equation}  \label{eq:bound_wN}
    w_N := \prod_{j=1}^N \gamma_j \sim N^{- \bar q}
    \quad \text{with $\bar q \in (1,p)$.}
\end{equation} 
To see this, note that
\begin{equation}
\label{eq:define_barq_old} - \ln \gamma_{n+1} = \bar q  n^{-1} + \mathcal O(n^{-2}) \quad 
\text{with $\bar q =  \frac{p-1}{b^{p-1} + 1} 
+ \frac{b}{b+1}$} 
\end{equation}
Thus 
$ \ln w_N -  \bar q \ln N$ is uniformly bounded from above and below and hence 
\begin{equation}  \label{eq:bound_wN2}
   \frac1C N^{-\bar q} \le  w_N \le C N^{-\bar q}
\end{equation}for some constant $C$. 
Clearly $\bar q < p$. Moreover, with $a = b^{-1}$ we have
$$ \bar q - 1 =  \frac{p-1}{b^{p-1} + 1} 
- \frac{1}{b+1} = \frac{(p-1) a^{p-1}}{1+ a^{p-1}} - \frac{a}{1+a}.
$$
Since $p-1 \in (0,1)$ we have $a^{p-1} \gg a$ for $0 < a \ll 1$
and we conclude that $\bar q > 1$ for sufficiently 
small $a > 0$ or, equivalently, for sufficiently large $b >1$.
This concludes the proof of \eqref{eq:bound_wN}
where $\sim$ is understood in the sense of 
 \eqref{eq:bound_wN2}.

Hence we are in the situation of 
Proposition  \ref{pr:laminates_to_reduction} with $K = \{(1,-1)\}$,  $K' = K_p \cap \mathcal D$, 
$(1-\gamma_1)\mu_1  = \frac{b-1}{b+1} \delta_{(-1,-1)}$ and 
$$ (1-\gamma_n) \mu_n  = \alpha_n\delta_{(x_{n-1}, x_{n-1}^{p-1})}
+ \beta_n \delta_{(-z_{n}, -y_n)}  \quad \text{for $n \ge 2$.}
$$
As in Proposition  \ref{pr:laminates_to_reduction} denote by
$\nu^N$ the laminate of finite order obtained by iteratively spliting
$\delta_{A_{n-1}}$ for $1 \le n \le N$. Thus
$$ \nu^N = \sum_{n=1}^N w_{n-1} (1-\gamma_n) \mu_n + w_N \delta_{(x_N, -y_N)} $$
with $w_0 = 1$. Now $\mu_n$ is supported in the annulus
$ \{ X : c n \le |X| \le C n\}$ for suitable $C > c > 0$.
Thus it follows from  \eqref{eq:bound_wN2} and \eqref{eq:define_barq}
that, for some $C > 0$ and all $N \in \N$,
\begin{equation}  \label{eq:plaplace_staircase_weak}
    \nu^N( \{ X :|X| >t\}) \le C t^{-\bar q}
\end{equation}
and
\begin{equation} \label{eq:plaplace_staircase_infty}
    \int  |X|^{\bar q} \, d\nu^\infty(X) = \infty.
\end{equation}

Our main observation is that 
these estimates for the staircase 
laminate with barycenter $(1,-1)$  
easily allow one to construct a  non-trivial solution 
$u$ of the differential inclusion with $u \in W^{1,q}$
for all $q < \bar q$, but $u \notin W^{1,\bar q}$.

\medskip

\paragraph{Step 3: approximate solutions}


\begin{proposition}   \label{pr:plaplace_approximate_solutions}
There exists a constant $M$ with the following property.
For each regular set $\Omega \subset \R^2$, $A \in \R^{2,2}$, 
$b \in \R^2$, $s \in (1, \infty)$ and $\eps > 0$
there exists a $u \in W^{1,1}(\Omega)$ such that,
with $\Omega_{error} = \{ x \in \Omega : \nabla u(x) \notin K_p \}$
we have
\begin{equation}  \label{eq:plaplace_error_bound}
\int_{\Omega_{error}} (1+ |\nabla u|)^s < \eps,  
\end{equation}
\begin{equation}  \label{eq:weakLbarq_approximate}
   |\{ x \in \Omega : |\nabla u| > t\}| 
   \le M (1+ |A|^{\frac{\bar q}{p-1}}) |\Omega| t^{-\bar q}
   \text{ \, for all $t > 0$,}
\end{equation}
and
\begin{equation}  \label{eq:approximate_u1_notw1barq}
    \int_{\Omega\setminus\Omega_{error}} |\nabla u|^{\bar q} \, dx = \infty
\end{equation}
\end{proposition}

\begin{proof}
 {\bf Claim 1}: The result holds for $A = (1,-1) \in \mathcal D$.\\ 
This follows from Proposition 
\eqref{pr:laminates_to_reduction} with $K = \{(1,-1)\}$, 
$K' = K_p \cap \mathcal D$,
\eqref{eq:plaplace_staircase_weak},
\eqref{eq:approximate_u1_notw1barq}, and 
\eqref{eq:bound_by_staircase_laminate3}.

{\bf Claim  2}: The result holds for $ A \in O(2) \setminus SO(2)$.\\
In this case, $A$ can be written as
$A = (1,-1) R$ with $R \in SO(2)$. 
Let $\bar u$ be the solution for $A = (-1,1)$ on the set $R^{-1} \Omega$ and define $u := \bar u \circ R$. Since $K_p$   and the norm are invariant under the right action of $SO(2)$, $u$ has the desired properties. 

{\bf Claim 3}: The result holds if $|A| < 1$.\\
It is well-known (see Lemma  \ref{le:laminate_O2} below) 
that  there exists a laminate
of finite order $\nu$ supported on $O(2)$ with barycentre $A$ and such that $\nu(O(2) \setminus SO(2)) > 0$. 

First apply basic construction to this laminate.
Do nothing on $\Omega_{error}$. 
Extend trivially if $A_i \in SO(2)$ 
and use Claim 2 if $A_i \in O(2) \setminus SO(2)$.

{\bf Claim 4}: The result holds if $|A| \ge 1$. \\
Let $\lambda = 2 |A|^{\frac{1}{p-1}}$ and 
$$\bar A =  \begin{pmatrix} \lambda^{-1} & 0 \\ 0 & \lambda^{1-p}
 \end{pmatrix}  A.$$
 Then $|A| < 1$. 
Apply Claim 3 to $\bar A$ 
(with $\eps$ replaced by  $\eps' = \lambda^{-s} \eps$)
Let $\bar u$ be the corresponding solution.
Then 
\begin{equation}  \label{eq:weakLbarq_approximate2}
   |\{ x \in \Omega : |\nabla \bar u| > t\}| 
   \le M'  t^{-\bar q} |\Omega|
   \text{ \, for all $t > 0$.}
\end{equation}
Define $ u_1 = \lambda  \bar u_1$, $u_2 = \lambda^{p-1} \bar u_2$. 
Since $p \in (1,2)$ we have $|\nabla u| \le \lambda |\bar u|$
and hence $1 + |\nabla u| \le \lambda (1 + | \nabla \bar u)$.
Thus
\begin{align*}
   & \,  |\{ x \in \Omega : |\nabla  u| > t\}| 
   \le   |\{ x \in \Omega :
   |\nabla  \bar u| > \frac{t}{\lambda}\}|   \\
\le & \,    M'  \lambda^{\bar q}  t^{-\bar q}  |\Omega| 
\le    2^{\bar q} M' |A|^{\frac{\bar q}{p-1}} 
t^{-\bar q}  |\Omega|
\end{align*}
and 
\begin{align*}
  \,    \int_{\Omega_{error}} (1+ |\nabla u|)^s \, dx 
    \le \lambda^s  \int_{\Omega_{error}}
    (1+ |\nabla \bar u|)^s \, dx 
< \, \lambda^s \eps' \le \eps.
\end{align*}
\end{proof}

In the proof we use the following well known fact.
\begin{lemma}  \label{le:laminate_O2}
Assume that  $A \in \R^{2 \times 2}$ and $|A|^2 := \sum_{i,j} A_{ij}^2 < 1$. Then there exists a laminate of finite order  $\nu$  such that $\nu$ has barycenter $A$, $\nu$ is supported on $O(2)$ and
$\nu(O(2)\setminus SO(2)) > 0$.
\end{lemma}

\begin{proof} Since left multiplication by $SO(2)$ preserves $O(2)$,
rank-one connections,  and $|X|$, we may assume that $A$ is symmetric. Hence there exist $R \in SO(2)$ such that $A = R^{-1} D  R$ und $D \in \mathcal D$.
Since $|X| \le 1$ we have $D = (d_1, d_2)$ with $|d_i| < 1$. 
Hence a laminate of finite order with barycenter $A$ is 
obtained from the three  splitings
$$ \delta_{R^{-1} (d_1, d_2) R} \mapsto  \frac{1-d_1}{2} \delta_{R^{-1} (-1, d_2) R}  +  \frac{1+d_1}{2} \delta_{R^{-1} (1, d_2) R}, $$
$$ \delta_{R^{-1}(\pm 1, d_2) R} 
\mapsto  \frac{1-d_2}{2} \delta_{R^{-1} (\pm 1, 1) R}  +  \frac{1+d_2}{2} \delta_{R^{-1}(\pm 1, 1) R}. $$
Clearly the laminate obtained by splitting is 
supported on $O(2)$, but not on $SO(2)$.
\end{proof}

\medskip

\paragraph{Step 4: solutions of the differential inclusion}

\paragraph{Step 5:} Solution of the differential conclusion that are not in $W^{1,\bar q}(U)$ for any open $U$.

\begin{corollary}
   also can get, for every open disc $B'$,  
 \begin{equation}  \label{eq:exact_nowhere_W1barq_cor}
     \int_{B'} |\nabla u|^{\bar q} \, dx = \infty
 \end{equation}
 but then of course  $u$ no longer piecewise affine
\end{corollary}

In view of  \eqref{eq:nablav_equiv_nablau} this implies 
Theorem \ref{th:plaplace_very_weak}.

\begin{proof}
Now  need to do one  application of the theorem first. Then work inductively as before. Thus $u_0$ is not affine by given by the first application of the theorem.

We might want to put the proof in an appendix not included in the journal version. The argument is more or less standard and we want to keep our alternative proof of the main point in CT22 short.

To show assertion (2),  we first assume that $A \in K_p$. 
Consider the family of dyadic cubes $\mathcal F := \{ 2^{-k} (j + (0,1)^n) : k \in \N, j \in \Z^m\}$
and let 
$\mathcal F' := 
\{ Q \in \mathcal F : 
Q \cap \Omega \neq \emptyset\}$ denote those cubes which intersect $\Omega$. Since $\mathcal F'$ is countable, 
 there exists a
 bijection $k \mapsto Q_k$ from $\N$ to $\mathcal F'$.
Using assertion (1) and modifications on small subsets of $Q_k$ we will inductively construct a sequence of piecewise affine maps 
$u^{(k)}$ such that 
$\int_{Q_k} |\nabla u^{(k+1)}_1|^p \, dx = \infty$ 
and we will show that the limit of the $u^{(k)}$ has the desired properties. 

More precisely, we fix $\alpha \in (0,1)$ and  construct 
sequences of 
families $\mathcal U_k$ of disjoint regular subsets
of $\Omega$ and of  maps  $u_k: \Omega \to \R^2$ such that  
$|\Omega \setminus \bigcup_{U \in \mathcal U_k} U| = 0$,  
$u^{(k)}$ is affine on each $U \in \mathcal U_k$,
$u^{(k)} \in W^{1,\frac{q}{p-1}}(\Omega) \cap C^\alpha(\overline \Omega)$ and $u^{(k)} = l_A$ on $\partial \Omega$. 

Let  $u_0(x) = Ax $ and $\mathcal U_0 = \{ \Omega\}$. Given $u_k$
and $\mathcal U_k$ there exists $U' \in \mathcal U_k$  such that 
$U \cap Q_k \neq \emptyset$. Since $U'$ is open,  
there exists an open non empty  cube $Q'_k$ (not necessarily dyadic) 
such that $\overline{Q'_k} \subset U' \cap Q_k$ and 
\begin{equation}  \label{eq:bound_Qprime}
    |Q'_k| \le 2^{-k-1} \frac{1}{M'+1} |U'|
\end{equation}
Let $A'$ denote the value of $\nabla u^{(k)}$ in $U'$.
By assertion (1), there exist a piecewise affine map
$v: Q'_k \to \R^2$ such that $v = u_k$ on $\partial Q'_k$,
$\nabla v \in K_p$ a.e.\ in $Q'_k$  and
\begin{subequations}
\begin{align}
\label{eq:hoelder_assertion2}
\| v - u_k\|_{C^\alpha(\overline{Q'_k})} \le  & \, 2^{-k}\\
 \label{eq:v_exact_hq}
    \int_{Q'_k} 1+  h_q(\nabla v)  \, dx \le  & \, (M'+1) (1+ h_q(A')) |Q'_k|, \\
 \label{eq:v_exact_no_w1p}
    \int_{Q'_k} |\nabla v_1|^p \, dx = & \,  \infty.
\end{align}
\end{subequations}
Let $V_i$, $i \in \N$ denote the regular domains in $Q'$ on which $v$ is affine. More precisely the $V_i$ are the components of $\mathring Q'_v$. 
Set 
\begin{equation} \mathcal U_{k+1} := (\mathcal U_k \setminus \{U'\}) \cup \bigcup_{i \in \N} \{V_i\} \cup \{ U' \setminus Q'_k\} 
\end{equation}
and 
\begin{equation}  \label{eq:Qprime_replacement}
    u_{k+1} =
    \begin{cases}
    v & \text{in $Q'$,} \\
    u_k & \text{in $\Omega \setminus Q'$.}
    \end{cases}
\end{equation}
Then $u_{k+1}$ is affine on each element of $\mathcal U_{k+1}$ and  $\mathcal U_{k+1}$ is a refinement of $U_k$, i.e., for for each $V
\in U_{k+1}$ there exists a $\tilde V \in \mathcal U_k$ with $V \subset \tilde V$. 
Moreover, 
\begin{equation}   \label{eq:pnorm_Qprime}
\int_{Q'_k} |\nabla u^{(k+1)}_1|^p \, dx = \infty \quad \text{and} \quad
| Q'_k \setminus \bigcup_{ U \in \mathcal U_{k+1} : U \subset Q'_k} U| = 0.
\end{equation}
Since $\nabla u^{(k)} = A'$ in $U'$ it follows from 
\eqref{eq:v_exact_hq} and \eqref{eq:bound_Qprime} that 
\begin{equation}
    \int_{U'} 1+  h_q(\nabla u^{(k+1)}) \, dx \le (1 + 2^{-k}) 
     \int_{U'} 1+ h_q(\nabla u^{(k)}) \, dx
\end{equation}
By induction we get
\begin{equation}  \label{eq:hq_uk_assertion2}
    \int_{U'}1+ h_q(\nabla u^{(k)}) \, dx \le  \prod_{j=0}^{k-1} ( 1+ 2^{-k-1})  (1+h_q(A)) |\Omega|.
\end{equation}
In particular $ \sup_k \|  \nabla u^{(k)}\|_{L^{\frac{q}{p-1}}} < \infty$.

It follows from \eqref{eq:Qprime_replacement} and  
 \eqref{eq:bound_Qprime} that 
 \begin{equation}  \label{eq:small_volume_replacement}
     \{ x \in U : \nabla u_{k+1} \ne \nabla u_k\} \le 2^{-k-1} |U|
     \quad \text{for all $U \in \mathcal U_k$.}
 \end{equation}
 
 Thus $\nabla u^{(k)}$ converges in measure and, 
 in view of the uniform bound 
 in $L^{\frac{q}{p-1}}$, also in $L^1$.  
 It follows from \eqref{eq:hoelder_assertion2} that  
$u_k \to u$ in $C^\alpha(\overline \Omega)$.
Thus $u \in W^{1,1} \cap C^\alpha$ and $\nabla u^{(k)} \to \nabla u$
in measure.

Inductive application of \eqref{eq:small_volume_replacement}
yields
 \begin{equation}  \label{eq:small_volume_replacement2}
     \{ x \in U : \nabla u = \nabla u_k\} \ge \prod_{j=k}^\infty
     (1- 2^{-j-1})  |U|\ge \gamma |U|
     \quad \text{for all $U \in \mathcal U_k$.}
 \end{equation}
 where $\gamma = \prod_{j=0}^\infty (1- 2^{-j-1}) > 0$.
We now apply  \eqref{eq:small_volume_replacement2}  to each of the sets $U \in \mathcal U_{k+1}$ with $U \subset Q'_k$ and use 
 \eqref{eq:pnorm_Qprime}  and the fact that $\nabla u^{(k+1)}$ 
 is constant on $Q'$. This yields
 \begin{align}  \label{eq:assertion2_now1p}
  &  
   \int_{Q'_k} |\nabla u_1|^p \, dx
  = \,  \sum_{U \in \mathcal U_{k+1}: U \subset Q'_k} \int_{U} |\nabla u_1|^p \, dx \\
  \ge  & \, \gamma \sum_{U \in \mathcal U_{k+1}: U \subset Q'_k} \int_{U} |\nabla u^{(k+1)}_1|^p \, dx = 
  \int_{Q'_k} |\nabla u^{(k+1)}_1|^p \, dx =  \infty.  \nonumber 
 \end{align}
 Since every open disc $B' \subset \Omega$ contains one of the cubes $Q_k$  and hence one of the cubes $Q_k'$
 we obtain   \eqref{eq:exact_nowhere_W1barq_cor}.

Since $u^{(k)} = l_A$ on $\partial \Omega$ and $u^{(k)}$ converges in $C^\alpha(\overline \Omega$,  we get  $u=l_A$ on $\partial \Omega$.

Let $\Omega_k = \{ x \in \Omega : \nabla u(x) = \nabla u^{(k)}(x)$.
It follows from \eqref{eq:small_volume_replacement} that 
$|\Omega \setminus \Omega_k| \le 2^{-k}  |\Omega|$.
Since $\nabla u_k \in K_p$ a.e.\, it follows that $\nabla u \in K_p$ a.e.

Finally we claim that 
\begin{equation}  \label{eq:hq_bound_assertion2_AKp}
\int_{\Omega} h_q(\nabla u) \, dx \le 2 h_q(A) |\Omega|.
\end{equation}   
To see this, 
set $E_k = \bigcup_{j \ge k} (\Omega \setminus \Omega_k)$.
Then $|E_k| \le 2^{-k+1} |\Omega$ and $E_{k+1} \subset E_k$. 
Since $\nabla u = \nabla u_k$ on $\Omega \setminus E_k$
\eqref{eq:hq_bound_assertion2_AKp} follows from 
\eqref{eq:hq_uk_assertion2} by applying the monotone convergence theorem to $f_k(x) = 1_{\Omega \setminus E_k} h_q(\nabla u^{(k)}(x))$.

\medskip

Finally,  we prove assertion (2) for a general $A \in \R^{d \times m}$. Let $\tilde u$ be as in assertion (1). Then $\tilde u$ is 
piecewise affine and $\nabla \tilde u \in K_p$ in each connected component $\Omega_i$  of
$\mathring \Omega_{\tilde u}$. Thus we can apply the previous
argument in each $\Omega_i$. Using   \eqref{eq:assertion2_now1p} \eqref{eq:hq_bound_assertion2_AKp} for $\Omega_i$ instead of $\Omega$ we easily conclude.
\end{proof}

}

\gray{LS:\, commented out; here gray for comparison:

\medskip

\section{Proof of Proposition~\ref{pr:diamond} for $\max(p,q) < \infty$} \label{se:diamond}

Here we give a detailed proof of Proposition~\ref{pr:diamond}. We restate the relevant part of  result for the convenience of the reader.
Note that the cases $p=\infty$ or $q= \infty$ were already discussed in the proof of Proposition~\ref{pr:diamond} in 
Section~\ref{se:general}

\begin{proposition}[Propagation of weak $L^p$-estimates under splitting]  \label{pr:diamond2}
Let $V$ be a finite-dimensional normed space. We use the abbreviation $\{ |X| > t\} := \{ X  \in V : |X| > t\}$. For a probability measure
$\nu$ on $V$ and $1 \le p < \infty$ we define the weak-$L^p$ moment by 
\begin{equation}
|\nu|_p := \left(  \sup_{t > 0} \,  t^p \nu( \{ |X| > t\}   \right)^{1/p} 
\end{equation}

Let $\nu' = \sum_{i \in \N} \lambda_i \delta_{A_i}$ be a probability measure on $V$ and let $\nu'' = \{ \nu''_i\}_{i \in \N}$
be a family of probability measures on $V$. We define the iterated probability measure $\nu' \diamond \nu''$ by
\begin{equation}
\nu' \diamond \nu'' = \sum_i \lambda_i \nu_i.
\end{equation}

Suppose that there exist $p, q \in [1, \infty)$  and $M'' \ge 1$ such that 
\begin{equation}
|\nu'|_p < \infty, 
 \quad |\nu''_i|_q  \le  M'' (1 +|A_i|)
\end{equation}
Then 
\begin{eqnarray}
|\nu' \diamond \nu''|_p &\le&     \left(  1+ \frac{q}{q-p} \right)^{\frac1p} 2 \max(|\nu'|_p, 1)     M''   \quad \text{if $p < q$},
\label{eq:iteration1app} \\
|\nu' \diamond \nu''|_q &\le&   \left( 1+  \frac{q}{p-q} \right)^{\frac1q} ( |\nu'|_p + 1)  M''   \quad \text{if $p >  q$}. 
\label{eq:iteration2app}
\end{eqnarray}
\end{proposition}

\begin{proof}
Assume first that $p < q$. Assume in addition that $q < \infty$ and  $t > 2 M''$. Then $\frac{t}{M''} -1 \ge \frac{t}{2 M''} > 1$ and

\begin{eqnarray} 
& & \mu( \{ |X| > t \})   \label{eq:iteration_weak_estimate1}\\
 &\le & \sum_{i : t \le  M'' (1 + |A_i|)} \lambda_i  + \sum_{i : t >  M'' (1 + |A_i|)} \lambda_i \nu''_i(\{ |X| > t\}),   \nonumber \\
&\le & \nu'( \{ |X| \ge  \frac{t}{M''} - 1 \}) +  \sum_{i : t >  M'' (1 + |A_i|)} \lambda_i  (M'')^q \frac{ (1+|A_i|)^q}{t^q},  \nonumber \\
& \le & \nu'( \{ |X| \ge  \frac{t}{2M''}  \})   +  \frac{(M'')^q}{t^q}  \sum_{i : t > M'' (1 + |A_i|)} \lambda_i   \left[ (1+|A_i|)^q -  1 \right] 
 \nonumber \\
    & &  + \frac{(M'')^q}{t^q}  \sum_{i : t >  M'' (1 + |A_i|)} \lambda_i  \nonumber  \\
&\le &  |\nu'|_p^p \left( \frac{2 M''}{t} \right)^p  +   \frac{(M'')^q}{t^q}    \nonumber \\
& &  +    \frac{(M'')^q}{t^q}  \sum_{i : t >  M'' (1 + |A_i|)} \lambda_i   \left[ (1+|A_i|)^q -  1 \right]   \nonumber 
\end{eqnarray}
Now
$$ \sum_{i : t >  M'' (1 + |A_i|)} \lambda_i   \left[ (1+|A_i|)^q -  1 \right]  = \int_{ \{ |X| < R \}}  \left[(1+ |X|)^q - 1 \right] \, d\nu'(X) $$
with $R = \frac{t}{M''} - 1$.
Moreover by Fubini's theorem we have 
$$ \int_{ \{ |X| < R \}} h(|X|) \, d\nu'(X)  = \int_0^R       h'(s) \, \nu'( \{s < |X| < R \})\, ds. $$
for every $C^1$ function $h$ with $h(0) = 0$. 
Hence
\begin{eqnarray*}
& &  \sum_{i : t >  M'' (1 + |A_i|)} \lambda_i   \left[ (1+|A_i|)^q -  1 \right] |\\
& \le & \int_0^{\frac{t}{M''}-1} q (1 + s)^{q-1}  \, \min\left( \frac{|\nu^{(1)}|^p_p}{s^p} , 1\right)  \, ds\\
& \le & \int_0^{\frac{t}{M''}-1} q (1 + s)^{q-1}  \, 2^p \frac{ \max(|\nu^{(1)}|^p_p, 1)}{(1+s)^p}  \, ds\\
&=& 2^p  \max(|\nu^{(1)}|^p_p, 1)     \frac{q}{q-p}\left( \frac{t^{q-p}}{(M'')^{q-p}} - 1  \right).
\end{eqnarray*}
Here the second inequality follows by considering the case $0 \le s \le 1$ and $s >1$ separately. 
Inserting this estimate into \eqref{eq:iteration_weak_estimate1}
we get
\begin{equation}
 \mu( \{ |X| > t \}) \le 2^p (M'')^p |\nu|_p^p t^{-p} +  2^p  \max(|\nu'|^p_p, 1)     \frac{q}{q-p}  (M'')^p t^{-p}
\end{equation}
whenever $t > 2 M''$. In combination  with the trivial estimate
\begin{equation} t^p  \mu( \{ |X| > t \}) \le 2^p (M'')^p \quad \text{if $0 < t \le 2 M''$}
\end{equation}
we get
$$ |\mu|_p^p \le   2^{p} (M'')^p  \left(1 + \frac{q}{q-p}  \right) \max(|\nu'|^p_p, 1)$$  
if $p < q < \infty$.

\smallskip

Now assume that $q < p$.  Then
\begin{eqnarray}
& & \mu( \{|X| > t\}) \\
 &\le &   \sum_i  \lambda_i  t^{-q}  (M'')^q  (1+|A_i|)^q  \nonumber  \\
 & = & t^{-q} (M'')^q  \int (1+|X|)^q \, d\nu'(X).
\end{eqnarray}
%
%
Using that $\frac{s+1}{s} \le \frac{R+1}{R}$ if $s \ge R > 0$ we get
\begin{eqnarray*}
&  &  \int  (1 + |X|)^q  \, d\nu'(X)  \nonumber  \\
&\le &  ( |\nu'|_p + 1)^q  + \int _{ \{|X| > |\nu'|_p \}}  \left[(1 + |X|)^q - (|\nu'|_p + 1)^q \right] \, d\nu'(X)    \nonumber  \\
& \le &     ( |\nu'|_p + 1)^q  + \int_{ |\nu'|_p}^\infty  q (1+s)^{q-1}   \nu'( \{ |X| > s \} \, ds \\
& \le &  ( |\nu'|_p + 1)^q    + \int_{ |\nu'|_p}^\infty  q (1+s)^{q-1}  \frac{|\nu'|_p^p}{s^p}  \, ds \\
& \le &  ( |\nu'|_p + 1)^q   + \left(\frac{( |\nu'|_p + 1}{ |\nu'|_p} \right)^p  
 \int_{ |\nu'|_p}^\infty  q (1+s)^{q-1-p}      |\nu'|_p^p \, ds \\
 & \le &   ( |\nu'|_p + 1)^q   + \left( |\nu'|_p + 1 \right)^p  \frac{q}{p-q}  \frac{1}{ (|\nu'|_p + 1)^{p-q}} \\
 & \le &  ( |\nu'|_p + 1)^q  \left( 1 + \frac{q}{p-q} \right)
\end{eqnarray*}
This implies  \eqref{eq:iteration2}  for $q < p < \infty$. 
\end{proof}
}

\gray{LS:\, commented out; here gray for comparison:
\medskip

\section{$L^\infty$ estimates for approximate solutions}  \label{se:app_linfty}

The goal of this section is to prove the following result. Recall that $\rightharpoonup$ denotes weak convergence. 

\begin{theorem}  \label{th:approximate_appendix}   There exists a constant $M$ with the following property.
Let  $ A \in \R^{2n \times 2n} \setminus L$ with  $\rank A = 1$.
Let $\Omega \subset \R^{2n}$ be bounded and open,
 and  $\alpha \in [0, 1)$. Then there exists a  sequence of  maps
 $u^{(j)} : \overline \Omega \to \R^{2n}$ such that
$u^{(j)} \in  W^{1,1}(\Omega) \cap C^\alpha(\overline \Omega)$, $u^{(j)}=l_A$ on $\partial\Omega$  
and
\begin{eqnarray}
\left|\{|\nabla u^{(j)}| > t \} \right| &\le& M (1 + |A|)^{2n}  |\Omega| t^{-2n} \quad  \text{for $t > 0$,} \label{eq:approximate_weak_L2n} \\
\lim_{j \to \infty} \sup_{t > 1}  t^{2n}  & &  \!  \! \! \!    \!  \! \! \!  \! \! \!  \!  \! \! \!\left| \{  \nabla u^{(j)} \notin L \cap \Sigma\} \cap \{|\nabla u^{(j)}| > t \} \right|  = 0 \label{eq:approximate_bad_set_L2n}\\
\det Du^{(j)} &\to& 1   \quad \text{in $L^\infty(\Omega)$,}    \label{eq:approx_conv_det}     \\
u^{(j)} &\to&  l_A  \quad \text{in $C^\alpha(\overline \Omega)$,}\\
u^{(j)} &\rightharpoonup& l_A \quad \text{in $W^{1,p}(\Omega)$ for all $p \in [1, 2n)$}
\end{eqnarray}
and
\begin{equation}  \label{eq:approximate_both_Li}
\liminf_{j \to \infty} \|  \dist(Du^{(j)},L_i) \|_{L^1(\Omega)} > 0  \quad \text{for $i=1, 2$.}
\end{equation}
Moreover, the assertion  \eqref{eq:approx_conv_det} can be replaced by 
\begin{equation}
\dist(\nabla u^{(j)}, L \cap \Sigma) \to 0 \quad \text{in $L^\infty(\Omega)$.}
\end{equation}
\end{theorem}

As in the proof of Theorem~\ref{th:approximate} it is easy to see that 
Theorem~\ref{th:approximate_appendix} easily follow from Proposition~\ref{p:stage2}
and the following refined version of Proposition~\ref{p:stage3}.

\begin{proposition}[Stage 3, refined]\label{p:stage3_refined} There exists a constant $M$ with the following property.
Let $\Omega\subset\R^{2n}$,  $\alpha\in (0,1)$, and $\varepsilon>0$.  
 Then, for  any affine map $l_{A,b}$ with $A \in L_1$  there exists a piecewise affine map $u\in W^{1,1}(\Omega)\cap C^{\alpha}(\overline{\Omega})$
  with $u=l_{A,b}$ on $\partial\Omega$ such that, 
with 
$$
\Omega_{error}:=\{x\in\Omega:\,\nabla u(x)\notin L_1\cap\Sigma\}
$$ 
we have
\begin{subequations}
\begin{align}
|\det \nabla u - 1| &\le \eps  \quad \text{a.e. in $\Omega$,} \label{e:stage3-0_refined}\\
| \Omega_{error} \cap \{x\in\Omega:\,|\nabla u(x)|>t\}|& \le \eps  (1 + |A|^{2n}) |\Omega| t^{-2n}  \label{e:stage3-1_refined}\\
|\{x\in\Omega:\,|\nabla u(x)|>t\}|&\leq M (1+|A|^{2n})|\Omega|t^{-{2n}}.\label{e:stage3-2_refined}
\end{align}
\end{subequations} 
Moreover, condition \eqref{e:stage3-0_refined} can be replaced by
\begin{equation} \label{e:stage3-0bis_refined}
\dist(\nabla u, L_1 \cap \Sigma) \le \eps  \quad \text{a.e. in $\Omega$.}
\end{equation}
\end{proposition}

\purple{ Check how exactlyProposition~\ref{p:stage2} and Proposition~\ref{p:stage3_refined} need to be combined.
Perhaps it is easier  to directly state a 'refined' version of the combination of Stage 2 and Stage 3}

\subsection{Laminates and differential inclusions (refined)}\label{ss:di_refined}

\begin{lemma}\label{l:roof_refined}
Let $A,A_1,A_2\in\R^{d\times m}$ be matrices such that 
\begin{equation*}
\rank(A_1-A_2)=1,\textrm{ and }A=\lambda_1 A_1+\lambda_2 A_2
\end{equation*} 
for some $\lambda_1,\lambda_2> 0$, $\lambda_1+\lambda_2=1$. For any $b\in\R^d$, any $\delta  \in (0, \frac12 |A_1 - A_2|)$,  $\eps>0$ and any regular domain $\Omega\subset\R^m$ there exists a piecewise affine Lipschitz map $u:\Omega\to\R^d$ such that 
$u=l_{A,b}$ on $\partial\Omega$ and, for $i =1, 2$, 
\begin{equation*} 
 \left|\left\{x\in\Omega:\,\nabla u(x) \in B_\delta(A_i)\right\}\right|  =  \lambda_i|\Omega|.
\end{equation*}
and
\begin{equation*}
(1- \eps) \lambda_i 
|\Omega| \leq \left|\left\{x\in\Omega:\,\nabla u(x)=A_i\right\}\right|\leq \lambda_i|\Omega|.
\end{equation*}
In particular,  $\nabla u \in \cup_{i=1}^2 B_\delta(A_i)$ a.e. in $\Omega$ and
\begin{equation}    \label{eq:roof_linfty_refined} 
 \|\nabla u\|_{L^{\infty}(\Omega)}\leq \max(|A_1|,|A_2|) + \delta
\end{equation}  
Moreover the set  $\{  X : | \{ \nabla u = X \}| > 0 \}$ of $\nabla u$ is finite, i.e., the induced measure $\nu_u$ is a finite
combination of Dirac masses. 
\end{lemma}

\begin{proof} This is a minor variant of the thin diamond construction which appears, e.g., in  XXXXX CITE  MS Annals.
Since $\rank(A_1-A_2)=1$, there exist nonzero vectors $\xi\in\R^m$, $\eta\in\R^d$ such that $A_2-A_1=\eta\otimes\xi$. Note that we can write
$A_1=A-\lambda_2\eta\otimes\xi$ and $A_2=A+\lambda_1\eta\otimes\xi$. We may assume that $|\xi| = 1$.

Set $\xi^{(1)} = \xi$ and let $\xi^{(2)},\dots,\xi^{(m)}\in \R^m$ be such that $(\xi^{(1)},\dots,\xi^{(m)})$ is an orthonormal basis of $\R^m$. 
We will construct  a bounded open set $\Omega_0$ and a Lipschitz function $f: \Omega_0 \to \R$  with $f|_{\partial \Omega_0} = 0$ such that 
$\nabla f$ is close to $-\lambda_2 \xi$ or $\lambda_1 \xi$  a.e. in $\Omega$ and the set $\{x \in \Omega : \nabla f \notin \{ -\lambda \xi, (- \lambda_2 \xi, \lambda_1 \xi)
\}$ has small measure.

 Set 
 $$ h(t) = \lambda_1 \lambda_2  +  \begin{cases}  - \lambda_2 t & \text{if $t  \ge 0$}, \\
 \lambda_1 t & \text{if $t< 0$.}
\end{cases}
$$
and
$$ f_0(x) = h(x \cdot \xi).$$
Then $\nabla f_0  =  -\lambda_2 \xi$ if $x \cdot \xi > 0$ and $\nabla f_0 =  \lambda_1 \xi$ if $x \cdot \xi < 0$. 
Moroever, 
$$ S :=  \{x \in \R^m :  f(x) > 0\} = \{ x :  - \lambda_2    <  x\cdot \xi < \lambda_1 \}.$$

Let $ \rho > 0$ and $R \ge 1$,  to be chosen later,  and set
$$ f(x) := f_0(x) -  \frac{ \rho}{2\sqrt{m-1}}  \max(\sum_{i=2}^m  |x \cdot \xi^{(i)}| - R, 0),$$
$$ \Omega_0 := \{ x \in \R^m : f(x) > 0\}.$$
Then $\nabla f$ takes only finitely many values (off a nullset),  $|\nabla f - \nabla f_0| \le \frac{\rho}{2}$ and $\Omega_0$ is a bounded, open, convex subset of  $S$. 
More precisely
$$ \Omega_0 = \{ x :   \sum_{i=2}^m  |x \cdot \xi^{(i)}|<  \frac1\rho h(x \cdot \xi)  +R \}$$
Setting $\Omega_\pm =\{ x \in  \Omega_0 : \pm (x \cdot \xi) > 0$ we see that
$$ |\Omega_+| = \lambda_1 |\Omega_0|, \quad |\Omega_-| =  \lambda_2 |\Omega_0|$$
and thus
$$ \{ x \in \Omega_0 : \nabla f \in B_{  \rho}(A_1)\} = \lambda_1 \Omega, 
\quad 
  \{ x \in \Omega_0 : \nabla f \in B_{  \rho}(A_2)\} = \lambda_2 \Omega.$$
Set  
$$ \Omega_0^R := \{ x \in \Omega_0 : \sum_{i=2}^m  |x \cdot \xi^{(i)}| <  R \}.$$
Then $f = f_0$ in $\Omega_0^R$. We  claim that
\begin{equation}  \label{eq:roof_volume_fraction}
\lim_{R \to \infty} \frac{|\Omega_0 \setminus \Omega_0^R|}{\Omega_0} = 0.
\end{equation}
Indeed, we have
\begin{eqnarray*} & &  |\Omega_0 \setminus \Omega_0^R| \\
& \le &
|\{ x :  - \lambda_2 < x \cdot \xi < \lambda_1 : R \le \sum_{i=2}^m  |x \cdot \xi^{(i)}| <  R +  \lambda_1 \lambda_2  \frac1\rho \}| \\
& \le& C R^{m-2} \rho^{-1}
\end{eqnarray*}
Thus,  for fixed $\rho$,
\begin{equation}
\lim_{R \to \infty} \frac{|\Omega_0 \setminus \Omega_0^R|}{\Omega_0} = 0.
\end{equation}
Moreover,  $\Omega_0^R = \{ x \in \R^m :  \sum_{i=2}^m  |x \cdot \xi^{(i)}| < R ,  -\lambda_2 < x \cdot \xi  < \lambda_1 \}$ and hence   
$|\Omega_0^R| =  c R^{m-1}$ with $c >0$. Thus  \eqref{eq:roof_volume_fraction} holds.

Finally set $\rho = \frac{\delta}{|\eta|}$. Then the map $u: \Omega_0 \to \R^d$ given by 
$$ u:= l_{A,b} + f$$
has all the required properties  for the special domain $\Omega_0$, provided we choose $R$ large enough. 
For a general regular domain $\Omega$ we apply the  rescaling and covering argument from Section \ref{ss:basic}.
\end{proof}

An obvious iteration of Lemma \ref{l:roof} along the splitting sequence of any laminate $\nu\in \L(\R^{d\times m})$ (c.f.~Section \ref{ss:laminates})
leads to the following lemma, which makes laminates so useful for inclusion problems of the type \eqref{e:diffincl}.

\begin{lemma}\label{l:basicconstruction_refined}
Let $\nu\in\L(\R^{d\times m})$ be a laminate of finite order with center of mass $A$. Write $\nu=\sum_{j=1}^J\lambda_j\delta_{A_j}$ with $\lambda_j>0$
and $A_j \ne A_k$ for $j \ne k$.
For any $b\in\R^d$, any $\delta \in (0, \frac12 \min_{j \ne k} |A_j-A_k|)$,  $\eps>0$ and any regular domain $\Omega\subset\R^m$ there exists a piecewise affine Lipschitz map $u:\Omega\to\R^d$ such that 
$u=l_{A,b}$ on $\partial\Omega$ and, for $i =1, \ldots, J$, 
\begin{equation*} 
 \left|\left\{x\in\Omega:\,\nabla u(x) \in B_\delta(A_i)\right\}\right|  =  \lambda_i|\Omega|.
\end{equation*}
and
\begin{equation} \label{e:basicconstruction_refined}
(1- \frac{\eps}{J}) \lambda_i 
|\Omega| \leq \left|\left\{x\in\Omega:\,\nabla u(x)=A_i\right\}\right|\leq \lambda_i|\Omega|.
\end{equation}
In particular,  $\nabla u \in \cup_{i=1}^J B_\delta(A_i)$ a.e. in $\Omega$ and
\begin{equation}    \label{eq:basic_linfty_refined} 
 \|\nabla u\|_{L^{\infty}(\Omega)}\leq \max(|A_1|,|A_2|) + \delta
\end{equation}  
Moreover the set $\{  X : |\{\nabla u = X \}| > 0 \}$  of $\nabla u$ is finite, i.e., the induced measure $\nu_u$ is a finite
combination of Dirac masses. 
\end{lemma}
We remark that, since $\sum_{j=1}^J\lambda_j=1$, estimate \eqref{e:basicconstruction_refined} also implies
\begin{equation}\label{e:basicconstruction_refined-2}
\left|\{x\in\Omega:\,\nabla u(x)\notin \supp\nu\}\right|\leq \eps|\Omega|.
\end{equation}

\subsection{Stage 3, refined}

\begin{proof}[Idea of proof Proposition~\ref{p:stage3_refined}] 
This is very similar to the proof of Proposition~\ref{p:stage3}.
One difference is that we use Lemma~\ref{l:basicconstruction_refined} instead of
 Lemma~\ref{l:basicconstruction}. This ensures that  $\nabla u_N$ is uniformly close
 to the support of the laminates used in the construction. 

For $A \in \mathcal D_{>1}$
the main difference is in the definition of the  sets $\Omega^{(N)}_{inductive}$.
In the proof of Proposition~\ref{p:stage3} these sets  where given by $ \{ x \in \mathring\Omega_{u_N} : \nabla u_N(x) = 2^N A\}$.
Now we set
$$  \Omega^{(N)}_{inductive} := \{ x \in \mathring\Omega_{u_N} : \nabla u_N(x) \in B_{\delta_N}(2^N A)\}. $$
Here $\delta_N$ is chosen such that in $\mathring\Omega_{u_N} \setminus  \Omega^{(N)}_{inductive}$ the gradient
$\nabla u_N$ is close to the 'good' points in the support of the laminate.
In particular $|\det \nabla u_N - 1| < \eps_N$ in $\mathring\Omega_{u_N} \setminus  \Omega^{(N)}_{inductive}$.

The difficulty is that in the course of the iteration not only diagonal matrices appear, since $B_{\delta_N}(2^N A)$
also contains non-diagonal matrices. For $A' \in B_{\delta_N}(2^N A)$ we can find $R \in SO(2n)$ and $Q \in SO(2n)$
which are both close to the identity such that $A' = R D Q$ with a diagonal matrix $D$. If $\nu_D = \sum_j \lambda_j \delta_{B_j} + \gamma \delta_{2D}$
is the laminate in  Lemma~\ref{l:staircase1}, then $\nu_A' := \sum_j \lambda_j \delta_{R B_j Q} + \gamma \delta_{2A'}$
is also a laminate of finite order 
and we can apply  Lemma~\ref{l:basicconstruction_refined} to construct a corresponding piecewise affine map. 

We have $\det(R B_j Q) = \det B_j = 1$. 
In general, however, $R$ and $Q$ are not in $L_1$. Thus $R B_j Q$ in general will not be in $L_1$.
We still know that $|R B_j Q| \sim 2^N$ and we can use an estimate of the form 
$$  \left|  \{ x \in \mathring\Omega_{u_N} : \nabla u_N(x) \in B_{\delta_N}(2^N) A \setminus \{ 2^N A\} \} \right|  \le \eps' 2^{-2n N}   |\Omega| $$
to get  \eqref{e:stage3-1_refined}

If we would like to show \eqref{e:stage3-0bis_refined} rather than \eqref{e:stage3-0_refined} we can argue similarly, but in this case
we do not use the $SO(2n)$ action but simply translation. 
More specifically, we set $\nu_A' := \sum_j \lambda_j \delta_{B_j + A'-2^N A} + \gamma \delta_{2A'}$.
Since $ A'-2^N A$ is small, we  get \eqref{e:stage3-0bis_refined}.
\end{proof}

}

\bibliographystyle{amsalpha}
\bibliography{product_quotient}

\end{document}

%% file: macros_kmx.tex
\usepackage{amsmath,amssymb,amsthm,graphicx,color,amscd}
\usepackage[usenames,dvipsnames]{xcolor}
\usepackage{enumitem}

\usepackage{hyperref}

\numberwithin{equation}{section}
\hypersetup{bookmarksdepth=3}

\theoremstyle{plain}
\newtheorem{theorem}{Theorem}

\newtheorem{corollary}[theorem]{Corollary}

\newtheorem{proposition}[theorem]{Proposition}
\newtheorem{lemma}[theorem]{Lemma}

\theoremstyle{definition}
\newtheorem{definition}{Definition}
\newtheorem{example}{Example}

\newtheorem{question}[equation]{Question}

\theoremstyle{remark}
\newtheorem{remark}{Remark}

\numberwithin{theorem}{section}
\numberwithin{definition}{section}
\numberwithin{example}{section}
\numberwithin{remark}{section}

\newcommand{\ben}{\begin{enumerate}}
\newcommand{\bit}{\begin{itemize}}

\newcommand{\diag}{\operatorname{diag}}

\newcommand{\dist}{\operatorname{dist}}
\newcommand{\een}{\end{enumerate}}
\newcommand{\eit}{\end{itemize}}

\newcommand{\eps}{\varepsilon}

\newcommand{\Id}{\operatorname{Id}}

\renewcommand{\L}{{\mathcal L}}

\newcommand{\loc}{\operatorname{loc}}

\newcommand{\N}{\mathbb{N}}

\newcommand{\Om}{\Omega}

\newcommand{\R}{\mathbb{R}}
\newcommand{\ra}{\rightarrow}

\newcommand{\rank}{\operatorname{rank}}
\definecolor{gray}{gray}{0.7}

\newcommand{\gray}[1]{{\color{gray}{GRAY#1}}}

\newcommand{\Z}{\mathbb{Z}}

\def\XXint#1#2#3{{\setbox0=\hbox{$#1{#2#3}{\int}$ }
\vcenter{\hbox{$#2#3$ }}\kern-.6\wd0}}
